\documentclass[a4paper,11pt]{amsart}

\usepackage{amsfonts}
\usepackage{amsthm}
\usepackage{amsmath}
\usepackage{amssymb}
\usepackage[mathscr]{euscript}
\usepackage{enumerate}
\usepackage{verbatim} 
\usepackage{mathtools}
\usepackage{url}
\usepackage{tikz-cd}
\usepackage{xypic}
\usepackage{todonotes}
\usepackage[margin=20mm]{geometry}
\usepackage[pagebackref, colorlinks=true, linkcolor=blue, citecolor = red]{hyperref}

\usepackage[nameinlink,capitalise,noabbrev]{cleveref}
\crefname{subsection}{Subsection}{Subsection}
\crefname{equation}{Diagram}{Diagram}

\renewcommand{\eqref}[1]{Equation \ref{#1}}

 \renewcommand*{\backref}[1]{}
 \renewcommand*{\backrefalt}[4]{({%
     \ifcase #1 Not cited.%
           \or On p.~#2%
           \else On pp.~#2%
     \fi%
     })}

\theoremstyle{definition} 
\newtheorem{defn}[equation]{Definition}
\newtheorem{notn}[equation]{Notation}
\newtheorem{prop-def}[equation]{Proposition-Definition}

\theoremstyle{plain} 
\newtheorem{thm}[equation]{Theorem}
\newtheorem{lemma}[equation]{Lemma}
\newtheorem{prop}[equation]{Proposition}
\newtheorem{cor}[equation]{Corollary}
\newtheorem*{thm*}{Theorem}

\theoremstyle{remark} 
\newtheorem{rmk}[equation]{Remark}
\newtheorem{ex}[equation]{Example}

\newtheorem{conj}[equation]{Conjecture}

\numberwithin{equation}{section}

\newtheoremstyle{TheoremNum}
{}{}              
{\itshape}                      
{}                              
{\bfseries}                     
{.}                             
{ }                             
{\thmname{#1}\thmnote{ \bfseries #3}}
\theoremstyle{TheoremNum}
\newtheorem{thmn}{Theorem}
\newtheorem{corn}{Corollary}

\DeclareMathAlphabet{\mathbbe}{U}{bbold}{m}{n}

\def\DDelta{{\mathbbe{\Delta}}}
\newcommand{\DD}{\DDelta}

\newcommand{\A}{\mathscr{A}}
\newcommand{\B}{\mathscr{B}}
\newcommand{\C}{\mathscr{C}}
\newcommand{\D}{\mathscr{D}}
\renewcommand{\cD}{\mathcal{D}}
\newcommand{\F}{\mathscr{F}}

\newcommand{\G}{\mathscr{G}}
\newcommand{\cG}{\mathcal{G}}
\renewcommand{\O}{\mathscr{O}}
\renewcommand{\S}{\mathscr{S}}
\renewcommand{\P}{\mathcal{P}}
\newcommand{\V}{\mathscr{V}}
\newcommand{\bN}{\mathbb{N}}
\newcommand{\bZ}{\mathbb{Z}}
\newcommand{\Map}{\mathrm{Map}}
\newcommand{\Hom}{\mathrm{Hom}}
\newcommand{\Aut}{\mathrm{Aut}}

\newcommand{\Obj}{\mathrm{Obj}}
\newcommand{\Mor}{\mathrm{Mor}}
\newcommand{\Dom}{\mathrm{Dom}}
\newcommand{\Cod}{\mathrm{Cod}}

\newcommand{\Fun}{\mathrm{Fun}}
\newcommand{\Sha}{\mathscr{S}\mathrm{ha}}
\newcommand{\HH}{\mathrm{HH}}
\newcommand{\THH}{\mathrm{THH}}
\newcommand{\HHinfty}{\mathrm{HH}_\infty}
\newcommand{\biN}{\mathrm{biN}}
\newcommand{\biHH}{\mathrm{biHH}}
\newcommand{\biHHinfty}{\mathrm{biHH}_\infty}

\newcommand{\M}{\mathscr{M}}
\newcommand{\N}{\mathscr{N}}

\newcommand{\set}{\mathscr{S}\text{et}}
\newcommand{\cat}{\mathscr{C}\mathrm{at}}
\newcommand{\grph}{\mathscr{G}\mathrm{rph}}
\newcommand{\Grpd}{\mathscr{G}\mathrm{rpd}}

\newcommand{\sSet}{\text{s}\set} 
\newcommand{\Sp}{\mathrm{Sp}}
\newcommand{\Ho}{\mathrm{Ho}}
\newcommand{\Mod}{\mathscr{M}\mathrm{od}}
\newcommand{\Bypass}{\mathscr{B}\mathrm{ypass}}
\newcommand{\Comp}{\mathrm{Comp}}
\newcommand{\ds}{\displaystyle}
\newcommand{\shadow}[1]{\langle\langle #1 \rangle \rangle}
\newcommand{\Cocone}{\mathscr{C}\mathrm{ocone}}
\newcommand{\St}{\mathscr{S}\mathrm{t}}
\newcommand{\Un}{\mathscr{U}\mathrm{n}}
\newcommand{\colim}{\mathrm{colim}}
\newcommand{\coTHH}{\mathrm{co}\THH}
\newcommand{\CSS}{\mathscr{C}\mathscr{S}\mathscr{S}}
\newcommand{\id}{\mathrm{U}}
\newcommand{\rid}{\mathrm{id}}
\newcommand{\tr}{\mathrm{tr}}
\newcommand{\Adj}{\mathscr{A}\mathrm{dj}}
\newcommand{\AdjEq}{\mathscr{A}\mathrm{dj}\mathscr{E}\mathrm{q}}
\newcommand{\End}{\mathscr{E}\mathrm{nd}}
\newcommand{\AdjEnd}{\mathscr{A}\mathrm{dj}\mathscr{E}\mathrm{nd}}
\newcommand{\AdjEqEnd}{\mathscr{A}\mathrm{dj}\mathscr{E}\mathrm{q}\mathscr{E}\mathrm{nd}}
\newcommand{\Mon}{\mathscr{M}\mathrm{on}}
\newcommand{\CoMon}{\mathscr{C}\mathrm{o}\mathscr{M}\mathrm{on}}
\newcommand{\Arr}{\mathscr{A}\mathrm{rr}}
\newcommand{\Iso}{\mathscr{I}\mathrm{so}}
\newcommand{\Fin}{\mathscr{F}\mathrm{in}}
\newcommand{\rcong}{\rotatebox[origin=c]{45}{$\cong$}}

\title{Shadows are Bicategorical Traces}
\date{June 2025}

\author{Kathryn Hess}
\address{{\'E}cole Polytechnique F{\'e}d{\'e}rale de Lausanne, SV BMI UPHESS, Station 8, CH-1015 Lausanne, Switzerland}
\email{kathryn.hess@epfl.ch}

\author{Nima Rasekh}
\address{Institut f{\"u}r Mathematik und Informatik, Universit{\"a}t Greifswald, Greifswald, Germany}
\email{nima.rasekh@uni-greifswald.de}



\begin{document}

\begin{abstract}
  Hochschild homology has proved to be an important invariant in algebra and homotopy theory, in particular due to its relevance in algebraic $K$-theory and fixed point theory, leading to the development of numerous variants of the original construction. Ponto introduced a bicategorical axiomatization of Hochschild homology-type invariants, called a \emph{shadow}, which captures the essential common properties of all known variants of Hochschild homology, such as \emph{Morita invariance}.
  
  In this paper we clarify the relationship between shadows and Hochschild homology. After extending the notion of Hochschild homology to bicategories in a natural manner, we prove the existence of a universal shadow on any bicategory $\B$, taking values in the Hochschild homology of $\B$, through which all other shadows on $\B$ factor. Shadows are thus co-represented by a bicategorical version of Hochschild homology. Using the universal shadow on the free adjunction bicategory, we can then establish a universal Morita invariance theorem, of which all known cases are immediate corollaries. 
  
  Building on this understanding of shadows on bicategories, we propose an $\infty$-categorical generalization of shadows as functors out of Hochschild homology of an $(\infty,2)$-category in the sense of Berman. As a first step towards constructing relevant examples of $\infty$-categorical shadows, we define the Hochschild homology of enriched $\infty$-categorical bimodules and prove that they assemble into a shadow.

  As part of this work we compute the Hochschild homology of several important $2$-categories (such as the free adjunction), which can be of independent interest. 
\end{abstract}

\maketitle

\addtocontents{toc}{\protect\setcounter{tocdepth}{1}}
\tableofcontents

\section{Introduction}
\subsection{From Hochschild homology to shadows}
{Hochschild homology} of associative rings, first studied by Hochschild \cite{hochschild1945hh} and Cartan-Eilenberg \cite{cartaneilenberg1956homologicalalgebra}, generalizes the notion of {K{\"a}hler differentials} from the commutative  to the merely associative setting. Hochschild homology is an important tool for non-commutative geometry \cite{connesmarcolli2008noncommgeometry},  which satisfies interesting properties such as {Morita invariance} \cite{loday1998cyclic}  and which provides a useful approximation to {algebraic $K$-theory} via the {Dennis trace} \cite{dennis1976ktheory,waldhausen1979ktheoryii}. The homotopical analogue of classical Hochschild homology, {topological Hochschild homology} ($\THH$) of ring spectra, admits an analogous Dennis trace map from algebraic $K$-theory \cite{bokstedt1985thh,ekmm1997stablehomotopy}. 

Since the definition of algebraic $K$-theory has been extended to many types of structured categories \cite{waldhausen1985ktheory}, the important connection between $\THH$ and algebraic $K$-theory of ring spectra inspired the definition of similar extensions of (topological) Hochschild homology: to exact categories \cite{mccarthy1994cyclic}, to dg-categories \cite{keller1999cyclichomology}, and to spectral categories \cite{blumbergmandell2012thh}, as a special case of topological Hochschild homology of bimodules over spectral categories. It can be challenging to establish properties, such as Morita invariance, of certain of these variants of Hochschild homology. 

An axiomatic framework that encompasses all of these different settings is therefore worth developing. Ponto provided one such framework, when she introduced the notion of a {\it shadow of bicategories}, as a tool to study fixed-point phenomena \cite{ponto2010shadow}. In \cite{campbellponto2019thh} she and Campbell proved that  topological Hochschild homology of bimodules over spectral categories indeed provides an example of a shadow. They  showed moreover that  Morita invariance of $\THH$ of ring spectra is simply a special case of a more general, abstract Morita invariance of shadows. This observation, combined with the possibility of performing concrete computations with shadows via string diagrams \cite{pontoshulman2013shadow}, provides strong incentive to determine whether the extensions of $\THH$ to other types of structured categories are also shadows, and, if so, how homotopy-coherent they are. The issue of homotopy coherence is the focus of our attention in this article.

\subsection{The rise of homotopy coherence}
Over the past two decades, a rich theory of $\infty$-categories has been developed \cite{joyal2008notes,joyal2008theory,lurie2009htt}, in which, instead of focusing only on homotopy categories, one studies objects of interest in a homotopy-coherent fashion.  The focus on $\infty$-categories has had a significant influence on the study of algebraic $K$-theory and $\THH$. In this framework, algebraic $K$-theory can be described as a functor associating spaces to stable $\infty$-categories of a certain type, which satisfies certain universal properties, giving rise to a new definition of  the Dennis trace to $\THH$ \cite{blumberggepnertabuada2013ktheory,blumberggepnertabuada2014ktheory,bgmn2021ktheory}. The $\infty$-categorical perspective provides new insight into $\THH$ of ring spectra and its cyclotomic structure \cite{nikolausscholze2018tc,nikolausscholze2019tccorrection}. 

As sketched above, the development of various notions of Hochschild homology eventually inspired the notion of a shadow on a bicategory, which captures the essential structure of those diverse constructions. Given that many of these constructions have now been generalized to the $\infty$-categorical framework, it is natural to seek an analogous definition of an $\infty$-categorical shadow. Since shadows were originally defined on bicategories, the $\infty$-categorical analogue should be defined on $(\infty,2)$-categories, which are the higher-categorical analogue of bicategories \cite{bergner2020surveyn}.

\subsection{From shadows back to Hochschild homology} 
To determine how best to generalize shadows from bicategories to $(\infty,2)$-categories, it is helpful to examine more carefully the relationship between shadows and Hochschild homology. In particular, given that shadows are defined in the abstract setting of bicategories, what is the underlying reason why they satisfy Morita invariance, which is often thought of as a concrete property of rings and their categories of modules? At this point, a more tangible description of shadows can be helpful. Concretely, a shadow $\shadow{-}$ on a bicategory $\B$ taking values in a category $\D$ associates to every $1$-endomorphism in $F\colon X \to X$ in $\B$ an object $\shadow{F}$ in $\D$, such that for two $1$-morphisms $F\colon X \to Y$ and $G\colon Y \to X$, there is an isomorphism $\shadow{FG} \cong \shadow{GF}$ in $\D$, along with sufficient functoriality and coherence conditions. See \cref{defn:shadow} for a detailed definition. 

In this article, we clarify the striking and somewhat mysterious relationship between shadows and Hochschild homology, via a careful analysis of the coherence conditions of shadows. We extend the notion of Hochschild homology to bicategories, to an invariant that we denote $\biHH$ and think of as a ``categorification'' of Hochschild homology, and prove the following result relating Hochschild homology of bicategories and shadows. 

\begin{thmn}[\ref{thm:main theorem}]
For any bicategory $\B$ and category $\D$, there is a natural equivalence of categories
 $$\Fun\big(\biHH(\B),\D\big) \xrightarrow{ \ \simeq \ } \Sha(\B,\D).$$
\end{thmn}
Here $\biHH(\B)$, which is a category, is the Hochschild homology of the bicategory $\B$ (\cref{def:biHH}), $\Fun (-,-)$ denotes the functor category, and $\Sha$ is the category of shadows on $\B$ taking values in $\D$ (\cref{def:category of shadows}).

Thanks to this theorem, we know why shadows resemble Hochschild homology: they are co-represented by a bicategorical version of Hochschild homology. One important consequence of this co-representation is the existence of a \emph{universal shadow}, which is a key input for many formal properties of shadows.

\begin{corn}[\ref{cor:universal}]
  Let $\B$ be a bicategory. There is a {\it universal shadow} $\shadow{-}_u$ on $\B$ taking values in $\biHH(\B)$ such that for every other shadow  $\shadow{-}$ on $\B$ taking values in a category $\D$, there is a unique functor $F: \biHH(\B) \to \D$ such that $F_*\shadow{-}_u = \shadow{-}$.
 \end{corn}

\subsection{Categorical obstruction theory} \label{subsec:obstruction}
Before we consider the implications of \cref{thm:main theorem}, it is instructive to make explicit the intuition underlying the theorem, which we call \emph{categorical obstruction theory}. Obstruction theory is an important technique in algebraic topology for reducing the proof of the existence of desired maps of topological spaces or spectra (often lifts) to establishing specific properties of certain homotopy or cohomology classes (often their vanishing). There are many prominent examples of such obstruction theory results, such as those exploiting characteristic classes of bundles \cite{milnorstasheff1974charclasses} or Goerss-Hopkins obstruction theory \cite{goersshopkins2004modulispace}. 

Since $\biHH(\B)$ is defined as a pseudo-colimit of categories \cref{eq:thh bicat}, a functor out of $\biHH(\B)$ is by definition a suitably coherent cocone, which in particular is a coherent lift. In \cref{sec:pseudo diag}, we reduce the existence of such lifts to the vanishing of several categorical obstructions, culminating in \cref{thm:trunc simpl lift}. From this perspective, a shadow can be understood as representing the vanishing of a minimal set of obstructions to the existence of a certain  pseudo-cocone. 

\subsection{Morita invariance and the universal Euler characteristic}
Beyond the theoretical implications, our characterization of shadows via functors out of $\biHH(\B)$ has many practical applications, in particular with regard to Morita invariance. In \cite{campbellponto2019thh}, Campbell and Ponto show that for any shadow $\shadow{-}$ on a bicategory $\B$ with values in a category $\D$ and any adjunction diagram $\sigma$
\[ 
  \begin{tikzcd}[row sep=0.5in, column sep=0.5in]
  C \arrow[r, bend left = 40, "F"] \arrow[loop, in=130, out=225, looseness=10, ""{name=D, below}, "GF"]  & 
  D \arrow[l, bend left =40, "G"] \arrow[loop, in=45, out=315, looseness=10, ""{name=U}, "FG"']
  \arrow[from=1-1, to=D, Rightarrow, "u" description]
  \arrow[from=U, to=1-2, Rightarrow, "c" description]
  \end{tikzcd}
\]
in $\B$, the associated \emph{Euler characteristic} $$\chi(\sigma)\colon \shadow{U_C} \xrightarrow{ \shadow{u} } \shadow{GF} \cong \shadow{FG} \xrightarrow{ \shadow{c} } \shadow{U_D}$$ in $\D$, where $\id_C$ and $\id_D$ are the respective identity $1$-cells, is an isomorphism whenever the adjunction is an adjoint equivalence. Applying this result to the bicategory of ring spectra and bimodule spectra and the shadow corresponding to $\THH$, one recovers the classical Morita equivalence, namely if the categories of modules over two ring spectra are equivalent, then their topological Hochschild homology spectra are equivalent. 

We construct here the \emph{universal Euler characteristic}, for which we establish an analogous universal Morita invariance result, which implies all known particular cases. More generally, we introduce the \emph{universal invariance method} (\cref{rem:universal}), which applies not only to the Euler characteristic, but also to other invariants introduced in \cite{campbellponto2019thh}, such as traces, and would even apply more broadly. 

Focusing on the particular case of the Euler characteristic, we recall the existence of the \emph{free adjunction $2$-category} $\Adj$ that co-represents adjunctions in a bicategory $\B$, i.e., every adjunction in $\B$ corresponds to a unique $2$-functor $\Adj \to \B$ \cite{schanuelstreet1986freeadj}. Combining these insights, we define the universal Euler characteristic as a certain morphism $\widehat{\chi}_U$ in $\biHH(\Adj)$ and prove that every other Euler characteristic can be derived from $\widehat{\chi}_U$. 

\begin{thmn}[\ref{thm:chi functor}]
    Let $\B$ be a bicategory. The functor 
     \[ 
         \widehat\chi\colon \Adj(\B)  = \Fun(\Adj,\B) \xrightarrow{ \ \biHH \ } \Fun\big(\biHH(\Adj),\biHH(\B)\big) \xrightarrow{ \ (\widehat \chi_U)^* } \Fun\big([1],\biHH(\B)\big) = \Arr\big(\biHH(\B)\big)
     \]   
    takes an adjunction $\sigma=(C,D,F,G,u,c)$ in $\B$ to its Euler characteristic with respect to the universal shadow $\shadow{-}_u$ on $\B$, 
    $$\widehat \chi(\sigma) \colon \shadow{\id_C}_u \to \shadow{GF}_u \cong \shadow{FG}_u \to \shadow{\id_D}_u$$
    in $\biHH(\B)$. 
   
    In particular, for any shadow $\shadow{-}$ on $\B$ with values in $\D$, the Euler characteristic $\chi(\sigma)$ is equal to $T \circ \widehat\chi(\sigma)$, where $T\colon\biHH(\B) \to \D$ corresponds to $\shadow{-}$ via \cref{thm:main theorem}. 
 \end{thmn}

Having established this universal perspective, we can directly deduce Morita invariance of shadows (\cref{prop:shadow morita invariant}). Beyond recovering known results, the universal Euler characteristic enables us to prove new results that were not approachable with previous methods. After showing that $\biHH(\Adj)$ is equivalent to a slight variant $(\Lambda_\infty)^{\lhd\rhd}$ of the paracyclic category $\Lambda_\infty$, we obtain the following uniqueness result for the universal Euler characteristic. 

\begin{corn}[\ref{cor:euler canonical}]
    The universal Euler characteristic $\widehat{\chi}_U$ is the \emph{unique} morphism from the initial to the terminal object in $\biHH(\Adj) \simeq (\Lambda_\infty)^{\lhd\rhd}$. Hence, the definition of the Euler characteristic is canonical.
\end{corn}

\subsection{A theory of $(\infty,2)$-shadows}
Finally, we use our characterization of shadows on bicategories via Hochschild homology to propose a generalization of shadows to $\infty$-categories. Concretely, in \cref{def:shadow infty} we define a shadow of $\infty$-categories as a functor out of the $\infty$-categorical Hochschild homology of an $(\infty,2)$-category (\cref{ex:infty two cats}). This definition relies fundamentally on the work of Berman, who extended the notion of Hochschild homology to enriched $\infty$-categories \cite{berman2022thh}. 

Building on work of Haugseng on enriched $\infty$-bimodules \cite{haugseng2016bimodules}, we generalize Berman's constructions and define Hochschild homology of enriched bimodules (\cref{def:thh bimod}). We prove that for a suitable enriching $\infty$-category, these assemble into a shadow, as formulated below. 

\begin{thmn}[\ref{thm:thh shadow}] 
  Let $\V$ be a presentably symmetric monoidal $(\infty,1)$-category and $\Mod_\V$ the $(\infty,2)$-category of $\V$-enriched categories and their bimodules (\cref{thm:vbimodules}). There is a functor of $\infty$-categories
  $$\widehat {\HH}_\V\colon \biHHinfty(\Mod_\V) \to \Ho\V$$
  such that the corresponding shadow on the homotopy category of the $(\infty, 2)$-category $\Mod_\V$  (cf. \cref{prop:shadow on infty bicat}) sends any $\C$-bimodule $\M$ to $\HH_\V(\C,\M)$, the $\V$-enriched Hochschild homology of $\C$ with coefficients in $\M$ (\cref{def:thh bimod}).
\end{thmn}

We conjecture that this construction in fact lifts to a proper shadow of $\infty$-categories, if certain technical challenges can be overcome, as further discussed in \cref{conj:main} and \cref{rmk:challenges}. 

\subsection{Future directions}
The results proven here suggest several interesting and natural  next steps.
\begin{enumerate}
    \item \textbf{Tricategorical Shadows:} As discussed in \cref{subsec:obstruction}, a key step in the proof of \cref{thm:main theorem} is an argument in categorical obstruction theory that characterizes pseudo-categorical $2$-truncated simplicial cocones (cf. \cref{sec:pseudo diag}). We venture that categorical obstruction theory can similarly be applied to define and study tricategorical shadows, via a characterization of the obstructions of $3$-truncated simplicial pseudo cocones.
    \item \textbf{coHochschild homology via shadows:} There is a natural definition of coHochschild homology (e.g., of coalgebra spectra \cite{hessshipley2021cothh, bohmann2018cothh,bgs2022cothh}), dual to that of Hochschild homology. Given this duality, there should be an axiomatic, shadow-type approach to coHochschild homology. However, since essentially the only examples of coalgebras in point-set models of spectra are suspension spectra \cite{perouxshipley2019coalgebras},  a strict bicategorical approach to studying coHochschild homology would be of limited interest for spectra, though there is an interesting strictly bicategorical approach in the differential graded context \cite{hessparentscott2009}, \cite{hess2016twisting}. The results here create an opportunity to define $\coTHH$ of a coalgebra spectrum as $\THH$ of its spectral category of comodules. 
    \item \textbf{Functoriality of Hochschild homology:} One current challenge when studying Hochschild homology of enriched $\infty$-categories is that the definition given in \cite{berman2022thh} is not known to be functorial. Having a working functorial construction would, for example, permit us to generalize the Morita invariance of shadows beyond bicategories (\cref{rem:functoriality thh v}). 
    \item \textbf{Hochschild homology as a trace:} A key result in \cite{campbellponto2019thh} is that $\THH$  for bimodules over spectral categories itself is a shadow. In \cref{thm:thh shadow} we generalize this result to a functor of $\infty$-categories valued in $\Ho(\V)$, however, the expectation is that it should lift to an $\infty$-categorical functor
    $$\biHH_\infty(\Mod_V) \to \V,$$
    which should, under the right circumstances, even be $\V$-enriched, as discussed in \cref{conj:main}, if some theoretical challenges regarding enriched $(\infty,2)$-categories can be resolved (\cref{rmk:challenges}).
\end{enumerate}

\subsection{Notation} \label{subsec:notation}
 We use several types of categories throughout this article and thus need to distinguish between them carefully.
 
 A {\it category} is a $(1,1)$-category. A crucial example of a category is the simplex category $\DD$, which has as objects finite linear ordinals $[n]$ (which we also view as categories in the usual way) and as morphisms order-preserving maps (which are also functors). We denote by $\cat$ the large $(2,1)$-category of small categories, i.e., $\cat$ has 
 \begin{itemize}
     \item small categories as objects,
     \item functors as morphisms, and
     \item natural isomorphisms as $2$-morphisms.
 \end{itemize}
In particular, the existence of a commuting diagram of functors in $\cat$
 \begin{center}
     \begin{tikzcd}[row sep=0.5in, column sep=0.5in]
      \C_1 \arrow[r, "F_1"] \arrow[d, "F_2"'] & \C_2 \arrow[d, "F_3"] \\
      \C_3 \arrow[r, "F_4"'] & \C_4
     \arrow[from=1-2, to=2-1, Rightarrow, "\rcong" description]
    \end{tikzcd}
 \end{center}
means that there is a natural isomorphism 
 $$\alpha\colon F_3 \circ F_1 \xrightarrow{ \ \cong \ } F_4 \circ F_2.$$
 
 For any two categories $\C,\D$, we let  $\Fun(\C,\D)$ denote the {\it functor category}, whose objects are functors from $\C$ to $\D$, and whose morphisms are natural transformations.
 
 In this article a {\it bicategory} $\B$ is a $(2,2)$-category, composed of
  \begin{itemize}
      \item a class of objects $X,Y,Z,...$,
      \item  a category of morphisms $\B(X,Y)$ for every pair of objects $X,Y$, where the unit object in $\B(X,X)$ is denoted $U_X$, and
      \item natural isomorphisms $a$, $r$ and $l$ that witness associativity, right unitality, and left unitality, respectively and that satisfy certain axioms that the reader can find, for example, in \cite{benabou1967bicat, leinster1998basic}.
  \end{itemize}  
 
Note that $\cat$ is an example of a bicategory. Finally, there is a fully faithful embedding from $(1,1)$-categories into bicategories, which takes a category $\C$ to the bicategory with only identity $2$-morphisms. Throughout, whenever needed, we will abuse notation and not provide explicit notation for this embedding.  

In \cref{sec:traces of infty two categories}, we also use (enriched) $\infty$-categorical and particularly $(\infty,2)$-categorical formalism, which we review at the beginning of that section. We note here that we work with various notions of the ``category of spectra''.  In \cref{sec:shadows}, we consider a monoidal model category of spectra, denoted $\mathcal Sp$, e.g., symmetric spectra or orthogonal spectra \cite{mmss2001spectra}. We denote the monoidal product on $\mathcal Sp$ by $\wedge$. We let $\Sp$ denote the $\infty$-category of spectra and $\Ho(\Sp)$ its homotopy category, which is equivalent to the homotopy category of the model category $\mathcal Sp$.
 
 \subsection{Background}
 We assume only basic familiarity with category theory and a healthy curiosity about $\THH$.
 We review in \cref{sec:shadows} the relevant definitions of  $\THH$ of spectral categories and shadows. Moreover, we review relevant material regarding enriched $\infty$-categories, and the $\infty$-categorical definition of $\THH$ of enriched $\infty$-categories in \cref{subsec:thh of enriched categories}, when we generalize to the $(\infty,2)$-categorical setting. 
 
\subsection{Acknowledgments}
 We would like to thank John Berman for helpful conversations and for clarifying several points in his paper \cite{berman2022thh}. We thank Rune Haugseng as well for helpful discussions regarding his paper \cite{haugseng2016bimodules} and Thomas Nikolaus for helpful discussions and suggestions.  We also express our appreciation to the referee for excellent suggestions for restructuring this article. Finally, the second author is grateful to the Max Planck Institute for Mathematics in Bonn for its hospitality and financial support.
 
\section{Shadows of bicategories}
\label{sec:shadows}
We review here the notion of a shadow, which can be viewed  as an axiomatization of topological Hochschild homology ($\THH$), which was originally defined for ring spectra \cite{bokstedt1985thh}, but then generalized to spectral categories, i.e., categories enriched over a monoidal model category of spectra $\mathcal Sp$ \cite{blumbergmandell2012thh}. Shadows first appeared in \cite{ponto2010shadow}, though \cite{campbellponto2019thh} is our main reference.

 \begin{defn}
   Let $\C$ and $\D$ be spectral categories. 
   \begin{enumerate}
       \item A {\it $\C$-module} is a spectral functor 
       $$\C \to \mathcal Sp.$$
       \item A {\it $(\C,\D)$-bimodule} is a spectral functor $$ \C^{op} \wedge \D \to \mathcal Sp,$$
       where $ \C^{op} \wedge \D$ is the spectral category with as objects ordered pairs $(c,d)$, where $c \in \C$ and $d \in \D$, and with mapping spectrum 
       $$\Map_{\C^{op} \wedge \D}\big((c_1,d_1),(c_2,d_2)\big) = \Map_\C(c_2,c_1) \wedge \Map_\D(d_1,d_2).$$
   \end{enumerate}
 \end{defn}
 
 \begin{defn}[{\cite[Definition 2.5]{campbellponto2019thh}}]
 \label{Def:thh of spectral cat}
  Let $\C$ be a be a pointwise-cofibrant spectral category (i.e., $\C(c,d)$ is a cofibrant spectrum for all objects $c,d$) and $\D$ a $(\C,\C)$-bimodule that is pointwise-cofibrant as a spectral category. The {\it topological Hochschild homology of $\C$ with coefficients in $\D$}, denoted $\THH(\C;\D)$, is  the geometric realization of the cyclic bar construction $N^{cy}_\bullet(\C, \D)$, i.e., the simplicial spectrum that at level $n$ is  
  $$N^{cy}_n(\C, \D) = \bigvee_{c_0,...,c_n} \C(c_0,c_1) \wedge \C(c_1,c_2) \wedge... \wedge \D(c_n,c_0),$$
  with faces given by composition in $\C$ or the action of $\C$ on $\D$ (together with a cyclic permutation for the last face) and degeneracies given by the unit,
i.e.,
  $$\THH(\C;\D) = |N^{cy}_\bullet(\C, \D)|.$$
\end{defn} 

 Ponto's key insight was that this definition could be axiomatized. Concretely, let $\Ho(\Mod)$ be the bicategory with pointwise-cofibrant small spectral categories as objects  and with the morphism category from $\C$ to $\C'$ equal to the homotopy category of $(\C,\C')$-bimodules, $\Ho(\Mod_{(\C,\C')})$.
For every pointwise-cofibrant spectral category $\C$, topological Hochschild homology gives rise to a collection of functors 
 $$\THH\colon \Ho(\Mod)(\C,\C) = \Ho(\Mod_{(\C,\C)}) \to \Ho(\Sp)\colon \D \mapsto \THH (\C; \D),$$
satisfying certain properties, which become the axioms in the definition of a shadow, which we recall now.
 
 \begin{defn}[{\cite[Definition 2.16]{campbellponto2019thh}}]
 \label{defn:shadow}
  Let $\B$ be a bicategory with associator $a$, left unit $l$, right unit $r$, and identity $1$-cells $U_X$, and let  $\D$ be a category. 
  A {\it shadow} on $\B$ with values in $\D$  is a functor 
  $$\shadow{-} \colon \coprod_{X \in \mathrm{Obj}(\B)} \B(X,X) \to \D$$
  that satisfies the following conditions. 
  
  For every pair of 1-morphisms $F\colon X \to Y$ and $G\colon Y \to X$ in $\B$ that are composable in either order, there is a natural isomorphism 
  $$\theta\colon \shadow{FG} \xrightarrow{ \ \cong \ } \shadow{GF}$$
  such that for all $F\colon X \to Y$, $G\colon Y \to Z$, $H\colon Z \to X$, and $K\colon X\to X$ the following diagrams in $\D$ commute up to natural isomorphism.
  
 \begin{equation} \label{eq:shadow cond one}
     \begin{tikzcd}[row sep=0.5in, column sep=0.5in]
     \shadow{H(GF)} \arrow[r, "\theta"] \arrow[d, "\shadow{a}"'] & \shadow{(GF)H} \arrow[r, "\shadow{a}"] & \shadow{G(FH)} \\
     \shadow{(HG)F} \arrow[r, "\theta"] & \shadow{F(HG)} \arrow[r, "\shadow{a}"] & \shadow{(FH)G} \arrow[u, "\theta"'] 
     \end{tikzcd}
 \end{equation}
 
 \begin{equation} \label{eq:shadow cond two}
     \begin{tikzcd}[row sep=0.5in, column sep=0.5in]
      \shadow{KU_X} \arrow[r, "\theta"] \arrow[dr, "\shadow{r}"'] & \shadow{U_XK} \arrow[d, "\shadow{l}"] \arrow[r, "\theta"] & \shadow{KU_X} \arrow[dl, "\shadow{r}"] \\
      & \shadow{K} & 
     \end{tikzcd}
 \end{equation}
 \end{defn}
 
The primary example of a shadow is $\THH$ of spectral categories.
 
 \begin{thm}[{\cite[Theorem 2.17]{campbellponto2019thh}}]\label{thm:cp}
  Topological Hochschild homology of spectral categories is a shadow, i.e., the family of functors
  $$\Big\{ \THH (\C;-)\colon \Ho (\Mod_{(\C,\C)}) \to \Ho(\Sp)\colon \D \mapsto \THH(\C; \D)\mid \C \in \mathrm{Ob}\,\Ho (\Mod)\Big\}$$
  is such that for all pointwise-cofibrant spectral categories $\C$ and $\C'$, $(\C, \C')$-bimodules $\D$, and $(\C', \C)$-bimodules $\D'$, there is a natural isomorphism 
$$\theta\colon \THH(\C, \D\wedge_{\C'} \D')\xrightarrow{\cong}\THH(\C', \D'\wedge_{\C} \D) $$ 
satisfying the conditions of \cref{defn:shadow}.
 \end{thm}
 
It is common to simplify notation and to write
$$\THH(\C)=\THH(\C;\C),$$
where we consider $\C$ as a bimodule over itself in the obvious way. Note that computing THH of a spectral category yields a spectrum, i.e., an object in the enriching category.
 
The axioms of a shadow suffice to prove interesting properties that are known to hold for $\THH$. For example, Ponto and Campbell prove a general form of Morita equivalence for shadows \cite[Proposition 4.6, Proposition 4.8]{campbellponto2019thh} that implies the classical Morita equivalence for ring spectra \cite[Example 5.10]{campbellponto2019thh}. In \cref{thm:main theorem} we show that these properties stem from a deep connection between shadows and Hochschild homology. This is the content of the next section, after we introduce Hochschild homology of bicategories.
 
\section{Shadows vs. Traces} \label{sec:shadows vs traces}
In the last section we introduced shadows on bicategories, with $\THH$ of spectral categories a primary example thereof. In this section, we demonstrate that this example is not coincidental and, in fact, captures the essence of shadows. To do so, we extend Hochschild homology to bicategories in \cref{subsection:hh of bicategories} and then apply this notion in \cref{subsec:the main result}, to formulate the correspondence between shadows and functors out of bicategorical Hochschild homology, which we call \emph{traces}.  
  
\subsection{Hochschild homology of bicategories} \label{subsection:hh of bicategories}
In this subsection we introduce Hochschild homology of bicategories. As we saw in \cref{Def:thh of spectral cat}, the Hochschild homology of a spectral category is defined as the geometric realization (i.e., a colimit) of a certain simplicial object, namely the cyclic bar construction. Our goal is to define an analogous construction for bicategories, extracting a suitably defined cyclic bar construction out of every bicategory. As a first step, we define the relevant maps. 

\begin{rmk} \label{rmk:product notation}
	In order to simplify notation throughout, for a given bicategory $\B$, we use the following notational convention:
			$$\B(X_0,X_1,...,X_n,X_0) = \B(X_0,X_1) \times \B(X_1,X_2) \times ... \times \B(X_n,X_0).$$
\end{rmk}
	
\begin{notn}
	For any bicategory $\B$ and triple $X_0,X_1,X_2$ of objects of $\B$, the following functors play an important role in the definition of bicategorical Hochschild homology. 
 \begin{enumerate}
		\item $d_0,d_1\colon \B(X_0,X_1,X_0) \to \B(X_i,X_i)$ (where $i=0,1$), given by
\[
d_0(F,G) = FG, \quad
d_1(F,G) = GF.
\]
 \item $d_0,d_1,d_2\colon \B(X_0,X_1,X_2,X_0) \to \B(X_i,X_j)$ (where $ 0 \leq i < j \leq 2$), given by 
\[
d_0(F,G,H)=(FG,H), \quad 
d_1(F,G,H)=(F,GH), \quad
d_2(F,G,H)=(HF,G). 
\]
 \item $s_0\colon \B(X_0,X_0) \to \B(X_0,X_0,X_0)$, given by
\[
s_0(F) = (F,\id_{X_0}).
\]
\item $s_0, s_1\colon  \B(X_0,X_1,X_0) \to \B(X_0,X_1, X_i, X_0)$ (where $i=0,1$), given by 
\[s_0(F,G) = (F,\id_{X_1},G), \quad 
  s_1(F,G) = (F,G,\id_{X_0}).\] 
	\end{enumerate}
\end{notn}

Ideally, we would like these functors to give us a (truncated) simplicial object, which would be the intended cyclic bar construction. Due to the coherences inherent in the definition of a bicategory, however, we obtain instead a diagram in the $(2,1)$-category $\cat$ (\cref{subsec:notation}), as specified in the following lemma. Recall from \cref{subsec:notation} that here we consider $\DD_{\leq 2}$, which is the full subcategory of $\DD$ with objects $[0], [1],$ and $[2]$, as a $2$-category with only identity $2$-morphisms.

\begin{lemma} \label{lemma:bicategorical cyclic bar construction}
 Let $\B$ be a bicategory. There exists a $(2,1)$-functor $\biN^{cy}_\bullet(\B)\colon \DD_{\leq 2}^{op} \to \cat$ of the following form, where the rightwards-pointing arrows are appropriate choices of $s_i$. 
 {\begin{equation} \label{eq:thh bicat}
		\begin{tikzcd}[column sep=0.3in]
			\ds\coprod_{X_0} \B(X_0,X_0) 
			\arrow[r, shorten >=1ex,shorten <=1ex] 
			&
			\ds\coprod_{X_0,X_1} \B(X_0,X_1,X_0) 
			\arrow[l, shift left=1.2, "d_1"] \arrow[l, shift right=1.2, "d_0"'] 
					\arrow[r, shift right, shorten >=1ex,shorten <=1ex ] \arrow[r, shift left, shorten >=1ex,shorten <=1ex] 
			& 
			\ds\coprod_{X_0,X_1,X_2} \B(X_0,X_1,X_2,X_0)
				\arrow[l] \arrow[l, shift left=2, "d_2"] \arrow[l, shift right=2, "d_0"'] 
			\end{tikzcd}
			\end{equation}}			
\end{lemma}

\begin{proof}
 We can directly construct the functor, relying on \cref{lemma:trunc simp diag}. Here the objects and $1$-morphisms are already provided in the statement, the $\alpha_i$ are given by the associators in $\B$, the $\beta_i$ and $\gamma_i$ are given by the unitors in $\B$, and $\delta_0$ is the identity. The required coherence among the 2-cells is provided by the axioms of a bicategory.
	
	Alternatively, we can observe that the construction is a special case of Berman's construction for enriched $\infty$-categories, when the enrichment is given by $\cat$. See \cref{prop:bicat} for more details. 
\end{proof}

We can now define the Hochschild homology of a bicategory.

\begin{defn}\label{def:biHH}
 The \emph{Hochschild homology} of a bicategory $\B$, denoted $\biHH(\B)$, is defined as the colimit of	the functor $\biN^{cy}_\bullet(\B)$ in $\cat$.
\end{defn}
  
 \begin{rmk} \label{rem:thh bicat functorial}
			The construction of $\biN^{cy}_\bullet(\B)$ is evidently functorial in $\B$, and hence so is the construction of $\biHH(\B)$ as its colimit.
\end{rmk}
  
\begin{notn}\label{not:simplicial maps}
  Let $\B$ be a bicategory. For $B_1, B_2$ chosen among the three categories $$\coprod_{X_0}\B(X_0,X_0),\coprod_{X_0,X_1}\B(X_0,X_1,X_0),\coprod_{X_0,X_1,X_2}\B(X_0,X_1,X_2,X_0)$$ we use   $\Fun^\Delta(B_1,B_2)$ to denote the subcategory of $\Fun(B_1,B_2)$ consisting of functors and natural isomorphisms generated by those in \cref{eq:thh bicat}. 
\end{notn}
  
It is important to note that the colimit of  \cref{eq:thh bicat} is computed in the $(2,1)$-category $\cat$, which differs from the $1$-categorical colimit in the $1$-category of categories, but rather corresponds to the \emph{pseudo-colimit} \cite{kelly1989twolimits}. Computing general pseudo-colimits can be quite challenging, although there are certain helpful methods via weighted colimits of enriched categories \cite{kelly2005enriched} and  homotopy colimits of model categories \cite{gambino2008homotopylimits}. 

In general there are very few situations where one can compute $\biHH(\B)$ using formal arguments. We will see one example, when $\B$ is a \emph{bigroupoid} (i.e., all $1$-morphisms and $2$-morphisms are invertible) using abstract homotopy theory, see \cref{ex:thh groupoids}. The example below illustrates that such formal computations cannot generalize even to bicategories with one object.

\begin{ex} \label{ex:thh monoids}
  Let $(\C,\otimes,I)$ be a monoidal category, and let $B\C$ be the bicategory with one object and morphism category given by $\C$, with composition given by the monoidal product.  Interpreting \cref{eq:thh bicat} in this particular case, we see that $\biHH(B\C)$ is the colimit of the diagram 
  \begin{center} 
  \begin{tikzcd}[column sep=0.3in]
 \ds \C 
 \arrow[r, shorten >=1ex,shorten <=1ex] 
 &
 \ds \C \times \C 
 \arrow[l, shift left=1.2, "d_1"] \arrow[l, shift right=1.2, "d_0"'] 
   \arrow[r, shift right, shorten >=1ex,shorten <=1ex ] \arrow[r, shift left, shorten >=1ex,shorten <=1ex] 
 & 
 \ds \C \times \C \times \C
  \arrow[l] \arrow[l, shift left=2, "d_2"] \arrow[l, shift right=2, "d_0"'] 
 \end{tikzcd}.
 \end{center}
Although this diagram is easy to describe explicitly, there is no easy way to compute its colimit, as illustrated in \cref{sec:thh two cat}. In general the only description available of  the category $\biHH(B\C)$ is formulated in terms of  generators and relations \cite[Section II.8]{maclane1998categories}, which is generally not useful for computations.
\end{ex}

In certain cases we can give a more precise description of $\biHH(B\C)$ using very explicit computational arguments, which is the content of \cref{sec:thh two cat}, culminating in \cref{cor:thh two cat on obj}.

\begin{rmk}
 It is straightforward to check that any functor $T\colon\biHH(\B) \to \D$ satisfies a trace-like property, i.e., for any two $1$-morphisms $A\colon X \to Y$ and $B\colon Y \to X$ in $\B$, there is an isomorphism $T(AB)\cong T(BA)$ in $\D$. See \cref{rem:trace} for a more detailed discussion. 
\end{rmk}

The remark above motivates the following definition. 

\begin{defn} \label{defn:trace bicat}
 Let $\B$ be a bicategory and $\D$ a category. A \emph{bicategory trace} on $\B$ taking values in $\D$ is a functor $\biHH(\B) \to \D$. 
\end{defn}

As we observed above, there is usually no easy way to describe $\biHH(\B)$ and hence, in particular, functors out of $\biHH(\B)$. We therefore make use of the colimit presentation of $\biHH(\B)$ (\cref{eq:thh bicat}) to describe the category $\Fun\big(\biHH(\B),\D\big)$ of traces on a bicategory $\B$ with values in a category $\D$. 

Let $\big((\DD_{\leq 2})^{op}\big)^{\rhd}$ be the join of $(\DD_{\leq 2})^{op}$ with the terminal category. Intuitively $\big((\DD_{\leq 2})^{op}\big)^{\rhd}$ is the category $(\DD_{\leq 2})^{op}$ equipped with a new terminal object. Note there are evident inclusion functors $(\DD_{\leq 2})^{op} \to \big((\DD_{\leq 2})^{op}\big)^{\rhd}$ and $[0] \to \big((\DD_{\leq 2})^{op}\big)^{\rhd}$, sending the unique object $0$ to the terminal object, which we denote by $f$.

 \begin{defn} \label{def:cocone}
  Let $\B$ be a bicategory and $\D$ a category.
  
  The category $\Cocone(\B,\D)$ of \emph{$\B$-cocones taking values in $\D$}  has as objects the pseudofunctors $((\DD_{\leq 2})^{op})^{\rhd} \to \cat$ that fit into the diagram 
  \begin{center}
      \begin{tikzcd}
        (\DD_{\leq 2})^{op} \arrow[d, hookrightarrow] \arrow[dr,  "\biN^{cy}"] \\
        ((\DD_{\leq 2})^{op})^{\rhd} \arrow[r, dashed] & \cat \\
        \{f\} \arrow[ur, "\D"'] \arrow[u, hookrightarrow]
      \end{tikzcd}
  \end{center}
 and pseudonatural transformations as morphisms.
 \end{defn}
 
 Concretely we can depict an object in $\Cocone(\B,\D)$ as a diagram in $\cat$ of the form 
      \begin{equation} \label{eq:cocone}
       \begin{tikzcd}[row sep=0.3in, column sep=0.3in]
        \ds\coprod_{X_0} \B(X_0,X_0) 
        \arrow[r, shorten >=1ex,shorten <=1ex] \arrow[dr, "C_0"']
        &
       \ds\coprod_{X_0,X_1} \B(X_0,X_1,X_0) 
        \arrow[l, shift left=1.2, "d_1"] \arrow[l, shift right=1.2, "d_0"'] 
       \arrow[r, shift right, shorten >=1ex,shorten <=1ex ] \arrow[r, shift left, shorten >=1ex,shorten <=1ex] \arrow[d, "C_1"] 
      & 
      \ds\coprod_{X_0,X_1,X_2} \B(X_0,X_1,X_2,X_0)
  \arrow[l] \arrow[l, shift left=2, "d_2"] \arrow[l, shift right=2, "d_0"'] \arrow[dl, "C_2"]\\
  & \D & 
 \end{tikzcd},
\end{equation} 
and a morphism is a commutative diagram of natural transformations of the form
      \begin{equation} \label{eq:cocone morphism}
       \begin{tikzcd}[row sep=0.8in, column sep=0.3in]
         \ds\coprod_{X_0} \B(X_0,X_0) 
        \arrow[r, shorten >=1ex,shorten <=1ex] \arrow[dr, "C_0'"', bend right=20, ""{name=A}] \arrow[dr, "C_0" ""{name=U, below}, bend left=20] 
        &
       \ds\coprod_{X_0,X_1} \B(X_0,X_1,X_0) 
        \arrow[l, shift left=1.2, "d_1"] \arrow[l, shift right=1.2, "d_0"'] 
       \arrow[r, shift right, shorten >=1ex,shorten <=1ex ] \arrow[r, shift left, shorten >=1ex,shorten <=1ex] \arrow[d, "C_1'"' very near start, bend right=20, ""{name=B}] \arrow[d, "C_1" very near start, bend left=20, ""{name=V, below}] 
      & 
      \ds\coprod_{X_0,X_1,X_2} \B(X_0,X_1,X_2,X_0).
  \arrow[l] \arrow[l, shift left=2, "d_2"] \arrow[l, shift right=2, "d_0"'] \arrow[dl, "C_2'"' near start, bend right=20, ""{name=C}] \arrow[dl, "C_2", bend left=20, ""{name=W, below}]\\
  & \D & 
  \arrow[from=U, to=A, Rightarrow, "\alpha_0"', shorten <= 0.1in, shorten >= 0in]
  \arrow[from=V, to=B, Rightarrow, "\alpha_1"' near end, shorten <= 0.05in, shorten >= 0.0in]
  \arrow[from=W, to=C, Rightarrow, "\alpha_2"', shorten <= 0.1in, shorten >= 0.1in]
 \end{tikzcd}.
\end{equation}
 
Since the objects in $\Cocone(\B,\D)$ are pseudofunctors,  the commutativity in \cref{eq:cocone} and \cref{eq:cocone morphism} holds up to appropriate choices of natural isomorphisms. For example, the cocone condition in \cref{eq:cocone} implies that $C_1$ is naturally isomorphic to  $C_0 \circ d_0$ and $C_0 \circ d_1$, and $C_2$ is naturally isomorphic to $C_1 \circ d_0$, $C_1 \circ d_1$, and $C_1 \circ d_2$, such that the choices of natural isomorphisms are themselves unique, meaning there is essentially a unique coherence in each cocone. Understanding the precise data of a cocone and the various coherences and unique properties thereof is the main step of the proof of \cref{thm:main theorem}. As this description is quite technical, it has been relegated to \cref{sec:pseudo diag}.
 
 \begin{rmk} For every bicategory $\B$, there is a $\B$-cocone taking values in $\biHH(\B)$, given by the maps into the colimit.  As we observe below, this $\B$-cocone is universal, in the sense that every other $\B$-cocone factors through it.
 \end{rmk}

The identification below follows immediately from the colimit description of $\biHH(\B)$ (\cref{eq:thh bicat}) and the universal property of colimits. 
 
 \begin{prop} \label{lemma:cocones are traces}
 The functor 
 $$\Comp\colon\Fun\big(\biHH(\B),\D\big) \xrightarrow{\simeq}\Cocone(\B,\D) ,$$
which composes a functor $\biHH(\B) \to \D$ with the universal $\B$-cocone taking values in $\biHH(\B)$, is an equivalence of categories, which is natural in $\B$ and $\D$. 
 \end{prop}
 
 In other words, traces out of $\B$ into $\D$ (\cref{defn:trace bicat}) are equivalent to  $\B$-cocones  taking values in $\D$.
\subsection{The equivalence between shadows and traces}
\label{subsec:the main result}
We are finally ready to establish the main result of this section, that shadows (\cref{defn:shadow}) are equivalent to traces on bicategories (\cref{defn:trace bicat}).  To state the result precisely, we first introduce a notion of morphisms between shadows.

\begin{defn} \label{def:category of shadows}
 Let $\B$ be a bicategory and $\D$ a category. The \emph{category of shadows on $\B$ taking values in $\D$}, denoted 
 $\Sha(\B,\D)$, is specified as follows.
 \begin{enumerate}
     \item The objects are shadows, i.e., pairs $\big(\shadow{-},\theta\big)$ satisfying the conditions of \cref{defn:shadow}.
     \item Given two shadows $\big(\shadow{-}_1, \theta_1\big)$, $\big(\shadow{-}_2, \theta_2\big)$, a morphism between them consists of a natural transformation 
     $$\alpha_{-}\colon \shadow{-}_1 \to \shadow{-}_2$$
     such that:
     \begin{itemize}
         \item For any pair of 1-morphisms $F\colon X_0 \to X_1$, $G\colon X_1 \to X_0$ in $\B$, the diagram 
     \begin{equation} \label{eq:shadow cat cond one}
      \begin{tikzcd}[row sep=0.5in, column sep=0.5in]
        \shadow{FG}_1 \arrow[r, "\theta_1"] \arrow[d, "\alpha_{FG}"] & \shadow{GF}_1 \arrow[d, "\alpha_{GF}"] \\
        \shadow{FG}_2 \arrow[r, "\theta_2"] & \shadow{GF}_2,
      \end{tikzcd}
    \end{equation}
    in $\D$ commutes.
        \item For any three 1-morphisms $F\colon X_0 \to X_1$, $G\colon X_1 \to X_2$, $H\colon X_2 \to X_0$ in $\B$, the following diagram commutes
        \begin{equation} \label{eq:shadow cat cond two}
         \begin{tikzcd}[row sep=0.5in, column sep=0.5in]
         \shadow{(GF)H}_1 \arrow[r, "\shadow{a}_1"] \arrow[d, "\alpha_{(GF)H}"] & \shadow{G(FH)}_1 \arrow[d, "\alpha_{G(FH)}"] \\
         \shadow{(GF)H}_2 \arrow[r, "\shadow{a}_2"] & \shadow{G(FH)}_2
         \end{tikzcd}
        \end{equation}
        \item For any $1$-morphism $F\colon X_0 \to X_0$, the following diagram commutes 
        \begin{equation} \label{eq:shadow cat cond three}
         \begin{tikzcd}[row sep=0.5in, column sep=0.5in]
             \shadow{FU_{X_0}}_1 \arrow[r, "\shadow{r}_1"] \arrow[d, "\alpha_{FU_{X_0}}"] & \shadow{F}_1 \arrow[d, "\alpha_{F}"] & \shadow{U_{X_0}F}_1 \arrow[l, "\shadow{l}_1"'] \arrow[d, "\alpha_{U_{X_0}F}"]\\
             \shadow{FU_{X_0}}_2 \arrow[r, "\shadow{r}_2"] & \shadow{F}_2 & \shadow{U_{X_0}F}_2 \arrow[l, "\shadow{l}_2"']
         \end{tikzcd}
        \end{equation}
     \end{itemize}  
 \end{enumerate}
\end{defn}

\begin{thm}\label{thm:main theorem}
For any bicategory $\B$ and category $\D$, there is an equivalence of categories 
 $$\Fun\big(\biHH(\B),\D\big) \xrightarrow{ \ \simeq \ } \Sha(\B,\D)$$
 that factors via the functor $\Comp$ (\cref{lemma:cocones are traces}) through the category $\Cocone(\B,\D)$ (\cref{def:cocone}) and that is natural in $\B$ and $\D$. 
\end{thm}

The proof is quite technical and hence delegated to \cref{sec:the main result proof}, however, the following remark provides some general intuition.

\begin{rmk}\label{rmk:strictification}
 The proof consists of a sort of ``strictification argument". We can think of a cocone in $\Cocone(\B,\D)$ as a general diagram that we can strictify to the data of a shadow $\Sha(\B,\D)$ via the strictification functor
	\[ \St\colon \Cocone(\B,\D) \to \Sha(\B,\D),\]
	defined in \cref{prop:st functor}. On the other hand, every shadow can be ``unstrictified'' to a cocone via the unstrictification functor
	\[ \Un\colon \Sha(\B,\D) \to \Cocone(\B,\D),\]
	defined in \cref{prop:un functor}. These two operations are indeed inverses, as we establish in \cref{prop:st and un inverses}, giving us the desired equivalence. 
	Hence, the proof demonstrates that no data is lost during the strictification process. A shadow can thus be characterized as the minimal amount of data required to describe a cocone. 
\end{rmk}

We end this section with several formal, but valuable, consequences of \cref{thm:main theorem}. 

\begin{cor} \label{cor:universal}
 Let $\B$ be a bicategory. There is a \emph{universal shadow} $\shadow{-}_u$ on $\B$ taking values in $\biHH(\B)$ such that for every other shadow $\shadow{-}$ on $\B$ taking values in a category $\D$, there is a unique functor $F\colon \biHH(\B) \to \D$ satisfying $F_*\shadow{-}_u = \shadow{-}$, where $F_*$ denotes postcomposition with $F$.
\end{cor}

\begin{proof}
 Let $\shadow{-}_u$ be the shadow that corresponds to the identity functor $\biHH(\B) \to \biHH(\B)$ under the equivalence in \cref{thm:main theorem}. The desired equality now follows from the fact that the following square commutes by naturality
 \begin{center}
     \begin{tikzcd}[row sep=0.5in]
       \Fun(\biHH(\B),\biHH(\B)) \arrow[r, "\simeq"] \arrow[d, "F_*"] & \Sha(\B,\biHH(\B)) \arrow[d, "F_*"]  \\
       \Fun(\biHH(\B),\D) \arrow[r, "\simeq"] & \Sha(\B,\D) 
     \end{tikzcd},
 \end{center}
for any functor $F\colon \biHH(\B) \to \D$.
\end{proof}

\begin{cor}
 For any bicategory $\B$, the functor $\Sha(\B,-)\colon \cat \to \cat$ is corepresentable and hence preserves limits. 
\end{cor}

\begin{rmk} \label{rmk:thh trace}
Applying \cref{thm:main theorem} to the shadow on the bicategory $\Ho (\Mod)$ of \cref{thm:cp} gives rise  to a trace (\cref{defn:trace bicat}) out of $\Ho (\Mod)$ into $\Ho(\Sp)$,
\[\THH\colon\biHH\big(\Ho (\Mod)\big) \to \Ho(\Sp),\]
as defined by Campbell and Ponto in \cite{campbellponto2019thh}.
We generalize this result to $\infty$-categorical bimodules in \cref{thm:thh shadow}.
\end{rmk}

\section{Morita Invariance of Shadows and the Universal Euler Characteristic} \label{sec:morita}
An important property satisfied by $\THH$ of ring spectra is {\it Morita invariance}: if two ring spectra $R$ and $R'$ are Morita equivalent (i.e., their model categories of modules are Quillen equivalent), then $\THH(R) \simeq \THH(R')$ \cite{blumbergmandell2012thh}. In line with the idea that a shadow is an axiomatic, bicategorical generalization of $\THH$, Campbell and Ponto proved that shadows satisfied a natural generalization of Morita invariance \cite[Proposition 4.6]{campbellponto2019thh}. Their result suggests the following natural question: Can we leverage our alternative characterization of shadows via traces to recover and possibly even strengthen Morita invariance of shadows? 

We formulate precise versions of this question and possible answers in \cref{subsec:morita invariance traces}. Doing so requires a precise understanding of the work in \cite{campbellponto2019thh}, which we hence first review in \cref{subsec:morita invariance shadows}.

\subsection{Morita invariance	of shadows} \label{subsec:morita invariance shadows}
In this section we review Campbell and Ponto's approach to Morita invariance via shadows \cite{campbellponto2019thh}. Let $\B$ be a bicategory equipped with a shadow functor $\shadow{-}$ taking values in a category $\D$. Given an adjunction in $\B$, i.e., a diagram 
\begin{equation} \label{eq:adj in C}
    \begin{tikzcd}[row sep=0.5in, column sep=0.5in]
    C \arrow[r, bend left = 40, "F"] \arrow[loop, in=130, out=225, looseness=10, ""{name=D, below}, "GF"]  & 
    D \arrow[l, bend left =40, "G"] \arrow[loop, in=45, out=315, looseness=10, ""{name=U}, "FG"']
    \arrow[from=1-1, to=D, Rightarrow, "u" description]
    \arrow[from=U, to=1-2, Rightarrow, "c" description]
    \end{tikzcd}
\end{equation}
of 0-, 1-, and 2-cells in $\B$ satisfying the triangle identities,  the {\it Euler characteristic} $\chi(F)$ of the  $1$-cell $F$ is defined to be the composite morphism 
\begin{equation} \label{eq:euler}
	\shadow{U_C} \xrightarrow{ \shadow{u} } \shadow{GF} \cong \shadow{FG} \xrightarrow{ \shadow{c} } \shadow{U_D}
\end{equation}
in $\D$.

Morita invariance of shadows is formulated as follows in terms of the Euler characteristic.

\begin{prop}[{\cite[Proposition 4.6]{campbellponto2019thh}}]\label{prop:morita}
 Shadows satisfy Morita invariance. Concretely, if the adjunction in \eqref{eq:adj in C} is an equivalence, i.e., if the 2-cells $c$ and $u$ are invertible, then $\chi(F)$ is an isomorphism in $\D$ with inverse given by $\chi(G)$.
\end{prop}

This result can be generalized to arbitrary endomorphisms. For a given endomorphism $Q\colon C \to C$, there is a chain of morphisms 
\begin{equation} \label{eq:trace}
 \shadow{Q} \xrightarrow{ \ \shadow{u\rid_Q}  \ } \shadow{GFQ} \xrightarrow{\shadow{\rid_{GFQ}u}} \shadow{GFQGF} \cong \shadow{FQGFG} \xrightarrow{ \shadow{\rid_{FQG}c} } \shadow{FQG},	
\end{equation}
which is known as the {\it trace of $Q$} (which is different from the trace defined in \cref{defn:trace bicat}) and denoted $\tr(u_Q)$. Now, we have the following result due to Campbell and Ponto.

\begin{prop}[{\cite[Proposition 4.8]{campbellponto2019thh}}]\label{prop:morita end}
   If the adjunction in \eqref{eq:adj in C} is an equivalence, i.e., if the 2-cells $c$ and $u$ are invertible, then $\tr(u_Q)$ is an isomorphism in $\D$ with inverse given by $\tr(c_Q)$.
\end{prop}

\begin{rmk} \label{rem:morita shadows}
 We can summarize the contents of \cref{prop:morita} and \cref{prop:morita end} as follows: Given a bicategory $\B$, a category $\D$, a shadow on $\B$ taking values in $\D$, and a certain diagram in $\B$, we can extract a morphism in $\D$ (the Euler characteristic or trace), which under suitable conditions on the diagram in $\B$, is an isomorphism. 
\end{rmk}

\subsection{Morita invariance and universality} \label{subsec:morita invariance traces}
In the previous subsection we reviewed the Morita invariance results due to Campbell and Ponto \cite{campbellponto2019thh}. In particular, we observed in \cref{rem:morita shadows} that Morita invariance of a shadow on $\B$ in $\D$ involves constructing (iso)morphisms in $\D$ based on given bicategorical data $D$ in $\B$. Given our new ability to articulate shadows as functors $\biHH(\B) \to \D$, we can naturally generalize Morita invariance using what we call the \emph{universal invariance method}.

\begin{rmk} \label{rem:universal}
 Let $T\colon \biHH(\B) \to \D$ correspond under the equivalence of \cref{thm:main theorem} to a shadow on a bicategory $\B$ taking values in a category $\D$. The \emph{universal invariance method} consists of the following steps, which can be carried out in a number of interesting cases. 
 \begin{enumerate}
	\item Find a $2$-category $\A$ such that the data of interest $D$ in $\B$ can be encoded as a $2$-functor $D\colon\A \to \B$, meaning $\A$ is the $2$-category freely generated by the data of interest. For example, if the data of interest $D$ is an adjunction, then $\A$ is the free adjunction $2$-category.
  \item Exploit an explicit description of $\biHH(\A)$ to extract a \emph{universal morphism} $\alpha_U\colon [1] \to \biHH(\A)$.
	\item Compose the universal morphism $\alpha_U\colon [1] \to \biHH(\A)$ with the functor $T\colon\biHH(\A) \to \D$ to obtain a morphism in $\alpha= T\circ \alpha_U$ in $\D$.
	\item \emph{$\A$-invariance} then simply corresponds to requiring that under suitable conditions on $D\colon\A \to \B$, the morphisms $\alpha $ is an isomorphism, which by functoriality reduces to whether $\alpha_U$ is an isomorphism in $\biHH(\A)$. For example, if the data of interest $D$ is an adjunction, then $\A$-invariance is Morita invariance, as formulated above.
 \end{enumerate}
\end{rmk}

We now apply the universal invariance method to the Euler characteristic and trace. In addition to reducing arguments about the Euler characteristic or trace to analyzing universal choices in $\biHH(\A)$ for a suitable choice of $\A$, we can ask whether our universal morphisms are indeed unique or depend on choices, by analyzing the categories $\biHH(\A)$.

\begin{rmk}
  Beyond the Euler characteristic and trace, we could use this method to look at a broad class of diagram shapes $\A$ and thus extract interesting formal properties of shadows via the computation of $\biHH(\A)$, for relevant choices of $\A$.   
\end{rmk}

 We commence with the Euler characteristic, which necessitates focusing on adjunctions. 
  
\begin{defn} \label{def:adj} 
 Let $\Adj$ be the free strict bicategory generated by the two objects, two $1$-morphisms and two $2$-morphisms given in \eqref{eq:adj in C} and satisfying relations given by the triangle equalities.
\end{defn}

 For more details regarding the $2$-category $\Adj$, see the original description in \cite{schanuelstreet1986freeadj}.
 
\begin{defn} \label{def:adjeq} 
 Let $\AdjEq$ be the localization of $\Adj$ that is a strict bicategory with the same objects, $1$-cells and $2$-cells as $\Adj$, but where all $2$-morphisms are invertible.
\end{defn}

As in \cref{eq:adj in C}, a bifunctor out of $\AdjEq$ can be depicted as a diagram 
 \begin{center}
     \begin{tikzcd}[row sep=0.5in, column sep=0.5in]
    0 \arrow[r, bend left = 40, "f"] \arrow[loop, in=130, out=225, looseness=10, ""{name=D, below}, "g \circ f"]  & 1 \arrow[l, bend left =40, "g"] \arrow[loop, in=45, out=315, looseness=10, ""{name=U}, "f \circ g"']
    \arrow[from=1-1, to=D, Rightarrow, "u" description, "\cong"' near end]
    \arrow[from=U, to=1-2, Rightarrow, "c" description, "\cong"' near start]
    \end{tikzcd}
 \end{center}

\begin{defn} \label{def: adjb adjeqb}
 For any bicategory $\B$, let
 $$ \Adj(\B) = \Fun(\Adj, \B)$$
 and
 $$ \AdjEq(\B) = \Fun(\AdjEq,\B),$$
 the categories of {\it adjunctions in $\B$} and of {\it adjoint equivalences in $\B$}, respectively.    
\end{defn}

There is an evident localization functor $\Adj \to \AdjEq$ that induces a restriction functor 
 \begin{equation} \label{eq:restriction}
   \AdjEq(\B) \to \Adj(\B).
    \end{equation}
    
To complete our set-up, we introduce the following notation.

\begin{notn}
 For any category $\D$, let
 $$ \Arr(\D) = \Fun([1],\D),$$
 and let $\Iso(\D)$ denote the full subcategory of $\Arr(\D)$ consisting of isomorphisms.
\end{notn}

The explicit colimit description of $\biHH(\Adj)$ as a pseudo-colimit in $\cat$ (\cref{def:biHH}) implies in particular that $\biHH(\Adj)$  has as objects the  endomorphisms $\id_0,\id_1$, $fg$, and $gf$ and as $1$-morphisms $u\colon \id_0 \to gf$, $c\colon fg \to \id_1$, and an isomorphism $fg \cong gf$. These observations imply that the following definition makes sense.

\begin{defn} \label{def:universal Euler characteristic}
  The \emph{universal Euler characteristic} is the morphism 
  $$\widehat \chi_U \colon \id_0 \xrightarrow{u} gf \xrightarrow{\cong} fg \xrightarrow{c} \id_1$$ 
  in $\biHH(\Adj)$. 
\end{defn}

The universal Euler characteristic is indeed an Euler characteristic in the sense of \cite{campbellponto2019thh}.

\begin{ex}
  If $\shadow{-}_U$ be the universal shadow functor from $\Adj$ to $\biHH(\Adj)$ (\cref{cor:universal}), then the Euler characteristic of the adjunction data in $\Adj$ is precisely the universal Euler characteristic $\widehat \chi_U$. 
\end{ex}

We can now leverage the universal Euler characteristic to establish the following immediate, yet important, result. Here we recall also that $\biHH$ is indeed functorial (\cref{rem:thh bicat functorial}).

\begin{thm} \label{thm:chi functor}
 Let $\B$ be a bicategory. The functor 
  \[ 
      \widehat\chi\colon \Adj(\B)  = \Fun(\Adj,\B) \xrightarrow{ \ \biHH \ } \Fun\big(\biHH(\Adj),\biHH(\B)\big) \xrightarrow{ \ (\widehat \chi_U)^* } \Fun\big([1],\biHH(\B)\big) = \Arr\big(\biHH(\B)\big)
  \]   
takes an adjunction $\sigma=(C,D,F,G,u,c)$ in $\B$ to its Euler characteristic with respect to the universal shadow $\shadow{-}_u$ on $\B$, 
    $$\widehat \chi(\sigma) \colon \shadow{\id_C}_u \to \shadow{GF}_u \cong \shadow{FG}_u \to \shadow{\id_D}_u$$
    in $\biHH(\B)$. 
   
    In particular, for any shadow $\shadow{-}$ on $\B$ with values in $\D$, the Euler characteristic $\chi(\sigma)$ is equal to $T \circ \widehat\chi(\sigma)$, where $T\colon\biHH(\B) \to \D$ corresponds to $\shadow{-}$ via \cref{thm:main theorem}.
\end{thm}

Note that it is essential here that the codomain of this functor is the category of morphisms in $\biHH(\B)$, which ensures the existence of the isomorphism $GF \xrightarrow{\cong} FG$.

\begin{proof}
  This identification follows directly from functoriality and the computation that for a given shadow $\shadow{-}$, with associated functor $T\colon \biHH(\B) \to \D$, the Euler characteristic $\sigma$ is obtained via the composition
  \[
    [1] \xrightarrow{ \ \widehat{\chi}_U \ } \biHH(\Adj) \xrightarrow{ \ \biHH \sigma \ } \biHH(\B) \xrightarrow{T} \D
  \]
  and hence the result follows.
\end{proof}

The corollary below is an immediate consequence of the fact that every functor preserves isomorphisms.

\begin{cor} \label{cor:chi functor on iso}
 Precomposing the functor $ \widehat\chi\colon \Adj(\B) \to \Arr\big(\biHH(\B)\big)$ with the restriction functor $\AdjEq(\B) \to \Adj(\B)$ given in \eqref{eq:restriction} gives rise to a functor 
 $$\widehat\chi\colon \AdjEq(\B) \to \Iso\big(\biHH(\B)\big).$$
\end{cor}

From \cref{cor:chi functor on iso} we can recover the Morita invariance of shadows, proven originally in \cite[Proposition 4.6]{campbellponto2019thh}.

\begin{prop} \label{prop:shadow morita invariant}
 Shadows satisfy Morita invariance, in the sense of \cref{prop:morita}.
\end{prop}

\begin{proof}
 Let $\shadow{-}$ be a shadow functor on $\B$ taking values in $\D$. By \cref{cor:universal}, there exists a functor $F\colon \biHH(\B) \to \D$, such that $\shadow{-} = T \circ \shadow{-}_u$, where $\shadow{-}_u$ is the universal shadow. Now, by \cref{thm:chi functor}, for a given adjunction $\sigma = (C,D,F,G,u,c)$ in $\B$, we have $\chi(\sigma) = T \circ \widehat\chi(\sigma)$ and so the result follows from \cref{cor:chi functor on iso}.
\end{proof}

We now move on to the more general case and study traces of endomorphisms $\tr(u_Q)$, which requires a better understanding of the $2$-category with one adjunction and one endomorphism. 
\begin{defn} \label{def:adjend}
 Let $\AdjEnd$ be the free category with the same objects, $1$-morphisms and $2$-morphisms as $\Adj$, along with one additional free endomorphism $[0] \to [0]$. Define $\AdjEqEnd$ as its localization such that all 2-cells are invertible, with localization map $\AdjEnd \to \AdjEqEnd$.
\end{defn}

\begin{rmk} \label{rmk:morphism in adjbn} 
If we denote the free endomorphism by $q\colon0 \to 0$, an arbitrary endomorphism $0 \to 0$ in $\AdjEnd$ is a word on the two letters $q$ and $gf$, whereas a free endomorphism on $1$ is a word on $fg$ and $fqg$.
\end{rmk}

\begin{defn} \label{def:adjendb}
 For a bicategory $\B$, let 
 \[\AdjEnd(\B) = \Fun(\AdjEnd,\B) \text{ and } \AdjEqEnd(\B) = \Fun(\AdjEqEnd,\B).\]  
\end{defn}

As in the case of $\Adj$ and $\AdjEq$, there is an inclusion functor $\AdjEqEnd(\B) \to \AdjEnd(\B)$. Moreover, since there is a pushout square 
\begin{center}
    \begin{tikzcd}
      {[0]} \arrow[r] \arrow[d] & \Adj \arrow[d] \\
      B\bN \arrow[r] & \AdjEnd 
    \end{tikzcd}
\end{center}
and a similar one for $\AdjEqEnd$, there are equivalences
$$\AdjEnd(\B) \simeq \Adj(\B) \times_\B \End(\B),$$
$$\AdjEqEnd(\B) \simeq \AdjEq(\B) \times_\B \End(\B).$$

The explicit descriptions of $\biHH(\AdjEnd)$ (\cref{rem:thh bicat functorial}) and of endomorphisms in $\AdjEnd$ (\cref{rmk:morphism in adjbn}) together imply that $\biHH(\AdjEnd)$ includes $1$-morphisms $uqu\colon q \to gfqgf$ and $fqgc\colon fqgfg \to fqg$ and that there exists an isomorphism $gfqgf \cong fqgfg$.

\begin{defn} \label{def:universal trace characteristic}
  The \emph{universal trace} is the morphism 
  $$\widehat \tr_U \colon q \xrightarrow{uqu} gfqgf \xrightarrow{\cong} fqgfg \xrightarrow{fqgc} fqg$$ 
  in $\biHH(\AdjEnd)$. 
\end{defn}

The universal trace is indeed a trace in the sense of \cite{campbellponto2019thh}.

\begin{ex}
  If $\shadow{-}_U$ be the universal shadow functor from $\AdjEnd$ to $\biHH(\AdjEnd)$ (\cref{cor:universal}), then the trace of the data in $\AdjEnd$ is precisely the universal trace $\widehat \tr_U$. 
\end{ex}

Functoriality of  $\biHH$  also implies the following generalization of \cref{thm:chi functor}. 

\begin{thm} \label{thm:tr functor}
  Let $\B$ be a bicategory. The functor 
  {\smaller{\[ 
      \widehat\tr\colon \AdjEnd(\B)  = \Fun(\AdjEnd,\B) \xrightarrow{ \biHH } \Fun\big(\biHH(\AdjEnd),\biHH(\B)\big) \xrightarrow{ (\widehat \tr_U)^* } \Fun([1],\biHH\B) = \Arr\big(\biHH(\B)\big)
  \]}}   
 \noindent takes an adjunction $\sigma=(C,D,F,G,u,c)$ in $\B$ and endomorphism $Q$ on $\B$ to the trace of $(\sigma,Q)$ via the universal shadow $\shadow{-}_u$ on $\B$, 
 \[\widehat \tr(\sigma,Q) \colon \shadow{Q}_u \xrightarrow{\shadow{uQu}_u} \shadow{GFQGF}_u \xrightarrow{\cong} \shadow{FQGFG}_u \xrightarrow{\shadow{FQGc}_u} \shadow{FQG}_u\]
 in $\biHH(\B)$.

 In particular, for any shadow $\shadow{-}$ on $\B$ with values in $\D$, the trace $tr(\sigma,Q)$ is equal to $T \circ \widehat\tr(\sigma,Q)$, where $T\colon\biHH(\B) \to \D$ corresponds to $\shadow{-}$ via \cref{thm:main theorem}.
\end{thm}

\begin{proof}
  The result follows immediately from the functoriality of $\biHH$ and of $ (\widehat \tr_U)^*$.
\end{proof}

We thus have a result analogous to \cref{cor:chi functor on iso}.

\begin{cor} \label{cor:tr functor on iso}
 Precomposing the functor $ \widehat\tr\colon \AdjEnd(\B) \to \Arr\big(\biHH(\B)\big)$ with the restriction functor $\AdjEqEnd(\B) \to \AdjEnd(\B)$ given in \cref{def:adjendb} gives rise to a functor 
 $$\widehat\chi\colon \AdjEqEnd(\B) \to \Iso\big(\biHH(\B)\big).$$
\end{cor}

Finally, we can recover \cite[Proposition 4.8]{campbellponto2019thh}, via the same argument used in \cref{prop:shadow morita invariant}, based on \cref{thm:main theorem}.

\begin{prop} 
 Let $\shadow{-}$ be a shadow on a bicategory $\B$ taking values in a category $\D$. For a given Morita equivalence $\sigma=(C,D,F,G,u,c)$ and endomorphism $Q\colon C \to C$, the map $\tr(u_Q)\colon \shadow{Q} \to \shadow{FQG}$ is an isomorphism in $\D$ with inverse $\tr(c_Q)$.
\end{prop}

It is thus possible to recover the Euler characteristic, the trace, and their Morita invariance properties from a careful analysis of $\biHH(\Adj)$ and $\biHH(\AdjEnd)$ and the universal cases. As a next step we would also like to understand how canonical these constructions are. Unfortunately, doing so involves explicitly computing these categories, which is quite technical and hence relegated to \cref{sec:thh two cat}. 

Let $\Lambda_\infty$ denote the \emph{paracyclic category} (see \cref{defn:paracyclic} for further details). We prove in \cref{thm:thh adj} that $\biHH(\Adj)$ is equivalent to the category $(\Lambda_\infty)^{\lhd\rhd}$ and deduce the following canonical strengthening of \cref{thm:chi functor}.

\begin{cor} \label{cor:euler canonical}
  The universal Euler characteristic $\widehat{\chi}_U$ is the \emph{unique} morphism from the initial to the terminal object in $\biHH(\Adj) \simeq (\Lambda_\infty)^{\lhd\rhd}$. Hence, the definition of the Euler characteristic is canonical.
\end{cor}

We give an intricate, yet explicit, description of $\biHH(\AdjEnd)$ in \cref{thm:thh adjend}, in particular proving $\biHH(\AdjEnd)$ has set of objects $\{[n]k \colon [n] \in \Obj_{\DD_+}, k \in \bN \} \cong \bN \times \bN$, implying the following non-trivial observation regarding the trace functor constructed in \cref{thm:tr functor}. 

\begin{cor} \label{cor:trace noncanonical}
  The universal trace $\widehat{\tr}_U$ is the morphism $\emptyset 1 \to [0]1$ in $\biHH(\AdjEnd)$ (\cref{thm:thh adjend}). As $\Aut(\emptyset 1)= \bZ$, it follows that $\widehat{\tr}$ is not canonically determined and is fixed only up to a choice of automorphism.
\end{cor}

The non-canonicity of the universal trance plays no role in \cref{cor:tr functor on iso}, but it is not clear whether the choice of automorphism might influence other future results.

\section{Shadows on \texorpdfstring{$(\infty,2)$}{(oo,2)}-categories} \label{sec:traces of infty two categories}

In this last section we move from bicategories to the realm of $\V$-enriched $\infty$-categories, with a particular focus on $(\infty,2)$-categories, which are the natural $\infty$-categorical analogue of bicategories. First in \cref{subsec:thh of enriched categories} we review the work of Berman defining Hochschild homology for arbitrary enriched $\infty$-categories. Then, in \cref{subsec:shadows on infty two cats}, we define $(\infty,2)$-categorical shadows and construct a first example (\cref{thm:thh shadow}), conjecturing further generalizations of interest (\cref{rmk:challenges}). We rely on the existing $\infty$-categorical literature \cite{lurie2009htt,lurie2017ha}, but review the most relevant definitions and results in a brief introductory section (\cref{subsec:foundations infinity categories}).

\subsection{\texorpdfstring{$\infty$}{oo}-category theoretical background} \label{subsec:foundations infinity categories}
Before we commence, we fix some $\infty$-categorical conventions. We use the terminology \emph{$\infty$-category} as a synonym for \emph{quasi-category}, one important model of $(\infty,1)$-categories, popularized by Joyal and Lurie \cite{joyal2008theory,lurie2009htt}. We denote the large $\infty$-category of small $\infty$-categories by $\cat_\infty$. The nerve functor $N$ from the category $\cat$ of small categories to that of simplicial sets enables us to see $\cat $ as a subcategory of $\cat_\infty$. We routinely suppress the functor $N$ to simplify notation. Recall the following, useful property of $N$ \cite[Proposition 1.2.3.1,Remark 1.2.3.2]{lurie2009htt}

\begin{lemma} \label{lemma:homotopy category}
  The product preserving functor $\pi_0\colon \S \to \set$ induces a product preserving left adjoint to the inclusion of $\cat $ into $\cat_\infty$ 
  $$\Ho\colon \cat_\infty \to \cat$$
  that takes each $\infty$-category to its \emph{homotopy category}.
\end{lemma}

For an $\infty$-category $\C$, let $\P(\C)$ denote the $\infty$-category of space-valued presheaves on $\C$. Following \cite[Proposition 5.1.3.1]{lurie2009htt}, it comes with a fully faithful functor $j\colon\C \to \P(\C)$ that has the following universal property, stated in \cite[Theorem 5.1.5.6]{lurie2009htt}.

\begin{lemma} \label{lemma:lan}
  Let $\C,\D$ be $\infty$-categories with $\D$ cocomplete. Then the following diagram admits a (homotopically) unique colimit-preserving lift
  \[ 
    \begin{tikzcd}
      \C \arrow[r, "F"] \arrow[d, "j"] & \D \\
      \P(\C) \arrow[ur, dashed, "\hat{F}"]
    \end{tikzcd} 
  \] 
\end{lemma}

We also review some essential facts regarding monoidal $\infty$-categories, following \cite{lurie2017ha}. For the next definition, let $\Fin_*$ denote the category of finite pointed sets $<n>$ and pointed maps. Moreover, for any functor $\C \to \Fin_*$, we denote the fiber over $<n>$ by $\C_{<n>}$.

\begin{defn}[{\cite[Definition 2.0.0.7]{lurie2017ha}}]
  A \emph{symmetric monoidal $\infty$-category} is a Cartesian fibration $\C^\otimes \to \Fin_*$ satisfying the Segal condition, meaning $\C^\otimes_{<n>} \simeq (\C^\otimes_{<1>})^n$.
  A \emph{symmetric monoidal functor} is a morphism of Cartesian fibrations over $\Fin_*$.
\end{defn}

\begin{rmk} \label{rem:monoidal structure}
  For a given symmetric monoidal $\infty$-category $\C^{\otimes} \to \Fin_*$, the fiber of $<1>$ is called the \emph{underlying $\infty$-category}, which we simply denote by $\C$. Moreover, the map $<2> \to <1>$ in $\Fin_*$ induces an $\infty$-functor $- \otimes - \colon \C_{<1>} \times \C_{<1>} \simeq \C_{<2>} \to \C_{<1>}$, giving us the monoidal structure on $\C$. The remainder of the Cartesian fibration encodes the relevant coherence data. 
\end{rmk}

Throughout this section we abuse notation and denote symmetric monoidal $\infty$-categories simply by the underlying $\infty$-category $\C$, leaving the remaining data implicit.  

Finally, we need terminology regarding presentable $\infty$-categories. Recall that an $\infty$-category is \emph{presentable} if it has small colimits and is accessible \cite[Definition 5.5.0.1]{lurie2009htt}. Presentability and monoidal structure combine as follows.

\begin{defn} \label{def:presentably monoidal}
  A \emph{presentably symmetric monoidal $\infty$-category} is a symmetric monoidal $\infty$-category $\C$ such that the underlying $\infty$-category is presentable, and the induced map $- \otimes - \colon \C \times \C \to \C$ (\cref{rem:monoidal structure}) preserves colimits in both variables.
\end{defn}

\subsection{\texorpdfstring{$\THH$}{THH} of enriched \texorpdfstring{$\infty$}{oo}-categories} \label{subsec:thh of enriched categories}
In this subsection we review the definition of $\THH$ of an enriched $\infty$-category, relying primarily  on work of Berman \cite{berman2022thh}. Before moving on to the enriched setting, we explain why we cannot simply generalize \cref{Def:thh of spectral cat} to the $\infty$-categorical world. By definition, an $\infty$-category $\C$ is a simplicial set satisfying additional conditions, the so-called ``inner horn lifting conditions" \cite{lurie2009htt}. These lifting conditions imply that for any two objects (i.e., $0$-simplices) $x$ and $y$, the simplicial set $\Map(x,y)$ given by the fiber of the map $\C^{\Delta^1} \to \C \times \C$ over the point $(x,y)$ is a Kan complex and thus a homotopy-meaningful mapping space. There is no direct composition map, however. Instead, for three objects $x,y,z$, there is a zig-zag 
$$\Map(x,y) \times \Map(y,z) \xleftarrow{ \simeq } \Map(x,y,z) \to \Map(x,z)$$
where $\Map(x,y,z)$ is the fiber of the map $\C^{\Delta^2} \to \C^3$ over $(x,y,z)$, and the lifting conditions imply that the first map is a trivial Kan fibration.

This observation has two unfortunate implications.
\begin{enumerate}
    \item We cannot just define an enriched $\infty$-category in terms of the existence of composition maps and thus need a more complicated notion of enrichment.
    \item We cannot just define $\THH$ via the cyclic bar construction, as the simplicial structure relies on the existence of a direct composition map.
\end{enumerate}

Fortunately, Gepner-Haugseng \cite{gepnerhaugseng2015enriched} and Hinich \cite{hinich2020yoneda} have developed notions of $\infty$-categories enriched in a monoidal $\infty$-category, both of which are based on an operadic approach to classical enriched category theory, generalized to the $\infty$-categorical setting. Their constructions are very powerful and have been used to prove deep results about enriched $\infty$-categories, but are also very intricate and can be difficult to use for computations.

Since we need a notion of enrichment only to study $\THH$, we can focus on the case where the enriching category is not just monoidal, but actually {\it symmetric monoidal}. 
There is a much more convenient way of defining enriched $\infty$-categories in the symmetric case, due to Berman \cite{berman2022thh}\footnote{This approach can be in fact generalized to encompass the monoidal case, however, we do not require this generality.}, which also provides a natural framework for a generalization of $\THH$.

The key idea is that of a {\it bypass operation}, which we define next. Recall that a {\it directed multigraph} is a pair of sets $\Gamma = (S,E)$  of vertices and of directed edges between them. For two multigraphs $\Gamma_1, \Gamma_2$ with the same set of vertices $S$, there is a  {\it bypass operation }from $\Gamma_1$ to $\Gamma_2$ if $\Gamma_2$ can be obtained from $\Gamma_1$ by a sequence of the following combinatorial moves. 
\begin{enumerate}
    \item Adding a loop
    \item Replacing a path of consecutive edges $e_1, ... , e_n$ by a single edge $e$.
\end{enumerate}

For a precise definition, see \cite[Definition 2.1]{berman2022thh}.

\begin{defn}
  Let $S$ be any set.  The {\it bypass category on $S$}, $\Bypass_S$, has as objects all directed multigraphs with fixed vertex set $S$ and a finite set of edges, and as morphisms all bypass operations between such.
\end{defn}
 
 Note that $\Bypass_S$ admits a symmetric monoidal product $\otimes$, where the set of edges of $\Gamma_1 \otimes \Gamma_2$ is the disjoint union of the sets of edges of $\Gamma_1$ and $\Gamma_2$, and the unit is the multigraph with no edges.
 
\begin{defn}\label{def:enriched infty cat}
 Let $\V$ be a symmetric monoidal $\infty$-category. A {\it $\V$-enriched $\infty$-category with set of objects $S$} is  a symmetric monoidal functor of $\infty$-categories 
 $$\C\colon \Bypass_S \to \V.$$
\end{defn}

The bypass category is designed to encode the structure of an enriched category. For example, if $(X_0,...,X_n)$ denotes the multigraph with vertex set $S$ and precisely $n$ edges, one from $X_i$ to $X_{i+1}$ for every $0 \leq i < n$, then 
$$\C(X_0, ..., X_n) \in \V$$
can be viewed as the object in $\V$ of sequences of $n$ composable maps $X_0 \to ... \to X_n$. In particular $\C(X_0,X_1)$ plays the role of the mapping object. We can then use a bypass operation of type (2) above to obtain the desired composition map 
$$\C(X_0,X_1,X_2) \to \C(X_0,X_2).$$
The following result should thus not be very surprising.

\begin{prop}[{\cite[Proposition 2.7]{berman2022thh}}]
 If $\V$ is a symmetric monoidal $\infty$-category, then the definition of $\V$-enriched $\infty$-category given here agrees with that in \cite{gepnerhaugseng2015enriched}.
\end{prop}

\begin{rmk}
 Intuitively, Gepner and Haugseng define $\V$-enriched $\infty$-categories as certain algebras in $\V$ of a suitably chosen colored operad $\O_S^\otimes$ \cite{gepnerhaugseng2015enriched}. However, as the target $\infty$-category is symmetric monoidal, such algebras are equivalent to symmetric monoidal functors out of the symmetric monoidal envelope of $\O_S^\otimes$, which Berman proves to be $\Bypass_S$, permitting the simplification of the structure of an enriched $\infty$-category.  
\end{rmk}

We can use this simplified definition of an enriched $\infty$-category (and in particular the fact that $\Bypass_S$ is a category) to define $\THH $ of a $\V$-enriched $\infty$-category as follows. Let $\Lambda$ be the category with finite cyclically ordered sets $0<1<...<n<0$, for $n\geq 0$, as objects  and maps that respect the cyclic structure as morphisms. (For more details see \cite[Definition 3.3]{berman2022thh}.) 

Let 
$$(\O_{hh_\V})_\bullet\colon \Lambda \to \P(\Bypass_S)$$ be the functor that takes a cyclic set $0<1<...<n<0$ to the coproduct of representables 
$$\coprod_{X_0,... ,X_n \in S}(X_0,X_1,...,X_n,X_0),$$
and let 
$$\O_{hh_\V}=\colim_{\DD^{op}} (\O_{hh_\V}|_{\DD^{op}})_\bullet \in \P(\Bypass_S).$$ 
We suppress the choice of the set $S$ of objects from the notation $(\O_{hh_\V})_\bullet$ and $\O_{hh_\V}$.

As Berman explains, \cref{lemma:lan} implies that any $\V$-enriched $\infty$-category $\C\colon \Bypass_S \to \V$ extends to a colimit-preserving functor 
\[
 \begin{tikzcd}
  \Bypass_S \arrow[r, "\C"] \arrow[d, "j"] & \V, \\
  \P(\Bypass_S) \arrow[ur, "\C_*"']
 \end{tikzcd}
\]
which allows us to formulate the following generalization of $\THH$, providing a way of associating an object in $\V$ to any $\V$-enriched $\infty$-category. 

\begin{defn}[{\cite[Definition 4.1]{berman2022thh}}]
 \label{def:thh of enriched cat}
 Let $\C$ be a $\V$-enriched $\infty$-category. The {\it $\V$-enriched Hochschild homology} of $\C$ is 
 $$\HH_\V(\C) = \C_*(\O_{hh_\V}) \in \mathrm{Obj} \V.$$
\end{defn}

\begin{rmk}
 The construction $\HH_\V$ is simply denoted by $\THH$ in \cite{berman2022thh}, however, as we analyze Hochschild homology with several different enrichments, we have chosen notation that makes the enrichment explicit.
\end{rmk}
Tracing through the definitions, since $\C_*$ preserves colimits, we see that $\HH_\V(\C)$ is the colimit of a cyclic diagram that at level $n$ is equivalent to 

$$\coprod_{X_0,...,X_n \in S} \C(X_0,...,X_n,X_0) \simeq \coprod_{X_0,..., X_n} \C(X_0,X_1) \otimes ... \otimes \C(X_n,X_0)$$
because $\C$ is symmetric monoidal. 

\begin{ex}
 If $\V = (\Sp, \wedge)$, the symmetric monoidal $\infty$-category of spectra with the smash product, then a $\V$-enriched $\infty$-category $\C$ is a spectral $\infty$-category, and  $\HH_{\Sp}(\C)$ is the $\infty$-categorical analogue of \cref{Def:thh of spectral cat}, and we therefore denote it simply by $\THH(\C)$.
\end{ex}

\begin{ex}
 If $\V=(\S,\times)$, the symmetric monoidal $\infty$-category of spaces, then $\V$-enriched $\infty$-categories are precisely 
 non-enriched $\infty$-categories. More concretely, we get complete Segal spaces \cite{rezk2001css}, as follows from \cite[Theorem 4.4.7, Remark 5.3.10]{gepnerhaugseng2015enriched}. We hence denote $\HH_\S$ by $\HHinfty$.
\end{ex}

\begin{ex} \label{ex:two cats}
 If $\V= (\cat, \times)$, then a $\V$-enriched $\infty$-category $\C$ a precisely a \emph{bicategory} \cite[Definition 6.1.1]{gepnerhaugseng2015enriched}, \cite[Proposition 2.7]{berman2022thh}.  By construction, if $\C$ is an bicategory, then  $\HH_{\cat}(\C)$ is a category.
\end{ex}
 
When applied to bicategories, the definition of $\HH_{\cat}$ does indeed coincide with that of $\biHH$ (\cref{def:biHH}), as the following proposition makes explicit.
 
\begin{prop}\label{prop:bicat}
	For every bicategory $\B$, there is an equivalence $\HH_{\cat}(\B) \simeq \biHH(\B)$, meaning $\HH_{\cat}$ is also the colimit of \cref{eq:thh bicat} in the $(2,1)$-category $\cat$.
\end{prop}
		
		\begin{proof}
	Since $\cat$ is a $(2,1)$-category, for any simplicial diagram $F\colon \DD^{op} \to \cat$, there is an equivalence 
	$$\colim ( \DD^{op} \xrightarrow{ F } \cat) \simeq \colim ( (\DD_{\leq 2})^{op} \xrightarrow{ i^{op} } \DD^{op} \xrightarrow{ F } \cat)$$ 
		where $i\colon \DD_{\leq 2} \to \DD$ is the natural inclusion map of the full subcategory with objects $[0],[1]$ and $[2]$ (see \cite[Lemma 1.3.3.10]{lurie2017ha}).  
	\end{proof}

We can leverage this abstract description of $\biHH$ and formal homotopical methods to compute $\biHH$ of bigroupoids. Doing so requires careful understanding of how bigroupoids relate to homotopy types, which is the content of the next remark. In this remark $\Grpd$ denotes the $(2,1)$-category of groupoids and $\Grpd_2$ denotes the $(3,1)$-category of bigroupoids. 

\begin{rmk}
 By \cite[Corollary 6.1.10]{gepnerhaugseng2015enriched}, there is an adjunction of $\infty$-categories
		\begin{equation} \label{eq:truncated spaces}
						\begin{tikzcd}[row sep=0.5in, column sep=0.5in]
						\S \arrow[r, shift left=1.8, "\Pi_{\leq 1}", "\bot"'] & \Grpd, \arrow[l, shift left=1.8, "N"] 
						\end{tikzcd}
			\end{equation}
			where $\Grpd$ denotes the $(2,1)$-category of groupoids, of which the right adjoint is fully faithful, with essential image given by $1$-truncated spaces (spaces with trivial homotopy groups above degree $1$). The right adjoint takes a groupoid $\G$ to a Kan complex of which the $0$-cells are the objects of $\G$, while for any two objects $X,Y$, there is an isomorphism of sets $\Map_{N\G}(X,Y) \cong \G(X,Y)$, where $\Map_{N\G}(X,Y)$ denotes the space of paths from $X$ to $Y$ (which in this case happens to be discrete as $N\G$ is $1$-truncated).
				
				The same corollary (\cite[Corollary 6.1.10]{gepnerhaugseng2015enriched}) implies that there is an adjunction
		\begin{equation} \label{eq:twotruncated spaces}
						\begin{tikzcd}[row sep=0.5in, column sep=0.5in]
						\S \arrow[r, shift left=1.8, "\Pi_{\leq 2}", "\bot"'] & \Grpd_2 \arrow[l, shift left=1.8, "N_2"] 
						\end{tikzcd}
			\end{equation}
	of which the right adjoint is fully faithful, with essential image given by $2$-truncated spaces. By \cite[Lemma 6.1.9]{gepnerhaugseng2015enriched} the functor $N_2$ takes a bigroupoid $\G$ to a Kan complex $N_2\G$, of which the $0$-cells are the objects of $\G$, while for any two objects $X,Y$, there is an equivalence of groupoids 
		$$\Pi_1\Map_{N_2\G}(X,Y) \simeq \G(X,Y).$$
	\end{rmk}
	
	We can use the adjunction above to compute $\biHH$ of a bigroupoid.
	
	\begin{ex} \label{ex:thh groupoids}
	For any bigroupoid $\G$, there is an equivalence 
		$$\biHH(\G) = |\coprod_{X_0,...,X_n} \G(X_0,X_1,...,X_0)| \simeq  |\coprod_{X_0,...,X_n} \Pi_1\Map_{N_2\G}(X_0,X_1,...,X_0)|$$
		$$ \simeq \Pi_1 |\coprod_{X_0,...,X_n} \Map_{N_2\G}(X_0,X_1,...,X_0)|$$
		where $\Map_{N_2\G}(X_0,X_1,...,X_0)$ denotes the space of loops through the vertices $X_0,X_1,...,X_n$ and the equivalence follows from the previous remark and the fact that $\Pi_1$ commutes with colimits. The geometric realization of this diagram of spaces is known as {\it unstable topological Hochschild homology} and has been computed to be the free loop space $(N_2\G)^{S^1}$ \cite[Corollary 4, Page 858]{hesselholtscholze2019thhoberwolfbach}. Since $N_2$ is fully faithful and $N_2 B\mathbb{Z} = S^1$, it follows that 
		$$\biHH(\G) \simeq (N_2\G)^{S^1} = (N_2\G)^{N_2B\mathbb{Z}} \simeq \Fun(\B\mathbb{Z},\G),$$ where the last equivalence follows from the fact that $N_2$ is fully faithful.
	\end{ex}

\begin{ex} \label{ex:infty two cats}
 If $\V= (\cat_\infty, \times)$, then we call a $\V$-enriched $\infty$-category $\C$ an {\it $(\infty,2)$-category}.  Note that if $\C$ is an $(\infty,2)$-category, then  $\HH_{\cat_\infty}(\C)$ is an $\infty$-category, and we hence use the notation $\biHHinfty(\C)$.
\end{ex}

\begin{rmk}
 There are many models of $(\infty,2)$-categories in the literature, such as $2$-complicial sets \cite{verity2008complicial}, $2$-fold complete Segal spaces \cite{barwick2005nfoldsegalspaces}, and $\Theta_2$-spaces \cite{rezk2010thetanspaces}. 
 The model of $(\infty,2)$-categories constructed in \cref{ex:infty two cats} corresponds to $2$-fold complete Segal spaces, as follows from the proof in \cite[Theorem 4.4.7, Remark 5.3.10]{gepnerhaugseng2015enriched}, with the $\infty$-category of complete Segal spaces $\CSS$ replacing the $\infty$-category $\S$.
\end{rmk}

\begin{rmk} \label{rem:trace}
 Let $\C$ be a $\V$-enriched $\infty$-category with set of objects $S$, and let $\tau\colon \HH_\V(\C)\to V$ be a morphism in $\V$. From $\tau$ we can derive a family of morphisms in $\V$
 $$\{\tau_{X}\colon \C(X, X) \to V\mid X\in S\}.$$
The cyclic structure of $(\O_{hh_\V})_\bullet$ implies that
$$\xymatrix{\C(X_0,X_1)\otimes \C(X_1, X_0)\simeq \C(X_0,X_1,X_0)\ar[r] \ar @<65pt> [dd]& \C(X_0,X_0)\ar [d]^{\tau_{X_0}}\\
&V\\
\C(X_1,X_0)\otimes \C(X_0, X_1)\simeq \C(X_1,X_0,X_1)\ar [r] &\C(X_1,X_1)\ar[u]_{\tau_{X_1}}
}$$
commutes for all $X_0,X_1\in S$, where the horizontal ``composition'' maps arise from bypasses, and the left hand vertical arrow is given by the cyclic structure.  In other words, the value of $\tau$ on a composite of two 1-cells that can be composed in either order is independent of the order of composition, just as the trace of a product of two matrices that can be multiplied in either order is independent of the order of multiplication.
\end{rmk}

This observation justifies the terminology in the next definition. 

\begin{defn}\label{defn:trace}
 Let $\C$ be a $\V$-enriched $\infty$-category, and let  $V$ be an object in $\V$.  A {\it $\V$-trace} out of $\C$ into $V$ is  a morphism $\HH_\V(\C) \to V$ in $\V$.
\end{defn}

\subsection{Shadows on \texorpdfstring{$(\infty,2)$}{(oo,2)}-categories} \label{subsec:shadows on infty two cats}
We are finally in a position to generalize shadows to $(\infty,2)$-categories. To motivate our definition, we first analyze the nature of shadows on the homotopy bicategory of a $(\infty,2)$-category with values in the homotopy category of an $(\infty, 1)$-category. 

Let $\B$ be an $(\infty,2)$-category specified by a symmetric monoidal functor
$$\B \colon \Bypass_{\B_0}\to \cat_\infty.$$
Its \emph{homotopy bicategory} $\Ho_2(\B)$ is the bicategory specified by the composite
$$\Bypass_{\B_0}\xrightarrow{\B} \cat_\infty\xrightarrow{\Ho}\cat.$$
This composite does indeed give rise to a bicategory (i.e., enriched over $\cat$ in the sense of \cref{def:enriched infty cat}), since $\Ho$ preserves finite products and thus is symmetric monoidal (\cref{lemma:homotopy category}).

\begin{prop} \label{prop:shadow on infty bicat}
Let $\B$ be an $(\infty,2)$-category and $\D$ an $(\infty,1)$-category. There is an equivalence 
$$\Sha\big(\Ho_2(\B),\Ho(\D)\big)  \simeq \Fun_\infty\big(\biHHinfty(\B),\Ho(\D)\big),$$
where on the right-hand side we implicitly view $\Ho(\D)$ as an $(\infty, 1)$-category via the nerve functor.
\end{prop}

\begin{proof}
There are equivalences 
 $$\Sha\big(\Ho_2(\B),\Ho(\D)\big) \simeq \Fun\big(\biHH(\Ho_2(\B)),\Ho(\D)\big) \simeq \Fun_\infty\big(\biHHinfty(\B),\Ho(\D)\big).$$
 The first equivalence is an immediate consequence of \cref{thm:main theorem}.
 The second follows from the fact that $\Ho$ commutes with colimits as a left adjoint and so 
 $$\biHH(\Ho_2(\B)) \simeq \Ho(\biHHinfty(\B)),$$
together with the fact that $\Ho$ is left adjoint to the nerve functor (\cref{lemma:homotopy category}).
\end{proof}

\begin{notn} Given a shadow $\shadow{-}$ on  $\Ho_2(\B)$ taking values in $\Ho(\D)$, we denote the corresponding $\infty$-functor from $\biHHinfty(\B)$ to $\Ho(\D)$ by $\widehat{\shadow{-}}$. 
\end{notn}

In particular, there is a localization functor 
$$\Fun_\infty(\biHHinfty(\B), \D) \to \Sha\big(\Ho_2(\B),\Ho(\D)\big)$$
that takes a trace of $(\infty,2)$-categories to a shadow.

Motivated by \cref{thm:main theorem} and the analysis above, we formulate the following definition. 

\begin{defn} \label{def:shadow infty}
  A \emph{shadow} on an $(\infty,2)$-category $\B$ with values in an $(\infty,1)$-category $\D$ is a $\cat_\infty$-trace (\cref{defn:trace}), i.e.,  an $\infty$-functor 
  \[\biHHinfty(\B) \to \D,\]
  where $\biHHinfty(\B)$ is constructed as in \cref{def:thh of enriched cat}.  
\end{defn}
This definition can be viewed as a homotopy coherent lift of the definition of a shadow on a bicategory. It is natural to ask when $(\infty,2)$-categorical shadows exist and how to construct them. In the important special case of $\V$-enriched categories and their bimodules, we conjecture that there should be an analogue of the Hochschild shadow for spectral categories, building on the following result by Haugseng.

\begin{thm}[{\cite[Theorem 1.2]{haugseng2016bimodules}}] \label{thm:vbimodules}
  Let $\V$ be a presentably symmetric monoidal $\infty$-category (\cref{def:presentably monoidal}). There exists an $(\infty,2)$-category $\Mod_\V$ with
  \begin{itemize}
      \item objects $\V$-enriched $\infty$-categories,
      \item $1$-morphisms given by $\V$-bimodules,
      \item $2$-morphisms given by morphisms of $\V$-bimodules.
  \end{itemize}
\end{thm}

 We conjecture that the functors $\HH_\V$ assemble into a shadow of $(\infty,2)$-categories on $\Mod_\V$ with values in $\V$. See \cref{conj:main} for a precise formulation.

While we cannot yet prove this conjecture, we take an important first step below and construct a functor to $\Ho \V$ (\cref{thm:thh shadow}), then discuss some technical challenges that arise when trying to lift to an $(\infty, 2)$-categorical shadow (\cref{rmk:challenges}). 

\begin{rmk} \label{rem:functoriality thh v}
    As we have not established functoriality of $\biHH_\infty$, we cannot yet conclude that an $(\infty,2)$-categorical shadow (\cref{def:shadow infty}) satisfies Morita invariance, along the lines of \cref{thm:chi functor}. However, once we have established functoriality, Morita invariance will follow immediately by an argument analogous to the one given above in the case of bicategories. In fact, this generalized argument would even imply Morita invariance of $\HH_\V$ for $\V$ an arbitrary $\infty$-category.  
\end{rmk}

We begin by relativizing Berman's approach to $\V$-enriched $\infty$-categories \cite{berman2022thh} in order to describe $\V$-enriched bimodules.

\begin{defn}
Let $S,T$ be two sets. A {\it directed graph from $S$ to $T$} is a directed graph with vertex set $S \coprod T$, set of edges $E$, and source and target maps $s,t\colon E \to S \coprod T$ such that if $t(e) \in S$, then $s(e) \in S$, i.e., there are no edges starting in $T$ and ending in $S$.
\end{defn}

\begin{defn}
Let $\Bypass_{S,T}$ be the full subcategory of $\Bypass_{S \coprod T}$ with objects the directed graphs from $S$ to $T$.
\end{defn}

This full subcategory is still symmetric monoidal, as graphs from $S$ to $T$ are closed under disjoint union, and the empty graph on $S \coprod T$ is in particular a graph from $S$ to $T$. 

The following notation for objects in $\Bypass_{S,T}$ proves useful below.

\begin{notn} \label{notn:bypassst}
 Let $X,Y \in S \coprod T$, where if $Y \in S$, then $X \in S$. There exists a graph with a unique edge from $X$ to $Y$, which we denote $(X,Y)$. More generally, let $(X_0,X_1, ... , X_n)$ be the graph with a single path $X_0 \to X_1 \to ... \to X_n$, where if $X_i \in S$, then $X_j \in S$ for all $j <i$.
\end{notn}

As $\Bypass_{S,T}$ is defined as a full symmetric monoidal subcategory of $\Bypass_{S\coprod T}$, the explanation following \cite[Definition 2.3]{berman2022thh} provides an explicit description of $\Bypass_{S,T}$ via generators and relations as follows.
\begin{itemize}
    \item Objects are pairs $(X,Y)$ of elements of $S\coprod T$, where if $Y \in S$, then $X \in S$.
    \item Morphisms are generated by $(X,Y,Z) = (X,Y) \otimes (Y,Z) \to (X,Z)$, which exists for every triple $X,Y,Z \in S \coprod T$, where $Z \in S$ implies $X,Y \in S$, and $Y \in S$ implies $X \in S$.
    \item For every $X \in S \coprod T$, there is an identity morphism $\emptyset \to (X,X)$
    \item The associativity and unitality relations of \cite[Definition 2.3]{berman2022thh}.
\end{itemize}

There is an evident inclusion functor $\Bypass_S \coprod \Bypass_T \to \Bypass_{S,T}$. We can use this to define our desired modules.

\begin{defn} \label{defn:bimodules}
 Let $\C\colon \Bypass_S \to \V$ and $\D\colon \Bypass_T \to \V$ be two $\V$-enriched $\infty$-categories. A {\it $\V$-enriched $(\C,\D)$-bimodule} is a symmetric monoidal functor $\M\colon \Bypass_{S,T} \to \V$ that makes the following diagram commute.
 \begin{center}
     \begin{tikzcd}[column sep=1in]
       \Bypass_S \arrow[dr, "\C"] \arrow[d, hookrightarrow] & \\
       \Bypass_{S,T} \arrow[d, hookleftarrow] \arrow[r, dashed, "\M" description] & \V \\
       \Bypass_T  \arrow[ur, "\D"']
     \end{tikzcd}
 \end{center}
\end{defn}

It is helpful to unpack this definition somewhat. The objects in $\Bypass_{S,T}$ can be classified into three distinct types: graphs whose edges all start and end at elements of $S$, graphs whose edges all start and end at elements of $T$, and graphs that have an edge that starts in $S$ and ends in $T$. The first two types of graph lie in the essential image of the inclusion functor $\Bypass_S \coprod \Bypass_T \to \Bypass_{S,T}$, and so the value of $\M$ on those graphs is predetermined. In particular for every $X_0,..., X_n \in S$,
$$\M(X_0,...,X_n) = \C(X_0,...,X_n),$$ 
and for every $X_0,...,X_n \in T$,
$$\M(X_0,...,X_n) = \D(X_0,...,X_n).$$ 
On the other hand if $X_0 \in S$ and $X_1 \in T$, then there are no constraints on the object $\M(X_0,X_1)$ in $\V$.

In $\Bypass_{S,T}$, for every $X_0, ..., X_n \in S$ and $X_{n+1}\in T$, there a bypass morphism $(X_0,..., X_{n+1}) \to (X_0,X_{n+1})$, the image of which under $\M$ is a map 
$$\C(X_0,X_1) \otimes ... \otimes\C(X_{n-1},X_n) \otimes \M(X_n,X_{n+1}) \to \M(X_0,X_{n+1}).$$ 
Similarly for $X_0 \in S$ and $X_1, ... , X_{n+1} \in T$, there is a map
$$\M(X_0,X_1)\otimes \D(X_0,X_{1}) \otimes ... \otimes \D(X_n,X_{n+1}) \to \M(X_0,X_{n+1}).$$ 
This is precisely the expected structural data of a bimodule. The composition rules in $\Bypass_{S,T}$ guarantee that these bimodule actions satisfy the appropriate coherence conditions.

Having defined bimodules, we also need a suitable notion of morphism. 

\begin{defn} \label{defn:bimodule morphism}
  Let $\C\colon \Bypass_S \to \V$ and $\D\colon \Bypass_T \to \V$ be $\V$-enriched $\infty$-categories, and let $\M$ and $\N$ be $(\C,\D)$-bimodules. A {\it morphism of bimodules} is a natural transformation $\alpha\colon \M \to \N$ such that $\alpha$ restricts to the identity on $\Bypass_S$ and $\Bypass_T$.
 \end{defn}

Before we proceed, it is important to confirm that these two definitions given here match with the existing literature. Bimodules of $\V$-enriched $\infty$-categories have been studied extensively by Haugseng \cite{haugseng2016bimodules}. We show that the definition above does indeed coincide with that of Haugseng, by a slight generalization of the argument given by Berman in \cite[Proposition 2.7]{berman2022thh}.

\begin{prop}
 The notions of bimodules formulated in \cref{defn:bimodules} and of morphism of bimodules \cref{defn:bimodule morphism} coincide with those of Haugseng \cite[Definition 4.3]{haugseng2016bimodules}.
\end{prop}

\begin{proof}
We first establish that \cref{defn:bimodules} coincides with that of Haugseng, by following the outline of the proof of \cite[Proposition 2.7]{berman2022thh}. First recall that a bimodule is an algebra on the non-symmetric $\infty$-operad $\DD^{op}_{S,T}$ (described explicitly in \cite[Definition 4.1]{haugseng2016bimodules}). Because $\V$ is symmetric monoidal, we can use the symmetric monoidal envelope of $\DD^{op}_{S,T}$, given by the active morphisms \cite[Construction 2.2.4.1]{lurie2017ha}. 
 
As in the proof of \cite[Proposition 2.7]{berman2022thh}, the active morphisms with codomain a pair $(A,B)$ of elements of $S\coprod T$ are of the form
$$(X_1,Y_1) \otimes ... \otimes (X_n,Y_n) \to (A,B).$$
It follows that the symmetric monoidal envelope of $\DD_{S,T}^{op}$ is given by $\Bypass_{S,T}$ and so the first claim follows from the universal property of symmetric monoidal envelopes \cite[Proposition 2.2.4.9]{lurie2017ha}.

We next analyze \cref{defn:bimodule morphism}. By \cite[Definition 4.3]{haugseng2016bimodules}, a morphism of bimodules is precisely a morphism of $\DD^{op}_{S,T}$-algebras, which, by \cite[Proposition 2.2.4.9]{lurie2017ha}, corresponds to a symmetric monoidal natural transformation out of the symmetric monoidal envelope of $\DD_{S,T}^{op}$, i.e., natural transformations out of $\Bypass_{S,T}$. This is precisely the content of \cref{defn:bimodule morphism} and hence establishes the second claim.
\end{proof}

We are now almost ready to define $\THH$ for $\infty$-categorical bimodules but need one last definition.

\begin{defn} \label{defn:bypasscircle}
 For $S,T$ two sets, define the symmetric monoidal category $\Bypass_{S\circlearrowleft T}$ as the pushout of symmetric monoidal categories
 \begin{center}
     \begin{tikzcd}
       \Bypass_S \coprod \Bypass_T \arrow[r] \arrow[d] & \Bypass_{S,T} \arrow[d] \\
       \Bypass_{T,S} \arrow[r] & \Bypass_{S\circlearrowleft T}.
     \end{tikzcd}
 \end{center}
\end{defn}

Let $\C\colon\Bypass_S \to \V$ be a $\V$-enriched $\infty$-category and $\M\colon \Bypass_{S,S} \to \V$ a $(\C,\C)$-bimodule. To distinguish between the two copies of $S$, we denote an element $X \in S$ by $X'$ if it is in the second copy. By the universal property of pushouts, $\M$ lifts to a functors $\Bypass_{S\circlearrowleft S} \to \V$, which we also denote by $\M$, to simplify notation.

Let $\M^*\colon \P(\Bypass_{S\circlearrowleft S}) \to \V$ be the left Kan extension of $\M$. Let $(\O_{hh_\V})_\bullet\colon\DD^{op} \to \P(\Bypass_{S\circlearrowleft S})$ be the functor defined in level $n$ by 
$$(\O_{hh_\V})_n = \coprod_{X_0,..., ,X_n \in S} (X_0,X_1,...,X_n,X_0'),$$
where we are using \cref{notn:bypassst}.

\begin{defn}\label{def:thh bimod}
For any $\V$-enriched $\infty$-category $\C$ and $(\C,\C)$-bimodule $\M$, the {\it $\V$-enriched Hochschild homology of $\C$ with coefficients in $\M$}, denoted $\HH_\V(\C,\M)$, is the object in $\V$ that is the colimit of $\M^*\circ(\O_{hh_\V})_\bullet$.
\end{defn}

Since $\M^*$ preserves colimits, $\HH_\V(\C,\M)$ is equivalent to the colimit of the simplicial diagram with $n^{\text{th}}$ level 
$$\coprod_{X_0,..., X_n} \C(X_0,X_1) \otimes ... \otimes \C(X_{n-1},X_n) \otimes \M(X_n,X_0).$$

The construction of $\HH_\V(\C,\M)$ is natural in the coefficient bimodule, i.e., it extends to an $\infty$-functor $$\HH_\V(\C,-)\colon\Mod_\V(\C,\C) \to \V$$ for every $\V$-enriched category $\C$. Indeed, for any $\C$-bimodule morphism $\alpha\colon \M \to \N$, we can define $\HH_\V(\C,\alpha)$ by 
 $$\HH_\V(\C,\alpha)=\colim_{\DD^{op}}\alpha^* \circ \O_{hh_\V}\colon \colim_{\DD^{op}} \M^* \circ \O_{hh_\V} \to \colim_{\DD^{op}}\N^* \circ \O_{hh_\V}. $$
 In other words, $\HH_\V(\C,-)$ is given by  the following composite of $\infty$-functors.
$$
   \Fun(\Bypass_{S\circlearrowleft S},\V) \xrightarrow{\textrm{Lan}} \Fun\big(\P(\Bypass_{S\circlearrowleft S}),\V\big)  \xrightarrow{(\O_{hh_\V})^*} \Fun(\DD^{op},\V)  \xrightarrow{\colim} \V  
$$

We are now ready to generalize \cref{rmk:thh trace} and show that the collection of induced functors on the homotopy categories 
\[\big\{\HH_\V(\C,-)\colon \Ho_2(\Mod_\V)(\C,\C) \to \Ho\V\mid \C \text{ a $\V$-enriched $\infty$-category}\big\}\]
underlies a shadow, for every presentably symmetric monoidal $\V$.

\begin{thm} \label{thm:thh shadow} 
 Let $\V$ be a presentably symmetric monoidal $(\infty,1)$-category and $\Mod_\V$ the $(\infty,2)$-category of $\V$-bimodules (\cref{thm:vbimodules}). There is a functor of $\infty$-categories
 \[\widehat{\HH}_\V\colon \biHHinfty(\Mod_\V) \to \Ho\V\]
 such that the corresponding shadow on $\Ho_2(\Mod_\V)$ (cf. \cref{prop:shadow on infty bicat}) sends a $\C$-bimodule $\M$ to $\HH_\V(\C,\M)$ (\cref{def:thh bimod}).
\end{thm}

\begin{proof} 
We need only to define the twisting isomorphism for the collection of functors $$\big\{\HH_\V(\C,-)\colon \Ho_2(\Mod_\V)(\C,\C) \to \Ho\V\mid \C \text{ a $\V$-enriched $\infty$-category}\big\}.$$ To do so, we adapt the Dennis-Morita-Waldhausen argument \cite[Proposition 6.2]{blumbergmandell2012thh}, which is also used by Campbell and Ponto \cite[Theorem 2.17]{campbellponto2019thh}, to the $\infty$-categorical setting. 

The first step towards defining the twisting isomorphism consists of the following bisimplicial construction.
 Let $(\O^{S\circlearrowleft T}_{thh})_{\bullet\bullet}\colon \DD^{op} \times \DD^{op} \to \P(\Bypass_{S\circlearrowleft T})$ be the functor that takes $([n],[m])$ to 
  $$\coprod_{X_0,...,X_n \in S, Y_0, ... , Y_m \in T} (X_0,..., X_n,Y_0,..., Y_m,X_0),$$
  and let $\O^{S\circlearrowleft T}_{thh}$ be the colimit of this bisimplicial diagram in $\P(\Bypass_{S\circlearrowleft T})$.
  Similarly, define  $(\O^{T\circlearrowleft S}_{thh})_{\bullet\bullet}\colon \DD^{op} \times \DD^{op} \to \P(\Bypass_{S\circlearrowleft T})$ as the functor with value
  $$\coprod_{X_0,...,X_n \in S, Y_0, ... , Y_m \in T} (Y_0,..., Y_m,X_0,..., X_n,Y_0),$$
 on $([n],[m])$ and $\O^{T\circlearrowleft S}_{thh}$ as its colimit in $\P(\Bypass_{S\circlearrowleft T})$.
  
The symmetric monoidal structure of the bypass categories provides us with a canonical isomorphism
   $$(X_0,..., X_n,Y_0,..., Y_m,X_0) \overset{\cong}{\longrightarrow} (Y_0,..., Y_m,X_0,..., X_n,Y_0),$$
   which implies that there is an equivalence
   \begin{equation}\label{eq:last one}
     \O^{S\circlearrowleft T}_{thh} \overset{\cong}{\Longrightarrow} \O^{T\circlearrowleft S}_{thh}.  
   \end{equation}
   in $\P(\Bypass_{S\circlearrowleft T})$.
  
 We are now ready to define the desired twisting isomorphism. Let 
 $$\C\colon\Bypass_S \to \V \quad\text{and}\quad \D\colon\Bypass_T \to \V$$ be two $\V$-enriched $\infty$-categories, and let $\M\colon \Bypass_{S,T} \to \V$ be a $(\C,\D)$-bimodule and $\N\colon \Bypass_{T,S} \to \V$ a $(\D,\C)$-bimodule. 
The universal property of the pushout implies that these data give rise to a functor 
$$\M \otimes \N\colon \Bypass_{S\circlearrowleft T} \to \V,$$ 
which we can left Kan extend to a functor 
$$(\M\otimes\N)^*\colon \P(\Bypass_{S\circlearrowleft T}^{op}) \to \V.$$ 
Note that $(\M\otimes\N)^*(\O^{S\circlearrowleft T}_{thh})$ is the colimit of the bisimplicial object 
 $$(\M\otimes\N)^* \circ (\O^{S\circlearrowleft T}_{thh})_{\bullet\bullet}\colon \DD^{op} \times \DD^{op} \to \V$$
that takes $([n],[m])$ to 
  $$\coprod_{X_0,...,X_n \in S, Y_0, ... , Y_m \in T} \C(X_0,X_1) \otimes ... \otimes \C(X_{n-1},X_n) \otimes \M(X_n,Y_0) \otimes \D(Y_0,Y_1) \otimes ... \otimes \D(Y_{m-1},Y_m)\otimes \N(Y_m,X_0).$$
  
Fix an object $[n]$ in $\DD$. By \cite[End of Section 5]{haugseng2016bimodules}, the colimit of the simplicial diagram with $n^{\text{th}}$ level
 \begin{align*}
   &(\M\otimes\N)^* \circ (\O^{S\circlearrowleft T}_{thh})_{n\bullet}=  \\
   &\coprod_{X_0,...,X_n \in S, Y_0, ... , Y_\bullet \in T} \C(X_0,X_1) \otimes ... \otimes \C(X_{n-1},X_n) \otimes \M(X_n,Y_0) \otimes \D(Y_0,Y_1) \otimes ... \otimes \N(Y_\bullet,X_0)
 \end{align*}
 is  
 $$\coprod_{X_0,...,X_n \in S} \C(X_0,X_1) \otimes ... \otimes \C(X_{n-1},X_n) \otimes (\M \otimes \N)(X_n,X_0),$$
 where $\M\otimes\N$ is the $(\C,\C)$-bimodule obtained via tensor product of $\M$ and $\N$ as defined in \cite[Remark 5.4]{haugseng2016bimodules}.
 By \cref{def:thh bimod} the colimit of this simplicial object is precisely $\HH_\V(\C,\M \otimes \N)$.
 Repeating the same argument with the roles of $\M$ and $\N$ reversed, we deduce that $\HH_\V(\D,\N\otimes \M)$ can be obtained as the colimit of the bisimplicial diagram of the form  
 $$\coprod_{X_0,...,X_n \in S, Y_0, ... , Y_m \in T} \D(Y_0,Y_1) \otimes ... \otimes \D(Y_{n-1},Y_m) \otimes \N(Y_m,X_0) \otimes \C(X_0,X_1) \otimes ... \otimes \M(X_n,Y_0),$$
i.e.,  $(\M\otimes\N)^*(\O^{T\circlearrowleft S}_{thh})$. The equivalence in \cref{eq:last one} thus implies that
 $\HH_\V(\D,\N\otimes \M)$ and $\HH_\V(\C, \M\otimes\N)$ are equivalent in $\V$ and therefore isomorphic in $\Ho \V$, as desired.
 
 Finally, as mentioned in the proof of \cite[Theorem 2.17]{campbellponto2019thh}, it is straightforward to check that this isomorphism satisfies the two compatibility conditions of \cref{defn:shadow}, and hence we can conclude. 
\end{proof}

We can apply \cref{{thm:thh shadow}} to the case of spectrally enriched $\infty$-categories  and thus, in particular, to stable $\infty$-categories, since every stable $\infty$-category is in fact enriched over $\Sp$ \cite[Example 7.4.14]{gepnerhaugseng2015enriched}.

\begin{cor}
There is a functor of $\infty$-categories
 $$\widehat{\THH}\colon \biHHinfty(\Mod_\Sp) \to \Ho\Sp$$
 such that that the corresponding shadow on $\Mod_\Sp$ (cf. \cref{prop:shadow on infty bicat}) sends any spectral bimodule $\M$ to $\THH(\C,\M)$.
\end{cor}

With this result at hand, we can now formulate a precise version of our conjecture and possible obstructions to proving it.

\begin{conj} \label{conj:main}
 Let $\V$ be a presentably symmetric monoidal $\infty$-category. Let $\Mod_\V$ be the $(\infty,2)$-category with as objects $\V$-enriched categories and morphisms given by bimodules (cf. \cite{haugseng2016bimodules}).

 The following diagram of $\infty$-categories
 \[
    \begin{tikzcd}
    & \V \arrow[d, "\Ho"] \\
        \biHH_\infty(\Mod_V) \arrow[ur, dashed] \arrow[r, "\widehat{\HH}_\V"'] & \Ho(\V)
    \end{tikzcd}
 \]
 admits a lift. Moreover, if $\V$ is closed symmetric monoidal, then the lift is a $\V$-enriched functor of $\V$-enriched $\infty$-categories.
\end{conj}

\begin{rmk} \label{rmk:challenges}
To establish \cref{conj:main}, we need to prove that if $\V$ is closed monoidal (i.e., if $\V =\Sp$), then the functor $\HH_\V\colon \biHHinfty(\Mod_\V) \to \V$, the existence of which we also must establish, is a $\V$-enriched functor of $\V$-enriched $\infty$-categories. While it is clear that  $\V$ is $\V$-enriched \cite[Corollary 7.4.10]{gepnerhaugseng2015enriched},  proving that $\biHHinfty(\Mod_\V)$ is $\V$-enriched requires showing that $\Mod_\V$ is a $\V$-enriched $(\infty,2)$-category i.e., enriched in $\V$-enriched $\infty$-categories rather than just $\infty$-categories. 
 
If $\V$ is presentable, then this is expected to be true and should be established in future work\footnote{Based on private conversation with Rune Haugseng.}. The existing literature, such as \cite{haugseng2016bimodules}, does not include such results, however, and makes the further study of $\HH_\V$ more challenging.
\end{rmk}

\appendix

\section{Truncated Simplicial Pseudo-Diagrams} \label{sec:pseudo diag}
In this appendix we perform a detailed analysis of pseudofunctors out of the truncated simplex categories $(\DD_{\leq 2})^{op}$ and $((\DD_{\leq 2})^{op})^{\rhd}$, which is essential for our proof of \cref{thm:main theorem}. The content of this appendix is a manifestation of categorical obstruction theory, in that we prove in \cref{thm:trunc simpl lift} that the existence of pseudo-functorial lifts for the inclusion $(\DD_{\leq 2})^{op} \to ((\DD_{\leq 2})^{op})^{\rhd}$ reduces to several manageable obstructions.

A $2$-category is by definition a category enriched over categories. Hence, there is a default notion of morphisms between $2$-categories, called \emph{$2$-functors}, which are strict functors that respect the categorical enrichment. On the other hand, we can also consider {\it pseudofunctors}, for which the composition and unit hold only up to chosen natural isomorphisms. For a more detailed discussion of pseudofunctors see \cite{grothendieck1971pseudo,benabou1967bicat}. 

We want to characterize pseudofunctors out of $(\DD_{\leq 2})^{op}$ and out of $((\DD_{\leq 2})^{op})^{\rhd}$ into a $2$-category $\B$. Constructing a pseudofunctor by hand can be quite challenging, however,  since we need to specify a natural isomorphism for every pair of composable morphisms in the domain. Fortunately, there is a way to simplify the task. In \cite{lack2002quillentwocat, lack2004quillenbicat} Lack constructs a model structure on the category of $2$-categories that has the property that any pseudofunctor between $2$-categories  with cofibrant domain is equivalent to a $2$-functor \cite[Remark 4.10]{lack2002quillentwocat}. 

We start therefore by recalling the characterization of the cofibrant objects in Lack's model structure, which requires the notion of  \emph{free categories}. Recall that a directed graph is specified by a pair of functions with the same domain and the same codomain. Every small category has an \emph{underlying directed graph}, given by the source and target functions on the set of morphisms, i.e., there is a forgetful functor $U\colon\cat \to \grph$, the category of directed graphs. The forgetful functor admits a left adjoint $\F_1\colon\grph \to \cat$, which takes a directed graph to its \emph{free category}. For more details regarding free categories and directed graphs see \cite[Section II.7]{maclane1998categories}. The following result regarding cofibrant objects summarizes \cite[Theorem 4.8]{lack2002quillentwocat}.

\begin{thm}\label{thm:lack cofibrant}
 A $2$-category  is cofibrant in Lack's model structure if and only its underlying $1$-category is a free category on a directed graph.
\end{thm}

Unfortunately, the categories $(\DD_{\leq 2})^{op}$ and $((\DD_{\leq 2})^{op})^{\rhd}$ are not cofibrant in Lack's model structure on $2$-categories. Indeed the generating morphisms satisfy non-trivial relations (cf. \cref{lemma:trunc simplex gen}). Hence our first task is to find appropriate cofibrant replacements, for which we apply the theory of {\it computads} introduced in \cite{street1976twolimits}. A computad consists of a directed graph $\G$ together with a set of $2$-arrows between parallel morphisms in the free category on $\G$. We refer to the original source \cite[Section 2]{street1976twolimits} for a more complete description, as we need only certain computads.

Before we proceed, it is helpful to describe fully the category $(\DD_{\leq 2})^{op}$, in particular its morphisms. As it is a full subcategory of $\DD^{op}$, we can state the following result in terms of the characterization of $\DD$ via generators and relations found in \cite[Section I.1]{goerssjardine2009simplicial} and \cite[Section VII.5]{maclane1998categories}
\begin{lemma} \label{lemma:trunc simplex gen}
 The category $(\DD_{\leq 2})^{op}$ can be described as follows.
  \begin{itemize}
     \item It has three objects: $0,1,2$.
     \item It has eight generating morphisms.
     \begin{itemize}
         \item $s_0\colon 0 \to 1$
         \item $d_0,d_1\colon 1 \to 0$
         \item $s_0,s_1\colon 1 \to 2$
         \item $d_0,d_1,d_2\colon 2 \to 1$
     \end{itemize}
     \item The generating morphisms satisfy the following relations.
     \begin{itemize}
        \item $d_0d_0= d_0d_1,d_0d_2=d_1d_0, d_1d_1 =d_1d_2$
         \item $d_1s_0=\mathrm{id} =d_0s_0$
         \item $d_0s_0=d_1s_1 =\mathrm{id}= d_1s_0=d_2s_1$ 
         \item $s_0s_0=s_1s_0$
     \end{itemize}
 \end{itemize}
\end{lemma}

Motivated by this description, we now construct the desired computad, starting with its underlying directed graph.

\begin{defn} \label{def:cdtwo}
 Let $\cD_2$ be the directed graph with three objects $\{0,1,2\}$ and eight arrows $\{s_0\colon0 \to 1; d_0,d_1\colon 1 \to 0; s_0,s_1\colon 1 \to 2; d_0,d_1,d_2\colon 2 \to 1\}$. 
\end{defn}

Next we add the relevant $2$-arrows based on the relations in \cref{lemma:trunc simplex gen}.

\begin{defn} \label{def:cg computad}
 Let $\cG_2$ be the computad with underlying directed graph $\cD_2$, equipped with following  $2$-arrows in $\F_1(\cD_2)$.
 \begin{itemize}
     \item $d_0d_0\Rightarrow d_0d_1,d_0d_2\Rightarrow d_1d_0, d_1d_1 \Rightarrow d_1d_2$
         \item $d_1s_0\Rightarrow\mathrm{id} \Leftarrow d_0s_0$
         \item $d_0s_0\Rightarrow d_1s_1 \Rightarrow\mathrm{id}\Leftarrow d_1s_0 \Leftarrow d_2s_1$ 
         \item $s_0s_0\Rightarrow s_1s_0$
 \end{itemize}
\end{defn}

Applying \cite[Theorem 2]{street1976twolimits} to the computad $\cG_2$ gives us following result.

\begin{lemma} \label{lemma:cg computad inverse}
 There is a $2$-category $\F_2(\cG_2)$ satisfying the universal property that every $2$-functor $\F_2(\cG_2) \to \B$ is specified by the following data.
 \begin{itemize}
     \item A functor from the free category $\F_1(\cD_2)$ on $\cD_2$ to the underlying category of $\B$.
     \item A choice of $2$-morphism in $\B$ with the appropriate source and target for every $2$-arrow of the computad $\cG_2$.
 \end{itemize}
\end{lemma}

Intuitively, every $2$-cell in $\F_2(\cG_2)$ has as source and target a collection of composable $1$-cells, and the ``appropriate'' source and target of the image of the $2$-cell in $\B$ is given by the composition of the images of these $1$-cells (which does exist in $\B$, as it is a $2$-category). We refer to \cite[Page 155]{street1976twolimits} for a more detailed explanation of what appropriate means in this context.

Applying the universal property of the free category functor, let 
 $$\pi_1\colon\F_1(\cD_2) \to (\DD_{\leq 2})^{op}$$ 
 denote the functor that  is the identity on objects and generating $1$-morphisms.  Note that by the universal property of $\F_2(\cG_2)$, there is a $2$-functor
 $$\pi_2\colon\F_2(\cG_2) \to (\DD_{\leq 2})^{op}$$
with underlying functor $\pi_1$ and sending every 2-arrow of the computad $\cD_2$ to an identity. 

 Before proceeding further, we analyze the structure of the $2$-category $\F_2(\cG_2)$ and the relations satisfied by its $2$-morphisms. 
 
 \begin{lemma} \label{lemma:cofib rep}
 The $2$-functor $\pi_2\colon\F_2(\cG_2) \to (\DD_{\leq 2})^{op}$ is a cofibrant replacement in Lack's model structure on bicategories. 
\end{lemma} 

\begin{proof}
 By \cref{thm:lack cofibrant}, the $2$-category $\F_2(\cG_2)$ is evidently cofibrant as its underlying category is just $\F_1(\cD_2)$, which is free by definition.
 
 We therefore need only to show that the projection map $$\pi_2\colon\F_2(\cD_2) \to (\DD_{\leq 2})^{op}$$
 is a weak equivalence in Lack's model structure. However, as follows from the proof  of \cite[Proposition 4.2]{lack2002quillentwocat} and the explanation immediately thereafter, this functor is actually a trivial fibration. 
\end{proof} 

We apply this result to establish a useful characterization of  the $2$-category $\F_2(\cG_2)$. To do so, it is helpful to review {\it contractible groupoids}.

\begin{lemma} \label{lemma:contractible groupoid}
 The forgetful functor $\Grpd \to \set$, which takes a groupoid to its set of objects, admits a right adjoint $I\colon \set \to \Grpd$, which takes a set $S$ to the groupoid $I(S)$ with the same set of objects and a unique isomorphism between any two objects. Moreover, a category with object set $S$ is contractible if and only if it is isomorphic to $I(S)$.
\end{lemma}

\begin{proof}
 Only the last sentence requires an argument. If $\C$ is a contractible category with object set $S$, then for any object $s \in S$, the functor $s\colon[0] \to \C$ is an equivalence and thus fully faithful, which implies that $\C(s,s)= \{\mathrm{id}_s\}$. 
 
Let $s,s' \in S$. Since the map $s\colon [0] \to \C$ is an equivalence, it is essentially surjective, whence there exists an isomorphism $s \to s'$ in $\C$, which must be unique. Indeed, the existence of two distinct isomorphisms $f,g\colon s \to s'$ would imply the existence of a non-trivial automorphism $g^{-1}f\colon s \to s$, in contradiction with the conclusion of the previous paragraph.  We conclude that $\C$ is  isomorphic to $I(S)$, as desired.
\end{proof}

 We call $I(S)$ the \emph{contractible groupoid based on the set $S$}, as it is equivalent to the terminal category. We can formulate an alternative characterization of the morphism categories in $\F_2(\cG_2)$ in terms of contractible groupoids.
 
 \begin{lemma}\label{lemma:alt description ftwogtwo}
  The $2$-category $\F_2(\cG_2)$ has three objects $0,1,2$. Moreover, for any  $i,j\in \{0,1,2\}$, the $1$-morphisms from $i$ to $j$ (i.e., the objects in $\F_2(\cG_2)(i,j)$) are the elements of the set $\F_1(\cD_2)(i,j)$, while for two given $1$-morphisms  $f,g\colon i \to j$, there is a {\it unique} $2$-morphism from $f$ to $g$ if and only if $\pi_1(f) = \pi_1(g)$ in $(\DD_{\leq 2})^{op}$.  In particular, $\F_2(\cG_2)(i,j)$ is a groupoid, and there is an isomorphism of groupoids
    $$\F_2(\cG_2)(i,j) \cong \coprod_{f \in (\DD_{\leq 2})^{op}(i,j)} I\big((\pi_1)^{-1}(f)\big).$$
\end{lemma}

\begin{proof}
 The characterization of the underlying category of the $2$-category $\F_2(\cG_2)$ follows from \cref{lemma:cg computad inverse}. For all $i,j \in \{0,1,2\}$,  \cref{lemma:cofib rep} implies that
 $$\F_2(\cG_2)(i,j) \to (\DD_{\leq 2})^{op}(i,j)$$ 
 is an equivalence of categories. Since the right hand side is discrete, $\F_2(\cG_2)(i,j)$ is a disjoint union of contractible categories. The desired isomorphism now follows from the characterization of contractible categories given in \cref{lemma:contractible groupoid}.
\end{proof}

We next use this cofibrant replacement to characterize pseudofunctors out of $(\DD_{\leq 2})^{op}$. To simplify notation, we henceforth denote the cofibrant replacement above by 
$$\pi_Q\colon Q(\DD_{\leq 2})^{op} \to (\DD_{\leq 2})^{op}.$$

The next result is an immediate consequence of \cref{lemma:alt description ftwogtwo} and \cite[Remark 4.10]{lack2002quillentwocat}.

\begin{cor} \label{cor:pseudo vs two}
Let $\B$ be a $2$-category.
 Precomposing with the cofibrant replacement functor $\pi_Q\colon Q(\DD_{\leq 2})^{op} \to (\DD_{\leq 2})^{op}$ induces an equivalence between the category of pseudofunctors $(\DD_{\leq 2})^{op} \to \B$ and that of $2$-functors $Q(\DD_{\leq 2})^{op} \to \B$. In other words, for every pseudofunctor $F\colon (\DD_{\leq 2})^{op} \to \B$, there exists a $2$-functor $\widehat{F}\colon Q(\DD_{\leq 2})^{op} \to \B$ and an equivalence $F \circ \pi_Q \simeq \widehat{F}$.
\end{cor} 

Combining the results above, we can now characterize pseudofunctors out of $(\DD_{\leq 2})^{op}$.

\begin{lemma} \label{lemma:trunc simp diag}
For any $2$-category $\B$, a pseudofunctor $(\DD_{\leq 2})^{op} \to \B$ is determined up to equivalence by a choice of the following data.
 \begin{itemize}
     \item Three objects $B_0,B_1,B_2$
     \item Eight $1$-morphisms 
     \begin{itemize}
         \item $s_0\colon B_0 \to B_1$
         \item $d_0,d_1\colon B_1 \to B_0$
         \item $s_0,s_1\colon B_1 \to B_2$
         \item $d_0,d_1,d_2\colon B_2 \to B_1$
     \end{itemize}
     \item Ten invertible $2$-morphisms
     \begin{itemize}
         \item $d_0d_0 \overset{\alpha_0}{\Rightarrow} d_0d_1,d_0d_2\overset{\alpha_1}{\Rightarrow} d_1d_0, d_1d_1 \overset{\alpha_2}{\Rightarrow} d_1d_2$ in $\B(B_2,B_0)$
         \item $d_1s_0 \overset{\beta_0}{\Rightarrow} \mathrm{id} \overset{\beta_1}{\Leftarrow} d_0s_0$ in $\B(B_0,B_0)$
         \item $d_0s_0 \overset{\gamma_0}{\Rightarrow} d_1s_1 \overset{\gamma_1}{\Rightarrow} \mathrm{id} \overset{\gamma_2}{\Leftarrow} d_1s_0 \overset{\gamma_3}{\Leftarrow} d_2s_1$ in $\B(B_1,B_1)$
         \item $s_0s_0 \overset{\delta_0}{\Rightarrow} s_1s_0$ in $\B(B_0,B_2)$
     \end{itemize}
    
 \end{itemize}
Moreover, any two $2$-morphisms generated by the $2$-morphisms above that have the same domain and same codomain are equal.
\end{lemma}

\begin{proof}
 By \cref{cor:pseudo vs two} a pseudofunctor $(\DD_{\leq 2})^{op} \to \B$ is determined up to equivalence by a $2$-functor $Q(\DD_{\leq 2})^{op} \to \B$. The universal property of $Q(\DD_{\leq 2})^{op}$ as formulated in \cref{lemma:cg computad inverse} implies that a $2$-functor $Q(\DD_{\leq 2})^{op} \to \B$ is specified by a choice of $0,1,2$-morphisms and relations as above.
 \end{proof}
 
Now that we have a useful description of pseudofunctors out of $(\DD_{\leq 2})^{op}$, the next step is to study pseudofunctorial lifts to $((\DD_{\leq 2})^{op})^\rhd$. Recall that the category $((\DD_{\leq 2})^{op})^\rhd$ is  the join of $(\DD_{\leq 2})^{op}$ and a final object (\cref{def:cocone}). To characterize pseudofunctors out of $((\DD_{\leq 2})^{op})^{\rhd}$  directly is challenging, so we prefer instead to  study $2$-functors out of its cofibrant replacement, which requires proving results analogous to \cref{lemma:cg computad inverse} and \cref{lemma:trunc simp diag}. First we describe $((\DD_{\leq 2})^{op})^\rhd$ in terms of generators and relations, analogously to \cref{lemma:trunc simplex gen}.

\begin{lemma} \label{def:deltaleqtwofinal}
 Let $((\DD_{\leq 2})^{op})^\rhd$ be the category constructed by joining a terminal object to $(\DD_{\leq 2})^{op}$. More precisely,  $((\DD_{\leq 2})^{op})^\rhd$ is specified by following generators and relations.
 \begin{itemize}
     \item Objects $0,1,2,f$
     \item Morphisms
      \begin{itemize}
          \item The generating morphisms in $(\DD_{\leq 2})^{op}$: $s_0\colon 0 \to 1$, $d_0,d_1\colon 1 \to 0$, $s_0,s_1\colon 1 \to 2$, $d_0,d_1,d_2\colon 2 \to 1$, 
          \item Three additional morphisms: $t_0\colon 0 \to f$,  $t_1\colon 1 \to f$, $t_2\colon 2 \to f$.
      \end{itemize} 
    \item Relations  
    \begin{itemize}
        \item The relations that already hold in $(\DD_{\leq 2})^{op}$: 
        \begin{itemize} 
        \item 
        $d_0d_0=d_0d_1$, $d_0d_2= d_1d_0$, $d_1d_1= d_1d_2\colon 2 \to 0$,
        \item $d_1s_0 = \mathrm{id} = d_0s_0\colon0 \to 0$,
        \item $d_0s_0 =d_1s_1= \mathrm{id}=d_1s_0=d_2s_1\colon1 \to 1$,
        \item $s_0s_0=s_1s_0\colon0 \to 2$,
        \end{itemize}
        \item The additional relations for the terminal object:
        \begin{itemize}
        \item $t_0=t_1s_0=t_2s_0s_0=t_2s_1s_0\colon0 \to f$,
        \item $t_1=t_0d_0=t_0d_1=t_2s_0=t_2s_1\colon 1 \to f$,
        \item $t_2=t_1d_0=t_1d_1=t_1d_2=t_0d_0d_0=t_0d_0d_1=t_0d_0d_2= t_0d_1d_0=t_0d_1d_2= t_0d_1d_1\colon2 \to f$.
       \end{itemize}
    \end{itemize}
 \end{itemize}
\end{lemma}

There is an obvious fully faithful inclusion $(\DD_{\leq 2})^{op} \to ((\DD_{\leq 2})^{op})^{\rhd}$. The list of relations for $((\DD_{\leq 2})^{op})^{\rhd}$ given above contains many redundancies, which we can reduce as follows.
 
 \begin{lemma} \label{lemma:extra relations}
  The category $((\DD_{\leq 2})^{op})^{\rhd}$ can be specified by the following data.
  \begin{itemize}
      \item The same objects as above.
      \item The generating morphisms in $(\DD_{\leq 2})^{op}$ together with one additional morphism $t_0\colon0 \to f$
      \item The relations that hold in $(\DD_{\leq 2})^{op}$ together with one additional relation $t_0d_0=t_0d_1\colon 1 \to f$.
  \end{itemize}
 \end{lemma} 
 
 \begin{proof}
 We define $t_1=t_0d_0$ and $t_2=t_0d_0d_0$.
 Since the relations listed in this lemma are a subset of the relations  in \cref{def:deltaleqtwofinal}, the necessity is evident.  We prove that these relations suffice by recovering the remaining relations from them, as follows. 
   \begin{enumerate}
     \item $t_2s_0s_0 \overset{s_0s_0=s_1s_0}{=}t_2s_1s_0$.
     \item $t_2\overset{def}{=}t_0d_0d_0\overset{t_1=t_0d_0}=t_1d_0$
     \item $t_2\overset{def}{=} t_0d_0d_0\overset{d_0d_0=d_0d_1}{=}t_0d_0d_1\overset{t_0d_0=t_1}{=}t_1d_1$
     \item $t_2\overset{def}{=}t_0d_0d_0\overset{t_1=t_0d_0}=t_1d_0\overset{t_1=t_0d_1}{=}t_0d_1d_0\overset{d_1d_0=d_0d_2}{=}t_0d_0d_2\overset{t_0d_0=t_1}{=}t_1d_2$
     \item $t_2 \overset{def}{=} t_0d_0d_0\overset{t_0d_0=t_0d_1}{=} t_0d_1d_0\overset{d_1d_0=d_0d_2}{=}t_0d_0d_2\overset{t_0d_0=t_0d_1}{=}t_0d_1d_2\overset{d_1d_2=d_1d_1}{=}t_0d_1d_1$
     \item 
     $t_2s_0\overset{t_2=t_0d_0d_0}{=}t_0d_0d_0s_0\overset{t_1=t_0d_0}{=}t_1d_0s_0 \overset{d_0s_0=\mathrm{id}} = t_1$
     \item $t_2s_1\overset{t_2=t_0d_0d_0}{=} t_0d_0d_0s_1\overset{d_0d_0=d_0d_1}{=}t_0d_0d_1s_1\overset{t_0d_0=t_1}{=}t_1d_1s_1\overset{d_1s_1=\mathrm{id}}{=}t_1$
     \item  $t_2s_0s_0\overset{t_2=t_0d_0d_0}{=}t_0d_0d_0s_0s_0\overset{\mathrm{id} = d_0s_0}{=}t_0d_0s_0  \overset{\mathrm{id} = d_0s_0}{=}t_0$
     \item $t_1s_0\overset{t_1=t_0d_0}{=}t_0d_0s_0 \overset{d_0s_0=\mathrm{id}} = t_0$
 \end{enumerate}
 As we have recovered all the equalities in \cref{def:deltaleqtwofinal}, the result follows.
 \end{proof}

Given this lemma, we can easily modify  the proofs of \cref{lemma:cg computad inverse} and \cref{lemma:trunc simp diag} to establish the following analogous result.

\begin{defn} \label{def:dtwo plus}
 Let $(\cD_2)^{\rhd}$ denote the directed graph with
 \begin{itemize}
     \item vertices $0,1,2,f$, and
     \item the same edges as in $\cD_2$ (\cref{def:cdtwo}), together with one additional edge $0 \to f$.
 \end{itemize}
\end{defn}

\begin{lemma} \label{lemma:cofib rep final}
  The $2$-category $\big((\DD_{\leq 2})^{op}\big)^{\rhd}$ admits a cofibrant replacement given by the $2$-category $Q\Big(\big((\DD_{\leq 2})^{op}\big)^{\rhd}\Big)$, which is the $2$-category determined by the computad with
  \begin{itemize}
      \item underlying graph is $(\cD_2)^\rhd$, and 
      \item the same $2$-arrows as  $\cG_2$ (\cref{def:cg computad}), along with one additional $2$-arrow $\theta\colon t_0d_0 \Rightarrow t_0d_1$.
  \end{itemize}
\end{lemma}

The morphism categories of $Q\Big(\big((\DD_{\leq 2})^{op}\big)^{\rhd}\Big)$ can be characterized as in \cref{lemma:alt description ftwogtwo}, leading in particular to the following observation.

\begin{cor}\label{cor:contractible-join}
 The category $Q\Big(\big((\DD_{\leq 2})^{op}\big)^{\rhd}\Big)(i,f)$ is a contractible groupoid for all  $i\in \{0,1,2\}$.
\end{cor}

\begin{proof}The proof is essentially the same as that of \cref{lemma:alt description ftwogtwo}, given that there are unique maps $i \to f$ in $Q\Big(\big((\DD_{\leq 2})^{op}\big)^{\rhd}\Big)$.
\end{proof}

\begin{rmk}\label{rem:contractible-maps}
\cref{cor:contractible-join} enables us to better understand $Q\Big(\big((\DD_{\leq 2})^{op}\big)^{\rhd}\Big)$. There are four objects $0,1,2,f$, and morphisms are words with letters 
 $$s_0,d_0,d_1,s_0,s_1,d_0,d_1,d_2,t_0.$$
 In particular, the only generating morphism that has codomain $f$ is $t_0\colon0 \to f$, whence every object in the groupoid $Q\Big(\big((\DD_{\leq 2})^{op}\big)^{\rhd}\Big)(i,f)$ is necessarily of the form $t_0 g$ for some morphism $g\colon i \to 0$, where $i \in \{ 0,1,2\}$. Contractibility of $Q\Big(\big((\DD_{\leq 2})^{op}\big)^{\rhd}\Big)(i,f)$ implies that there is precisely one isomorphism between $t_0g_0$ and $t_0g_1$ for any $g_0,g_1\colon i \to 0$. 
\end{rmk} 

The equivalence below of $2$-functors and pseudofunctors  follows from \cref{lemma:cofib rep final} and \cite[Remark 4.10]{lack2002quillentwocat} .

\begin{cor}\label{cor:ps-2}
Let $\B$ be a $2$-category.
 Precomposing with the cofibrant replacement functor $$\pi_Q^\rhd\colon Q\Big(\big((\DD_{\leq 2})^{op}\big)^{\rhd}\Big) \to \big((\DD_{\leq 2})^{op}\big)^{\rhd}$$ induces an equivalence between the category of pseudofunctors $\big((\DD_{\leq 2})^{op}\big)^{\rhd} \to \B$ and that of $2$-functors $Q\Big(\big((\DD_{\leq 2})^{op}\big)^{\rhd}\Big) \to \B$. In other words, for every pseudofunctor $F\colon \big((\DD_{\leq 2})^{op}\big)^{\rhd} \to \B$, there exists a $2$-functor $\widehat{F}\colon Q\Big(\big((\DD_{\leq 2})^{op}\big)^{\rhd}\Big) \to \B$ and an equivalence $F \circ \pi_Q^\rhd \simeq \widehat{F}$.
\end{cor}

We can use this result to characterize pseudofunctors out of $((\DD_{\leq 2})^{op})^{\rhd}$, using the inclusion functor $Q(\DD_{\leq 2})^{op} \to Q\Big(\big((\DD_{\leq 2})^{op}\big)^{\rhd}\Big)$ arising from the inclusion of computads.  In the statement and the proof below, we use the notation of \cref{lemma:trunc simp diag}.

\begin{thm} \label{thm:trunc simpl lift}
 Let $\B$ be a $2$-category, and let $F\colon(\DD_{\leq 2})^{op} \to \B$ be a pseudofunctor. An extension of $F$ to $\big((\DD_{\leq 2})^{op}\big)^{\rhd}$ 
 \begin{center}
     \begin{tikzcd}
      (\DD_{\leq 2})^{op} \arrow[d, hookrightarrow] \arrow[r, "F"] & \B \\
      ((\DD_{\leq 2})^{op})^{\rhd} \arrow[ur, "\widehat{F}"', dashed]
     \end{tikzcd}
 \end{center}
is specified by
 \begin{itemize}
     \item an object $B_f$ in $\B$,
     \item a $1$-morphism $t_0\colon B_0 \to B_f$, and
     \item an invertible $2$-morphism $t_0d_0 \overset{\theta}{\Rightarrow} t_0d_1$
\end{itemize}
 such that the following equalities hold, where $*$ denotes whiskering of 2-cells by 1-cells, and $\circ$ denotes vertical composition of 2-cells. 
     \begin{itemize}
         \item $\theta* s_0= t_0*((\beta_0)^{-1} \circ \beta_1)\colon t_0d_0s_0 \to t_0d_1s_0$
         \item $ (t_0 (\alpha_1)^{-1}) \circ (\theta * d_0) = (\theta^{-1} * d_2) \circ (t_0 * \alpha_2) \circ (\theta * d_1) \circ (t_0 * \alpha_0)\colon t_0d_0d_0 \to t_0d_0d_2$
  \end{itemize}
  Here the $s_i, d_j$ are $1$-morphisms in $\B$ and  $\alpha_i,\beta_j$ are $2$-morphisms in $\B$, characterizing the pseudofunctor $F$, as described in \cref{lemma:trunc simp diag}.
\end{thm}

\begin{proof}
 By \cref{cor:pseudo vs two} and \cref{cor:ps-2}, we can instead start with a $2$-functor $F\colon Q((\DD_{\leq 2})^{op}) \to \B$ and define an extension $\widehat{F}\colon Q((\DD_{\leq 2})^{op})^{\rhd} \to \B$. The nature of the data required to specify the image of $0,1$, and $2$-morphisms under $\widehat{F}$ follows from the characterization of $Q((\DD_{\leq 2})^{op})^{\rhd}$ in \cref{lemma:cofib rep final}.
 
 In order to complete the proof we make explicit the relations that guarantee that the groupoids $Q\Big(\big((\DD_{\leq 2})^{op}\big)^{\rhd}\Big)(i,f)$ are in fact contractible, i.e., that for any two morphisms $g_1,g_2\colon i \to f$, there is a {\it unique} natural isomorphism $g_1 \cong g_2$. Any other natural isomorphism generated by $\theta\colon t_0d_0\cong t_0d_1$ must therefore coincide with the original one. 
 
 We start by analyzing $Q\Big(\big((\DD_{\leq 2})^{op}\big)^{\rhd}\Big)(1,f)$. As explained in \cref{rem:contractible-maps}, every object in this category is of the form $t_0g'$ for some $g'\colon1 \to 0$. By construction, every object $g'$ in $Q\Big(\big((\DD_{\leq 2})^{op}\big)^{\rhd}\Big)(1,0)$ is {\it uniquely } isomorphic to either $d_0\colon1 \to 0$ or $d_1\colon1 \to 0$. For any two 1-cells $g_1,g_2\colon1 \to t$, with $g_1=t_0g_1'$, $g_2=t_0g_2'$, there are thus four possible scenarios.
 
 \begin{enumerate}
     \item There are unique isomorphisms $g_1' \overset{\kappa_1}{\cong} d_0$ and $g_2' \overset{\kappa_2}{\cong} d_0$, whence
     $$g_1= t_0g_1' \overset{t_0\kappa_1}{\cong} t_0d_0 \overset{t_0\kappa_2}{\cong} t_0g_2' = g_2$$
     is the unique isomorphism between $g_1$ and $g_2$.
     \item There are unique isomorphisms $g_1' \overset{\kappa_1}{\cong} d_1$ and $g_2' \overset{\kappa_2}{\cong} d_1$. This case is essentially identical to the previous one.
     \item There are unique isomorphisms $g_1' \overset{\kappa_1}{\cong} d_0$ and $g_2' \overset{\kappa_2}{\cong} d_1$. In this case the unique isomorphism between $g_1$ and $g_2$ is obtained as follows.
     $$g_1= t_0g_1' \overset{t_0\kappa_1}{\cong} t_0d_0 \overset{\theta}{\cong} t_0d_1 \overset{t_0\kappa_2}{\cong} t_0g_2' = g_2$$
     \item There are unique isomorphisms $g_1' \overset{\kappa_1}{\cong} d_1$ and $g_2' \overset{\kappa_2}{\cong} d_0$.  This case is the same as the previous one, up to permutation of the roles of $g_1$ and $g_2$.
 \end{enumerate}
 
Consider now the case of $Q\Big(\big((\DD_{\leq 2})^{op}\big)^{\rhd}\Big)(0,f)$. 
 There are three types of objects in this category.
 \begin{enumerate}
     \item $t_0g$, where $g\colon0 \to 0$ is uniquely isomorphic to the identity in $Q\Big(\big((\DD_{\leq 2})^{op}\big)^{\rhd}\Big)(0,0)$.
     \item $t_0d_0g$, where $g\colon0 \to 1$ is uniquely isomorphic to $s_0$ in $Q\Big(\big((\DD_{\leq 2})^{op}\big)^{\rhd}\Big)(0,1)$.
     \item $t_0d_1g$, where $g\colon0 \to 1$ is uniquely isomorphic to $s_0$ in $Q\Big(\big((\DD_{\leq 2})^{op}\big)^{\rhd}\Big)(0,1)$.
 \end{enumerate}
As in the previous case, there is a unique isomorphism $t_0d_0g \cong t_0d_1g$. We need therefore only  to compare a 1-cell of the form $t_0g_1$, where $g_1 \cong \mathrm{id}_0$, and a 1-cell of the form $t_0d_0g_2$, where $g_2 \cong s_0$, and to find the necessary and sufficient conditions to guarantee that there is a unique isomorphism between $t_0g_1$ and $t_0d_0g_2$. 
 
 By assumption, $t_0g_1 \cong t_0$ and $t_0d_0g_2 \cong t_0d_0s_0$, via unique isomorphisms. There are two potential isomorphisms between $t_0$ and $t_0d_0s_0$: 
 \begin{enumerate}
     \item $t_0d_0s_0 \overset{t_0*\beta_1}{\Rightarrow} t_0$, and
     \item $t_0d_0s_0 \overset{\theta* s_0}{\Rightarrow}t_0d_1s_0 \overset{t_0*\beta_0}{\Rightarrow} t_0$.
 \end{enumerate}
 Hence, there is a unique isomorphism $t_0g_1 \cong t_0d_0g_2$ if and only if $\theta* s_0 = t_0*((\beta_0)^{-1} \circ \beta_1)$.
 
The final case is that of $Q\Big(\big((\DD_{\leq 2})^{op}\big)^{\rhd}\Big)(2,f)$. If $g\colon2 \to f$ is any 1-cell, then $g = t_0g'$, for some $g'\colon2 \to 0$, which is necessarily of one of the following three types.
\begin{enumerate}
    \item $g' \cong d_0d_0$ via unique isomorphism
    \item $g' \cong d_0d_2$ via unique isomorphism
    \item $g' \cong d_1d_2$ via unique isomorphism
\end{enumerate}
It suffices therefore to find the necessary and sufficient conditions such that for a choice of $g_1',g_2'$ among the three types above, there is a unique isomorphism $t_0g_1' \cong t_0g_2'$.  There are three types of pairs of 1-cells to consider.

\begin{enumerate}
    \item In the case where $g_1' \cong g_2'\cong d_0d_0$, observe that the diagram of unique isomorphisms between 1-cells
    \begin{center}
     \begin{tikzcd}
      t_0d_0d_0 \arrow[r, "\theta* d_0", "\cong"'] \arrow[d, "t_0*\alpha_0"', "\cong"] &t_0d_1d_0 \arrow[r, "t_0*(\alpha_1)^{-1}", "\cong"'] & t_0d_0d_2  \\ t_0d_0d_1  \arrow[r, "\theta* d_1"', "\cong"]
      & t_0d_1d_1  \arrow[r, "t_0*\alpha_2"', "\cong"]
      & t_0d_1d_2,  \arrow[u, "\theta^{-1}* d_2"', "\cong"]
     \end{tikzcd}.
    \end{center}  
   commutes by uniqueness of isomorphisms between 1-cells. It follows that there is a unique isomorphism $t_0g_1' \cong t_0g_2'$ if and only if 
   $$(t_0 (\alpha_1)^{-1}) \circ (\theta * d_0) = (\theta^{-1} * d_2) \circ (t_0 * \alpha_2) \circ (\theta * d_1) \circ (t_0 * \alpha_0)
   $$
       
    \item When $g_1' \cong d_0d_0, g_2'\cong d_0d_2$, the same commuting diagram as in the previous case implies that the same necessary and sufficient conditions apply in this case. 
    
    \item When  $g_1'=d_0d_2$, $g_2'=d_1d_2$, the same commuting diagram again leads to the same necessary and sufficient conditions. 
\end{enumerate}

It follows that $Q\Big(\big((\DD_{\leq 2})^{op}\big)^{\rhd}\Big)(i,f)$ is contractible for $i = 0,1,2$ if and only if
\begin{align*}
\theta* s_0&= t_0*((\beta_0)^{-1} \circ \beta_1)\\
(t_0 (\alpha_1)^{-1}) \circ (\theta * d_0)&= (\theta^{-1} * d_2) \circ (t_0 * \alpha_2) \circ (\theta * d_1) \circ (t_0 * \alpha_0),
    \end{align*}
which finishes the proof.
 \end{proof}
 
\begin{rmk}
The results of this section (and in particular \cref{thm:trunc simpl lift}) focused on the case where $\B$ is a $2$-category, as we  apply these results only to the $2$-category of categories (in \cref{thm:main theorem}). However, following \cite[Lemma 9]{lack2004quillenbicat},  an
 analogous result should certainly hold for bicategories and pseudofunctors of bicategories (called {\it homomorphisms of bicategories} in \cite{lack2004quillenbicat}).
\end{rmk}

\section{Shadows vs. Traces: The Proof} \label{sec:the main result proof}
This section is dedicated to the proof of \cref{thm:main theorem}. We commence with a breakdown into several lemmas.	

\begin{proof}[Proof of \cref{thm:main theorem}]
	By \cref{lemma:cocones are traces} the functor 
	$$\Comp\colon \Fun\big(\biHH(\B),\D\big) \to \Cocone(\B,\D)$$
	is a natural equivalence of categories. In order to prove the theorem, we define two functors natural in $\B$ and $\D$ that we think of as performing ``strictification" and ``unstrictification" processes (\cref{rmk:strictification}),
	$$\St\colon \Cocone(\B,\D) \to \Sha(\B,\D)$$
	in \cref{prop:st functor} and
	$$\Un\colon \Sha(\B,\D) \to \Cocone(\B,\D),$$
	in \cref{prop:un functor} and then prove that they are mutually inverse in \cref{prop:st and un inverses}.
\end{proof}

\begin{lemma} \label{lemma:step two}
	Let $C = (C_0,C_1,C_2)$ be a cocone in $\Cocone(\B,\D)$. Then $C_0\colon \coprod_{X_0} \B(X_0,X_0) \to \D$ satisfies the trace-like property of a shadow.   
\end{lemma}

\begin{proof}
The natural isomorphism witnessing the trace-like property of $C_0$ is the composite of the natural isomorphisms witnessing the commutativity of the left hand triangles in \cref{eq:cocone}, 
	$$\theta_C\colon C_0d_0 \xrightarrow{\cong} C_1 \xrightarrow{\cong} C_0d_1.$$ 
Since this assignment is clearly natural in $\B$ and $\D$, what remains is to prove that  $(C_0,\theta_C)$ satisfies the axioms of a shadow, i.e., that the two  diagrams of \cref{defn:shadow} commute.  First we show that \cref{eq:shadow cond one} commutes. 

The fact that \cref{eq:cocone} is a cocone means that the image of the map 
\[
  {\ds C_0\colon \Fun^\Delta\big(\coprod_{X_0,X_1,X_2} \B(X_0,X_1,X_2,X_0), \coprod_{X_0} \B(X_0,X_0)\big) \to \Fun\big(\coprod_{X_0,X_1,X_2} \B(X_0,X_1,X_2,X_0),\D\big)}
\]
must be contractible (here we are using \cref{not:simplicial maps}). Concretely, this means the image under postcomposition with $C_0$ of any two functors $\coprod_{X_0,X_1,X_2} \B(X_0,X_1,X_2,X_0) \to \coprod_{X_0} \B(X_0,X_0)$ induced by the simplicial operators must be naturally isomorphic in a {\it unique} way (as a non-unique natural isomorphism would give us a non-trivial loop). 

\cref{eq:cocone} gives rise to the  cube below, in which each face commutes up to natural isomorphism.

{\small{     \strut
				\hspace{-1in}
					\begin{tikzcd}[row sep=0.3in, column sep=0.3in]
							&[-0.3in] & \ds\coprod_{X_0,X_1,X_2} \B(X_0,X_1,X_2,X_0) & & &[0.3in] &  \ds\coprod_{X_0,X_1} \B(X_0,X_1,X_0)\\
							\\[0.1in]
						\ds\coprod_{X_0,X_1} \B(X_0,X_1,X_0) & & & & \ds\coprod_{X_0}\B(X_0,X_0) & & \\
						\\[-0.2in]
						& & \ds\coprod_{X_0,X_1}\B(X_0,X_1,X_0) & & & & \ds\coprod_{X_0}\B(X_0,X_0) \\
						\\[0.2in]
						\ds\coprod_{X_0}\B(X_0,X_0) & & & & \D & & 
								\arrow[from=3-1, to=5-3, Rightarrow,"a" description, dashed]
								\arrow[from=5-3, to=1-7, Rightarrow,"a" description, dashed, bend left=10]
								\arrow[from=7-1, to=5-7, Rightarrow, "\theta" description, dashed]
								\arrow[from=3-5, to=5-7, Rightarrow, "\theta" description, dashed]
								\arrow[from=1-3, to=1-7, "d_2" description]
								\arrow[from=1-3, to=5-3, "d_1" description, crossing over]
								\arrow[from=1-3, to=3-1, "d_0"' description]
								\arrow[from=1-7, to=3-1, Rightarrow,"a"' description, dashed, crossing over]
								\arrow[from=1-7, to=5-7, "d_1" description]
								\arrow[from=1-7, to=3-5, "d_0" description]
								\arrow[from=3-1, to=3-5, "d_1" description, crossing over]
								\arrow[from=3-1, to=7-1, "d_0" description]
								\arrow[from=5-3, to=5-7, "d_1" description]
								\arrow[from=5-3, to=7-1, "d_0" description]
								\arrow[from=7-1, to=3-5, Rightarrow, "\theta" description, dashed, bend right=8, crossing over]
								\arrow[from=3-5, to=7-5, "C_0" description, crossing over]
								\arrow[from=5-7, to=7-5, "C_0" description]
								\arrow[from=7-1, to=7-5, "C_0" description]
					\end{tikzcd}}}

	There are exactly six paths in this cube  from the top left corner to the bottom right corner, 
	$$\ds \coprod_{X_0,X_1,X_2} \B(X_0,X_1,X_2,X_0) \to \D ,$$
corresponding to the six objects in \cref{eq:shadow cond one}. The natural isomorphisms on the faces of the cube correspond to the morphisms between the objects in \cref{eq:shadow cond one}. As explained above, the diagram of natural isomorphisms must commute since \cref{eq:cocone}  is a cocone.

It remains to prove that \cref{eq:shadow cond two} commutes. By symmetry it suffices to verify commutativity of the left-hand triangle, which we do by an argument similar to that above. Since  \cref{eq:cocone} is a cocone, the image of 
$$C_0\colon\Fun^\Delta\big(\coprod_{X_0}\B(X_0,X_0),\coprod_{X_0}\B(X_0,X_0)\big) \to \Fun\big(\coprod_{X_0}\B(X_0,X_0) ,\D\big)$$
must be contractible, which means that any two functors from  $\B(X_0,X_0)$ to $\D$ induced by the simplicial diagram in the domain functor category must be natural isomorphic in a unique manner. 

Since the morphisms in the left-hand triangle in \cref{eq:shadow cond two} are exactly the natural isomorphisms in the diagram below, we can conclude by the remark above that the triangle commutes as desired. 

	\begin{center}
		\begin{tikzcd}[row sep=0.5in, column sep=0.5in]
			\ds\coprod_{X_0} \B(X_0,X_0)  \arrow[dr, "s_0"] \arrow[ddr, "id"', bend right=30, ""{name=U, below}] \arrow[drr, "id", bend left=20, ""{name=A}] & & \\
			& \ds\coprod_{X_0,X_1} \B(X_0,X_1,X_0) \arrow[r, "d_1"] \arrow[d, "d_0"] 
			& \coprod_{X_0} \B(X_0,X_0) \arrow[d, "C_0"] \\
			& \ds \coprod_{X_0} \B(X_0,X_0) \arrow[r, "C_0"] & \D
			\arrow[from=2-2, to=U, Rightarrow, "r"', shorten=0.1in]
			\arrow[from=2-2, to=A, Rightarrow, "l", shorten=0.2in]
			\arrow[from=3-2, to=2-3, Rightarrow, "\theta", shorten=0.2in]
		\end{tikzcd}
	\end{center}
\end{proof}

\begin{lemma} \label{lemma:step	three}
 If $(\alpha_0,\alpha_1,\alpha_2)\colon (C_0,C_1,C_2) \Rightarrow (C_0',C_1',C_2')$ is a morphism in $\Cocone(\B,\D)$, then $\alpha_0$ is a morphism of shadows.
\end{lemma}

\begin{proof} 
	We proved in \cref{lemma:step two} that the functors $C_0$ and $ C_0'$ can be equipped with natural transformations $\theta$ and $\theta'$ with respect to which they are shadows. By definition 
	$$\alpha_0\colon C_0 \to C_0'$$
	is a natural transformation. We need to check that $\alpha_0$ satisfies the conditions formulated in \cref{def:category of shadows} with respect to $\theta$ and $\theta'$.
	
The compatibility of $\alpha_0$ and $\alpha_1$ with the natural isomorphisms in the cocones $C$ and $C'$ implies that the diagram
	\begin{center}
					\begin{tikzcd}[row sep=0.5in, column sep=0.5in]
						C_0d_0 \arrow[rr, "\theta", bend left = 30] \arrow[d, "\alpha_0"] & C_1 \arrow[d, "\alpha_1"] \arrow[r, "\cong"] \arrow[l, "\cong"'] & C_0d_1 \arrow[d, "\alpha_0"]\\
						C_0'd_0 \arrow[rr, "\theta'"', bend right = 30] & C_1' \arrow[r, "\cong"'] \arrow[l, "\cong"] &  C_0'd_1
					\end{tikzcd}
	\end{center}
        commutes and hence that \cref{eq:shadow cat cond one} commutes. 
        
        Next we see that $\alpha_0$ also commutes with associators, meaning \cref{eq:shadow cat cond two} commutes, which follows from the diagram below, by standard whiskering arguments in bicategories, again using, as in the previous step, that $(C_0,C_1, C_2)$ and $(C'_0,C'_1, C'_2)$ are cocones.
        
        {\small \begin{center}
		\begin{tikzcd}[row sep=0.3in, column sep=0.3in]
			& \ds \coprod_{X_0,X_1,X_0} \B(X_0,X_1,X_0) \arrow[dr, "d_0"] \arrow[dd, Rightarrow, "a" description, shorten =0.2in] & &[-0.45in] \\
			\ds\coprod_{X_0,X_1,X_2} \B(X_0,X_1,X_2,X_0) \arrow[ur, "d_0"] \arrow[dr, "d_1"'] & & 
			\ds\coprod_{X_0} \B(X_0,X_0) & \strut \arrow[rrr, bend left = 40, ""{name=B, below}, "C_0"] \arrow[rrr, bend right = 40, ""{name=U}, "C_0'"'] &[0.4in] & & \D \\
			& \ds\coprod_{X_0,X_1,X_0} \B(X_0,X_1,X_0) \arrow[ur, "d_0"'] & 
			\arrow[from=B, to=U, Rightarrow, "\alpha_0" description]
		\end{tikzcd}
	\end{center}}
        Finally, we also have to establish that $\alpha_0$ commutes with unitors, meaning \cref{eq:shadow cat cond three} commutes, which follows from the following diagram, again relying on whiskering and the property of cocones.
	{\small{\begin{center}
			\begin{tikzcd}[row sep=0.3in, column sep=0.3in]
					& \ds \coprod_{X_0,X_1} \B(X_0,X_1,X_0) \arrow[dr, "d_0"] & &[-0.45in] \\
					\ds \coprod_{X_0} \B(X_0,X_0) \arrow[rr, "\mathrm {Id}" near start, ""{name=C}, ""{name=D, below}] \arrow[ur, "s_0"] \arrow[dr, "s_0"'] & & \ds \coprod_{X_0} \B(X_0,X_0) & \strut \arrow[rrr, bend left = 40, ""{name=B, below}, "C_0"] \arrow[rrr, bend right = 40, ""{name=U}, "C_0'"'] &[0.4in] & & \D \\
					& \ds \coprod_{X_0,X_1} \B(X_0,X_1,X_0) \arrow[ur, "d_1"'] & 
					\arrow[from=1-2, to=C, Rightarrow, "r" description]
					\arrow[from=3-2, to=D, Rightarrow, "l" description]
					\arrow[from=B, to=U, Rightarrow, "\alpha_0" description]
			\end{tikzcd}
	\end{center}}}
\noindent Thus $\alpha_0$ is indeed a morphism of shadows.
\end{proof}

\begin{prop} \label{prop:st functor}
	The assignment $\St\colon \Cocone(\B,\D) \to \Sha(\B,\D)$ taking an object $(C_0,C_1,C_2)$ to $C_0$ and a morphism $(\alpha_0,\alpha_1,\alpha_2)$ to $\alpha_0$ is a functor.
\end{prop}

\begin{proof}
 We proved in \cref{lemma:step two} and \cref{lemma:step three} that the assignment is well-defined. Moreover, preservation of composition and of identities of $\St$ is evident from the construction.
\end{proof}

We now proceed to construct the inverse functor $\Un$. We first define the functor on objects, using \cref{thm:trunc simpl lift}, which allows us to construct cocones out of shadows via the following lemma. 

\begin{lemma} \label{lemma:step four}
If $\big(\shadow{-}, \theta\big)$ is a shadow on $\B $ taking values in $\D$, then the triple $\big(\D, \shadow{-}, \theta\big)$ satisfies the conditions of \cref{thm:trunc simpl lift}, hence lifts uniquely to a cocone $((\DD_{\leq 2})^{op})^{\rhd} \to \cat$.
\end{lemma}

\begin{proof}
 Based on the input, the first condition in \cref{thm:trunc simpl lift} translates to the following diagram 
	\begin{equation} \label{eq:step four eq one}
		\begin{tikzcd}[row sep=0.5in, column sep=0.5in]
			\shadow{d_0s_0(-)} \arrow[r, "\theta"] \arrow[dr, "\shadow{r}"'] & \shadow{d_1s_0(-)} \arrow[d, "\shadow{l}"] \\
			& \shadow{-}
		\end{tikzcd},
	\end{equation}
	which commutes by \cref{eq:shadow cond two}. 
	
	The second condition that the natural isomorphisms in \cref{thm:trunc simpl lift} must satisfy translates to the commutativity of the following diagram.
		\begin{equation} \label{eq:step four eq three}
					\begin{tikzcd}[row sep=0.5in, column sep=0.5in]
					\shadow{d_0d_0(-)} \arrow[r, "\theta"] \arrow[d, "\shadow{a}"'] & \shadow{d_1d_0(-)} \arrow[r, "\shadow{a}"] & \shadow{d_0d_2(-)} \\
					\shadow{d_0d_1(-)} \arrow[r, "\theta"] & \shadow{d_1d_1(-)} \arrow[r, "\shadow{a}"] & \shadow{d_1d_2(-)} \arrow[u, "\theta"'] 
					\end{tikzcd}
	\end{equation}
	
	Since the commutativity of this diagram is exactly \cref{eq:shadow cond one} in the definition of a shadow, we can conclude.
\end{proof}

\begin{defn} \label{def:un objects}
 Let $\big(\shadow{-}, \theta\big)$ be a shadow on $\B $ taking values in $\D$. Define $\Un(\big(\shadow{-}, \theta\big))$ to be the cocone uniquely defined via \cref{lemma:step four}. 
\end{defn}

We now define $\Un$ on morphisms. For this part it is instructive to review morphisms of cocones more explicitly. Informally speaking, a morphism of cocones $\alpha\colon C \to C'$  has the following shape. 
\[  
	\begin{tikzcd}[row sep=1.5in, column sep=1in]
			\ds\coprod_{X_0} \B(X_0,X_0) 
		\arrow[r, shorten >=1ex,shorten <=1ex] \arrow[dr, "C'_0"' very near start, bend right=15, ""{name=A}] \arrow[dr, "C_0" very near start, ""{name=U, below}, bend left=15] 
		&
	\ds\coprod_{X_0,X_1} \B(X_0,X_1,X_0) 
		\arrow[l, shift left=1.2, "d_1"] \arrow[l, shift right=1.2, "d_0"'] 
	\arrow[r, shift right, shorten >=1ex,shorten <=1ex ] \arrow[r, shift left, shorten >=1ex,shorten <=1ex] \arrow[d, "C_1'"' very near start, bend right=17, ""{name=B}] \arrow[d, "C_1" very near start, bend left=17, ""{name=V, below}] 
& 
\ds\coprod_{X_0,X_1,X_2} \B(X_0,X_1,X_2,X_0).
\arrow[l] \arrow[l, shift left=2, "d_2"] \arrow[l, shift right=2, "d_0"'] \arrow[dl, "C_2'"' very near start, bend right=10, ""{name=C}] \arrow[dl, "C_2" very near start, bend left=10, ""{name=W, below}]\\
& \D & 
\arrow[from=U, to=A, Rightarrow, "\alpha_0"', shorten <= 0.1in, shorten >= 0in]
\arrow[from=V, to=B, Rightarrow, "\alpha_1"' near end, shorten <= 0.05in, shorten >= 0.0in]
\arrow[from=W, to=C, Rightarrow, "\alpha_2"', shorten <= 0.1in, shorten >= 0.1in]
\end{tikzcd}
\] 
More precisely, a natural transformation $C \to C'$ is  a pseudofunctor 
			$$\alpha\colon ((\DD_{\leq 2})^{op})^{\rhd} \to \Fun([1],\cat)$$ 
			that fits into the following diagram.
			\begin{equation}\label{eq:lift nat}
							\begin{tikzcd}[column sep=0.6in]
							& & \cat \\
								(\DD_{\leq 2})^{op} \arrow[r] \arrow[rr, bend right = 20, "\mathrm{id}"' near start] &((\DD_{\leq 2})^{op})^{\rhd} \arrow[r, dashed, "\alpha" description] \arrow[ur, "C" near end] \arrow[dr, "C'"' near end]& {\Fun([ 1 ] , \cat )} \arrow[u, "\pi_1"'] \arrow[d, "\pi_2"]  \\
								& & \cat 
							\end{tikzcd}
			\end{equation}
	Here, the arrow labeled $\mathrm{id}$ corresponds the identity natural transformation on the functor $\biN^{cy}$ (see \cref{def:cocone}), which is the common value of the restrictions of $C$ and $C'$ to $(\DD_{\leq 2})^{op}$. Thus a morphism of cones is simply a lift of  $\mathrm{id}\colon (\DD_{\leq 2})^{op} \to \Fun([1],\cat)$ to $((\DD_{\leq 2})^{op})^{\rhd}$ and so we can again apply \cref{thm:trunc simpl lift}.

	Let $\alpha\colon \shadow{-}_1 \to \shadow{-}_2$ be a morphism in $\Sha(\B,\D)$. Then $\mathrm{id}_\D\colon \D \to \D$ is indeed an object in $\Fun([1],\cat)$, $\alpha$ is a morphism in $\Fun([1],\cat)$ (which are given by natural isomorphisms), and based on the definition of morphism of shadows, there is a natural isomorphism $\beta$ making the following diagram commute (\cref{def:category of shadows}).

\begin{equation}\label{eq:step five eq two}
		\begin{tikzcd}
				\shadow{d_0(-)}_1 \arrow[r, "\theta_1"] \arrow[d, "\alpha d_0"] & \shadow{d_1(-)}_1 \arrow[d, "\alpha d_1"] \\
				\shadow{d_0(-)}_2 \arrow[r, "\theta_2"] & \shadow{d_1(-)}_2
		\end{tikzcd}
\end{equation}

We now have the following lemma. 

\begin{lemma} \label{lemma:step five}
	 Let $\alpha\colon \shadow{-}_1 \to \shadow{-}_2$ be a morphism in $\Sha(\B,\D)$ and denote the natural isomorphism making \cref{eq:step five eq two} commute by $\beta$. Then the triple $(\mathrm{id}_\D, \alpha, \beta)$ satisfies the three conditions of \cref{thm:trunc simpl lift}, and thus lifts uniquely to a cocone $((\DD_{\leq 2})^{op})^{\rhd} \to \Fun([1],\cat)$.
\end{lemma}

\begin{proof}
	Based on the input, the first condition in \cref{thm:trunc simpl lift} translates to the following diagram 
\begin{equation} \label{eq:step five eq one}
				\begin{tikzcd}[row sep=0.4in, column sep=0.6in]
						& \shadow{d_0s_0}_1 \arrow[rr, "\theta_1"] \arrow[dl, "\alpha d_0s_0" description] \arrow[ddr, "\shadow{r}_1" near start] & & \shadow{d_1s_0}_1 \arrow[dl, "\alpha d_1s_0" description] \arrow[ddl, "\shadow{l}_1"] \\
						\shadow{d_0s_0}_2 \arrow[rr, "\theta_2", crossing over] \arrow[ddr, "\shadow{r}_2"'] & & \shadow{d_1s_0}_2 \arrow[ddl, "\shadow{l}_2" description, crossing over] & \\
						& & \shadow{-}_1 \arrow[dl, "\alpha" description] & \\
						& \shadow{-}_2 & &
				\end{tikzcd},
\end{equation}
whereas the second condition requires the diagram below to commute.
\begin{equation} \label{eq:step five eq three}
				\begin{tikzcd}[row sep=0.3in, column sep=0.1in]
				\shadow{d_0d_0(-)}_1 \arrow[dr, "\alpha d_0d_0" description] \arrow[rr, "\theta_1"] \arrow[dd, "\shadow{a}_1"'] & & \shadow{d_1d_0(-)}_1 \arrow[dr, "\alpha d_1d_0" description] \arrow[rr, "\shadow{a}_1"] & & \shadow{d_0d_2(-)}_1 \arrow[dr, "\alpha d_0d_2" description] \\
				& \shadow{d_0d_0(-)}_2 \arrow[rr, "\theta_2"] \arrow[dd, "\shadow{a}_2"' near start] & & \shadow{d_1d_0(-)}_2  & & \shadow{d_0d_2(-)}_2\\
				\shadow{d_0d_1(-)}_1 \arrow[rr, "\theta_1" near end] \arrow[dr, "\alpha d_0d_1" description] & &  \shadow{d_1d_1(-)}_1 \arrow[rr, "\shadow{a}_1"] \arrow[dr, "\alpha d_1d_1" description] & & \shadow{d_1d_2(-)}_1 \arrow[uu, "\theta_1"' near start] \arrow[dr, "\alpha d_1d_2" description] 
				\\
				& \shadow{d_0d_1(-)}_2 \arrow[rr, "\theta_2"] & &  \shadow{d_1d_1(-)}_2 \arrow[rr, "\shadow{a}_2"] & & \shadow{d_1d_2(-)}_2 \arrow[uu, "\theta_2"'] 
				\arrow[from=2-4, to=2-6, "\shadow{a}_2" near start, crossing over]
				\end{tikzcd}
\end{equation}
	To establish the commutativity of these two diagrams, we use that the outer squares and triangles in the diagram commute (by \cref{def:category of shadows}), whence the inside diagram commutes as well, since categories are $2$-coskeletal.
\end{proof}

\begin{defn} \label{def:un morphisms}
 Let $\alpha\colon \big(\shadow{-}_1, \theta_1\big) \to \big(\shadow{-}_2, \theta_2\big)$ be a morphism of shadows from $\B$ to $\D$. Define $\Un(\alpha)$ to be the cocone uniquely defined via \cref{lemma:step five}. 
\end{defn}

We now have all the pieces to construct the desired functor.

\begin{prop} \label{prop:un functor}
 The assignment $\Un\colon \Cocone(\B,\D) \to \Sha(\B,\D)$ taking an object $\shadow{-}$ to $\Un(\shadow{-})$ and a morphism $\alpha$ to $\Un(\alpha)$ is a functor.
\end{prop}

\begin{proof}
 We already have a construction on objects and morphisms. The fact that $\Un$ preserves composition follows directly from the fact that given two composable morphisms $\alpha_1, \alpha_2$, $\Un(\alpha_1 \circ \alpha_2)$ is uniquely determined via a lifting condition (as given in \eqref{eq:lift nat}), which is in particular satisfied by $\Un(\alpha_1) \circ \Un(\alpha_2)$. The case for identities is similar.
\end{proof}
	
We now need to prove that the two functors $\St$ and $\Un$ are mutual inverses. 

\begin{lemma} \label{lemma:stun is identity}
 The composition $\St\Un\colon \Sha(\B,\D) \to \Sha(\B,\D)$ is the identity.
\end{lemma}

\begin{proof}
If  $\big(\shadow{-}, \theta\big)$ is a shadow, then 
$$\Un\Big(\shadow{-},\theta \Big) = \Big(\shadow{-},\shadow{d_0(-)}, \shadow{d_0d_0(-)}\Big)$$ 
and so
	$$\St\Un\big(\shadow{-},\theta\big) = \big(\shadow{-},\theta\big)$$
Hence, $\St\Un$ is the identity functor.	
\end{proof}

\begin{lemma} \label{lemma:unst equivalent identity}
		The composition $\St\Un\colon \Sha(\B,\D) \to \Sha(\B,\D)$ is naturally equivalent to the identity.
	\end{lemma}

	\begin{proof}
		Let $C=(C_0,C_1,C_2)$ be an arbitrary cocone. By \cref{cor:ps-2} this cocone (which is by definition a pseudofunctor) is naturally equivalent to a strict cocone, which, by \cref{thm:trunc simpl lift} is of the form $(C_0,C_0d_0,C_0d_0d_0)$. By functoriality $\St$ and $\Un$ preserve natural isomorphisms, so it suffices to prove that $\Un\St$ takes this particular cocone to itself. 
	
 This follows by direct computation. Indeed $\St(C_0,C_0d_0,C_0d_0d_0)$ is a shadow with shadow functor $C_0$ and $\Un$, by its very definition, takes this shadow to the cocone $(C_0,C_0d_0,C_0d_0d_0)$.
 \end{proof}

	\begin{prop} \label{prop:st and un inverses}
		The pair $(\St, \Un)$ forms an equivalence between $\Cocone(\C,\D)$ and $\Sha(\B,\D)$.
	\end{prop}

	\begin{proof}
		By \cref{lemma:stun is identity} and \cref{lemma:unst equivalent identity}, both compositions $\St\Un$ and $\Un\St$ are naturally equivalent to the identity. 
	\end{proof}

\section{Hochschild Homology of \texorpdfstring{$2$}{2}-Categories} \label{sec:thh two cat}
In this section we compute the Hochschild homology of various $2$-categories of interest. We apply these computations to the study of Morita invariance in \cref{subsec:morita invariance traces}, but they can also be of independent interest. The main computational result is \cref{thm:thh two cat}, which gives an explicit presentation of $\biHH(\B)$ for any $2$-category $\B$. The main application is \cref{thm:thh adj}, which gives explicit descriptions of $\biHH(\Mon)$, $\biHH(\CoMon)$ and $\biHH(\Adj)$ in terms of the paracyclic category, where $\Adj$ is the free adjunction $2$-category, $\Mon$ is the free monad $2$-category, and $\CoMon$ is the free comonad $2$-category.

Recall that for any small category $\C$ and cocomplete category $\D$, there is a \emph{tensor (or coend) functor}
$$ - \otimes - \colon\Fun(\C,\D) \times \Fun(\C^{op},\D) \to \D$$ 
that takes a a pair $(F,G)$ to the coequalizer of
\begin{center}
    \begin{tikzcd}
      \ds \coprod_{f\colon c \to c'} F(c) \times G(c') \arrow[r, shift left=0.05in] \arrow[r, shift right=0.05in] & \ds \coprod_c F(c) \times G(c) 
    \end{tikzcd}.
\end{center}

\begin{rmk}\label{rmk:tensor facts}
 The tensor functor defined above satisfies the following properties.
 \begin{itemize}
     \item \cite[Example 4.1.3]{riehl2014categorybook}: If $\D$ has a terminal object and $F\colon \C \to \D$ is the terminal functor, then $F \otimes G = \colim_{\C^{op}} G$.
     \item \cite[Example 4.1.5]{riehl2014categorybook}: If $F = \Hom_\C(c,-)$, then $F \otimes G \cong G(c)$.
     \item The tensor functor preserves colimits. 
 \end{itemize}
 For more details regarding the tensor functor, see \cite[Section 4.1]{riehl2014categorybook}. 
\end{rmk}

Below we provide a more explicit description of pseudo-colimits in terms of the tensor functor, for which the following technical result proves useful.

\begin{lemma} \label{lemma:tensor functor left quillen}
 The tensor functor 
 $$- \otimes - \colon \cat^{\DD} \times \cat^{\DD^{op}} \to \cat$$ 
 is left Quillen, with respect to the canonical model structure on the codomain and the Reedy model structure on each factor of the domain.
\end{lemma}

\begin{rmk} \label{rmk:canonical model structure} 
 The canonical model structure on $\cat$ is Cartesian, and
 \begin{itemize}
     \item the cofibrations are functors that are injective on objects, and
     \item weak equivalences are categorical equivalences.
 \end{itemize}
 See also \cite[Example 11.3.9]{riehl2014categorybook} for more details.
\end{rmk}

\begin{proof}
  Since the canonical model structure on $\cat$ is Cartesian, we can conclude by an argument analogous to that in \cite[18.4.11]{hirschhorn2003modelcats}. See also the discussion in \cite[Section 14.3]{riehl2014categorybook}.
\end{proof}

To use this lemma to describe pseudo-colimits,  we need the following definition.

\begin{defn} \label{defn:i bullet}
 Let $I[n]$ denote the category with $n$ objects and exactly one morphism from any object to any other object (which in particular implies that all morphisms are isomorphisms). The collection of all $I[n]$ underlies a cosimplicial category 
 $$I[\bullet]\colon \DD \to \cat.$$
 \end{defn}

\begin{prop} \label{prop:pseudo-colimit homotopy colimit}
If $F\colon \DD^{op} \to \cat$ is a strict functor, then the pseudo-colimit of $F$ is equivalent to the tensor product $I[\bullet] \otimes F$.
\end{prop}

\begin{proof}
  By \cite{gambino2008homotopylimits} the pseudo-colimit of $F$ coincides with the homotopy colimit of $F$ in the canonical model structure on $\cat$.  By \cref{rmk:tensor facts} the colimit of $F$ is given by $* \otimes F$ and so, by \cref{lemma:tensor functor left quillen}, the homotopy colimit is the left derivative of this functor. 
  
  Observe that $F\colon \DD^{op} \to \cat$ is already Reedy cofibrant. Indeed, the latching object $L_nF$ is a full subcategory of $F_n$ and so by construction $L_nF \to F_n$ is a cofibration (\cref{rmk:canonical model structure}). We need thus only to find a Reedy cofibrant replacement of the terminal diagram. 
  
  We claim that $I[\bullet]\colon \DD \to \cat$ (\cref{defn:i bullet}) is such a cofibrant replacement. The map to $I[\bullet] \to *$ is obviously a level-wise weak equivalence (\cref{rmk:canonical model structure}), so it suffices to show that $I[\bullet]$ is Reedy cofibrant. This follows from direct computation, as the latching object $L_nI$ has $n+1$ objects, and the map $L_nI \to I[n]$ is the identity on objects and hence a cofibration (\cref{rmk:canonical model structure}).
\end{proof}

We next provide a more explicit description of $I[\bullet] \otimes F$. Recall there is an adjunction 

\begin{equation} \label{eq:fundamental cat}
     \begin{tikzcd}[row sep=0.5in, column sep=0.5in]
     \sSet \arrow[r, shift left=1.8, "\tau_1", "\bot"'] & \cat \arrow[l, shift left=1.8, "N"] 
     \end{tikzcd}
  \end{equation}
where the right adjoint is the nerve functor, and the left adjoint is known as the {\it fundamental category}.

\begin{rmk} \label{rmk:fundamental cat relations} 
The fundamental category can be described explicitly as follows.
\begin{itemize}
    \item $\Obj_{\tau_1X} = X_0$.
    \item The set $X_1$ generates the morphisms of $\tau_1X$.
    \item The morphisms of $\tau_1X$ satisfy the relations $d_1\sigma \sim d_0\sigma \circ d_2\sigma$ for all $\sigma \in X_2$ and $s_0x \sim \mathrm{id}_x$ for all $x \in X_0$.
\end{itemize}

See \cite[Section 11]{rezk2017qcats} for more details.
\end{rmk} 

The following simple property of the fundamental category proves useful to us below.

\begin{lemma} \label{lemma:colimit fundamental cat}
For all functors $F\colon J \to \cat$,  $$\colim_J F \cong \tau_1 \colim_{J}N(F).$$
\end{lemma}

In particular, for every strict functor $F\colon \DD^{op} \to \cat$,

$$I[\bullet] \otimes F\cong \tau_1(NI[\bullet] \otimes NF).$$

\begin{proof}
For any category $\D$, there is a chain of isomorphisms 
   \begin{align*}
  \Hom_{\cat}\big(\tau_1\colim_JN(F),\D\big) 
  & \cong \Hom_{\sSet}\big(\colim_JN(F),N\D\big) & \text{adjunction} (\tau_1,N)\\
  & \cong \lim_J\Hom_{\sSet}\big(N(F)(j),N\D\big) & \text{definition of colimit}\\
  & \cong \lim_J\Hom_{\cat}\big(F(j),\D\big) & \text{$N$ fully faithful,}\\ 
  \end{align*}
  which implies that  $\tau_1\colim_JN(F)$ satisfies the universal property of a colimit, and we can conclude.
\end{proof}

Before we proceed, we make the following general observation regarding computations with bisimplicial sets, which will be key further below.

\begin{rmk} \label{rmk:bisimplicial tensor}
  Let $\Delta[\bullet]\colon \DD \to \sSet$ be the Yoneda embedding. For every bisimplicial set $X$,  
  \[ 
   \Delta[\bullet]_n \otimes X_{\bullet k} \cong \Delta[n] \otimes X_{\bullet k} \cong X_{nk},    
  \]
   where the last step follows from \cref{rmk:tensor facts}, as $\Delta[n]$ is a representable functor.
  In particular, 
  $$(\Delta[\bullet] \otimes X)_k \cong X_{kk},$$ 
 i.e., $\Delta[\bullet] \otimes X$ is the diagonal of the bisimplicial set $X$.
\end{rmk}

By definition $\biHH(\B)$ is a pseudo-colimit, so we can use our newly gained understanding of pseudo-colimits of categories to give a more precise characterization of this construction via generators and relations. 

For the next proof, recall that if $\B$ is a $2$-category, then for every object $X$, there is a unit object $\id_X$ in $\B(X,X)$. Additionally, for every $1$-morphism $F$ in $\B(X,Y)$, there is an identity $2$-morphism $\rid_F$. Moreover,  the composition of two $1$-morphisms or $2$-morphisms in $\B(X,Y)$ and $\B(Y,Z)$ is denoted via $\ast$, to distinguish it from the composition internal to the  categories $\B(X,Y)$. 

\begin{rmk} \label{rmk:explicit}
 For the next theorem, we need a detailed understanding of the bisimplicial set $(N\biN^{cy}\B)_{\bullet\bullet}$ for a fixed 2-category $\B$, so we present here an explicit diagram for the benefit of the reader. 
 
 For a category $\C$, let $\Comp_\C$ denote the set of composable morphisms $(f,g)$, i.e.,  such that $\Dom_g=\Cod_f$. Now, using \cref{rmk:product notation}, we can depict the bisimplicial set $(N\biN^{cy}\B)_{\bullet\bullet}$ as follows, with certain morphisms described explicitly. 
 
\strut \hspace{-1in}
 \begin{tikzcd}[row sep=0.5in, column sep=0.5in]
   \ds\coprod_{X \in \Obj_\B} \Obj_{\B(X,X)} 
   \arrow[d, "\rid_F" description, shorten >=1ex,shorten <=1ex] \arrow[r, "{(F, \id)}" description]
   &  \ds\coprod_{X,Y \in \Obj_\B} \Obj_{\B(X,Y) \times \B(Y,X)} 
   \arrow[d, shorten >=1ex,shorten <=1ex, "\rid_{(F,G)}" description]
   \arrow[l, shift left=2, "G \ast F"] \arrow[l, shift right=2, "F \ast G"'] 
   \arrow[r, shift right=2.4] \arrow[r, shift left=2.4] 
   & \ds\coprod_{X,Y,Z \in \Obj_\B} \Obj_{\B(X,Y) \times \B(Y,Z) \times \B(Z,X)} 
   \arrow[d, shorten >=1ex,shorten <=1ex]
   \arrow[l, "{(F, G \ast H)}" description] \arrow[l, shift right=3.5, "{(F\ast G , H)}"'] \arrow[l, shift left=3.5, "{(H \ast F, G)}"]
   \arrow[r, shorten >=1ex,shorten <=1ex] \arrow[r, shift left=2, shorten >=1ex,shorten <=1ex] \arrow[r, shift right=2, shorten >=1ex,shorten <=1ex]
   & \cdots 
   \arrow[l, shift right=1] \arrow[l, shift left=1] \arrow[l, shift right=3] \arrow[l, shift left=3] 
   \\
    \ds\coprod_{X \in \Obj_\B} \Mor_{\B(X,X)} 
   \arrow[d, shift right=2.1, shorten >=1ex,shorten <=1ex ] \arrow[d, shift left=2.1, shorten >=1ex,shorten <=1ex]
   \arrow[u, shift left=3, "\Dom_\alpha"] \arrow[u, shift right=3, "\Cod_\alpha"'] \arrow[r, "{(\alpha, \rid_{\id})}" description]
   &  \ds\coprod_{X,Y \in \Obj_\B} \Mor_{\B(X,Y) \times \B(Y,X)} 
   \arrow[d, shift right=2.6, shorten >=1ex,shorten <=1ex ] \arrow[d, shift left=2.6, shorten >=1ex,shorten <=1ex ]
   \arrow[u, shift left=4.9, "{(\Dom_\alpha, \Dom_\beta)}"] \arrow[u, shift right=4.9, "{(\Cod_\alpha, \Cod_\beta)}"']
   \arrow[l, shift left=2, "\beta \ast \alpha"] \arrow[l, shift right=2, "\alpha \ast \beta"'] 
   \arrow[r, shift right=2.4] \arrow[r, shift left=2.4]
   &  \ds\coprod_{X,Y,Z \in \Obj_\B} \Mor_{\B(X,Y) \times \B(Y,Z) \times \B(Z,X)} 
   \arrow[d, shift right=2.6, shorten >=1ex,shorten <=1ex ] \arrow[d, shift left=2.6, shorten >=1ex,shorten <=1ex]
   \arrow[u, shift left=1.2] \arrow[u, shift right=1.2]
   \arrow[l, "{(\alpha, \beta \ast \gamma)}" description] \arrow[l, shift right=3.5, "{(\alpha\ast \beta , \gamma)}"'] \arrow[l, shift left=3.5, "{(\gamma \ast \alpha, \beta)}"]
   \arrow[r, shorten >=1ex,shorten <=1ex] \arrow[r, shift left=2, shorten >=1ex,shorten <=1ex] \arrow[r, shift right=2, shorten >=1ex,shorten <=1ex]
   & \cdots 
   \arrow[l, shift right=1] \arrow[l, shift left=1] \arrow[l, shift right=3] \arrow[l, shift left=3] 
   \\[0.5in]
    \ds\coprod_{X \in \Obj_\B} \Comp_{\B(X,X)}  
   \arrow[d, shorten >=1ex,shorten <=1ex] \arrow[d, shift left=2, shorten >=1ex,shorten <=1ex] \arrow[d, shift right=2, shorten >=1ex,shorten <=1ex]
   \arrow[u, "\alpha_2 \circ \alpha_1" description] \arrow[u, shift left=4.2, "\alpha_1"] \arrow[u, shift right=4.2, "\alpha_2"'] 
   \arrow[r, shorten >=1ex,shorten <=1ex]
   &  \ds\coprod_{X,Y\in \Obj_\B} \Comp_{\B(X,Y) \times \B(Y,X)}
   \arrow[d, shorten >=1ex,shorten <=1ex] \arrow[d, shift left=2, shorten >=1ex,shorten <=1ex] \arrow[d, shift right=2, shorten >=1ex,shorten <=1ex]
   \arrow[u]
    \arrow[u, "{(\alpha_2 \circ \alpha_1 ,}" description, start anchor={[yshift=0.2in]}, shorten >=0.5in,shorten <=0.5in, dash]
    \arrow[u, "\beta_2\circ \beta_1)" description, start anchor={[yshift=-0.1in]}, shorten >=0.5in,shorten <=0.5in, dash] 
    \arrow[u, shift left=5.6, "{(\alpha_1 , \beta_1)}"]
    \arrow[u, shift right=5.6, "{(\alpha_2 , \beta_2)}"'] 
   \arrow[l, shift left=1.2, "{(\beta_1\ast \alpha_1 , \beta_2\ast\alpha_2)}"] \arrow[l, shift right=1.2, "{(\alpha_1\ast \beta_1 , \alpha_2\ast\beta_2)}"'] 
   \arrow[r, shift right, shorten >=1ex,shorten <=1ex ] \arrow[r, shift left, shorten >=1ex,shorten <=1ex] 
   & \ds\coprod_{X,Y,Z \in \Obj_\B} \Comp_{\B(X,Y) \times \B(Y,Z) \times \B(Z,X)} 
   \arrow[d, shorten >=1ex,shorten <=1ex] \arrow[d, shift left=2, shorten >=1ex,shorten <=1ex] \arrow[d, shift right=2, shorten >=1ex,shorten <=1ex]
   \arrow[u]
    \arrow[u, "{(\alpha_2 \circ \alpha_1 ,}" description, start anchor={[yshift=0.3in]}, shorten >=0.5in,shorten <=0.5in, dash]
    \arrow[u, "{\beta_2\circ \beta_1,}" description, start anchor={[yshift=0in]}, shorten >=0.5in,shorten <=0.5in, dash]
    \arrow[u, "\gamma_2\circ \gamma_1)" description, start anchor={[yshift=-0.3in]}, shorten >=0.8in,shorten <=0.5in, dash]
    \arrow[u, shift left=5.6, "{(\alpha_1 , \beta_1 , \gamma_1)}"]
    \arrow[u, shift right=5.6, "{(\alpha_2 , \beta_2 , \gamma_2)}"'] 
   \arrow[l] \arrow[l, shift left=2] \arrow[l, shift right=2] 
   \arrow[r, shorten >=1ex,shorten <=1ex] \arrow[r, shift left=2, shorten >=1ex,shorten <=1ex] \arrow[r, shift right=2, shorten >=1ex,shorten <=1ex]
   & \cdots 
   \arrow[l, shift right=1] \arrow[l, shift left=1] \arrow[l, shift right=3] \arrow[l, shift left=3] 
   \\
   \ \vdots \ 
   \arrow[u, shift right=1] \arrow[u, shift left=1] \arrow[u, shift right=3] \arrow[u, shift left=3]
   & \ \vdots \ 
   \arrow[u, shift right=1] \arrow[u, shift left=1] \arrow[u, shift right=3] \arrow[u, shift left=3]
   & \ \vdots \ 
   \arrow[u, shift right=1] \arrow[u, shift left=1] \arrow[u, shift right=3] \arrow[u, shift left=3]
 \end{tikzcd}
\end{rmk}

\begin{thm} \label{thm:thh two cat}
If $\B$ is a $2$-category, then $\biHH(\B)$ admits the following presentation.
 \begin{itemize}
     \item $\Obj_{\biHH(\B)} = \coprod_{X \in \Obj_\B} \Obj_{\B(X,X)}$
     \medskip
    
     \item Generating morphisms
     \begin{enumerate}
         \item Symbols $\alpha\colon \Dom_\alpha \to \Cod_\alpha$ for all $\alpha \in \coprod_{X \in \Obj_\B} \Mor_{\B(X,X)}$.
         \item Symbols $(F,G)\colon F*G \to G*F$ for all $F \in  \Obj_{\B(X,Y)}$, $G \in \Obj_{\B(Y,X)}$ and all $X,Y \in \Obj_\B$.
     \end{enumerate}
     \medskip
     
     \item Relations
     \begin{enumerate}
         \item $(F,\id_X)\colon F \to F$ is the identity morphism of the object $F \in \Obj_{\B(X,X)}$.
         \item The composite of a pair of symbols $\alpha, \beta\in \Mor_{\B(X,X)}$ such that $\Dom_\beta = \Cod_\alpha$ is equal to their composite $\beta\circ \alpha$ in the category $\B(X,X)$.
         \item All symbols $(F,G)$ are isomorphisms.
         \item For all  symbols $\alpha\colon F \to G$ and $\beta\colon H \to L$ in $\Mor_{\B(X,X)}$, 
         $$(G,L) \circ (\alpha\ast\beta) = (\beta\ast\alpha) \circ (F,H).$$
         \item For any three $1$-morphisms $F \in \Obj_{\B(X,Y)}$, $G \in \Obj_{\B(Y,Z)}$, $H \in \Obj_{\B(Z,X)}$, 
         $$(F, G \ast H) = (H\ast F,G) \circ (F \ast G,H).$$
     \end{enumerate}
 \end{itemize}
\end{thm}

\begin{proof}
According to \cref{def:biHH}, $\biHH(\B)$ is the colimit of the simplicial diagram $\biN^{cy}_\bullet(\B)$, which we denote henceforth by $\biN^{cy}_\bullet$ to simplify notation within the proof. Because $\B$ is a $2$-category, $\biN^{cy}_\bullet$ is a strict diagram rather than just a pseudo-diagram, whence by \cref{prop:pseudo-colimit homotopy colimit}, $\biHH(\B)$ is equivalent to the tensor $I[\bullet] \otimes \biN^{cy}_\bullet$, which is isomorphic to the fundamental category of $NI[\bullet] \otimes (N\biN^{cy})_{\bullet\bullet}$, by \cref{lemma:colimit fundamental cat}.
 
In order to evaluate the fundamental category we need to better understand levels $0,1,$ and $2$ of the simplicial set $NI[\bullet] \otimes (N\biN^{cy}_{\bullet\bullet})$. For a fixed $k \geq 0$, the evaluation map $(-)_k\colon\sSet \to \set$ preserves colimits, so there is a bijection of sets
$$(N(I[\bullet]) \otimes N\biN^{cy}_{\bullet\bullet})_k \cong NI[\bullet]_k \otimes (N\biN^{cy})_{\bullet k}.$$ 
 
If $k = 0$, then $I[\bullet]_0 = \Delta[\bullet]_0$, the representable functor, and so by the argument above
 $$\Obj_{\biHH(\B)}=(NI[\bullet] \otimes N\biN^{cy}_{\bullet\bullet})_0 = (N\Delta[\bullet] \otimes N\biN^{cy}_{\bullet\bullet})_0= (N\biN^{cy})_{00} = \coprod_X \Obj_{\B(X,X)},$$
as desired.
 
If $k = 1$, then $NI[\bullet]_1 = \Mor_{I[\bullet]} = \Delta[\bullet]_1 \coprod_{\Delta[\bullet]_0} \Delta[\bullet]_1$, where the pushout is that of the degeneracy map $s_0\colon \Delta[\bullet]_0 \to \Delta[\bullet]_1$ with itself, since any two objects of $I[n]$ are connected by a unique isomorphism. Hence, 
\begin{align*}
  (NI[\bullet] \otimes N\biN^{cy}_{\bullet\bullet})_1 & \cong N(I[\bullet])_1 \otimes (N\biN^{cy}_{\bullet 1})\\
  & \cong (\Delta[\bullet]_1 \coprod_{\Delta[\bullet]_0} \Delta[\bullet]_1) \otimes (N\biN^{cy}_{\bullet 1})\\
  & \cong (\Delta[\bullet]_1 \otimes N\biN^{cy}_{\bullet 1}) \coprod_{(\Delta[\bullet]_0 \otimes N\biN^{cy}_{\bullet 1})} (\Delta[\bullet]_1 \otimes N\biN^{cy}_{\bullet 1}) \\
  & \cong (N\biN^{cy})_{11} \coprod_{(N\biN^{cy})_{01}}  (N\biN^{cy})_{11},
\end{align*}
where the last step follows from \cref{rmk:bisimplicial tensor}.
  
Recall from \cref{rmk:explicit} that  $(N\biN^{cy})_{11}= \coprod_{X,Y} \Mor_{\B(X,Y) \times \B(Y,X)}$, i.e., its elements are pairs $(\alpha,\beta)$ with $\alpha\colon F \to G$ and $\beta\colon H \to L$, with face maps made explicit in \cref{rmk:explicit}. It follows that the set   $(NI[\bullet] \otimes N\biN^{cy}_{\bullet\bullet})_1$ consists of pairs $(\alpha,\beta)\colon \Dom_\alpha \ast \Dom_\beta \to \Cod_\beta \ast \Cod_\alpha$ and $(\alpha, \beta^{-1})\colon \Dom_\beta \ast \Dom_\alpha \to \Cod_\alpha \ast \Cod_\beta$, such that for every $\alpha \in \Mor_\B(X,X)$ we have $(\alpha,\id_X) = (\alpha,\id_X^{-1})\colon \Dom_\alpha \to \Cod_\alpha$.
 
 We now use the information in level $2$ to describe the various relations between the morphisms. Notice that 
 $$NI[\bullet]_2 \cong NI[\bullet]_1 \times_{NI[\bullet]_0} NI[\bullet]_1.$$ 
 Hence, the elements in the set $(NI[\bullet] \otimes N\biN^{cy}_{\bullet \bullet})_2$ are of the form 
 \begin{itemize}
     \item $((\alpha_0,\alpha_1),(\beta_0,\beta_1),(\gamma_0,\gamma_1))$,
     \item $((\alpha_0,\alpha_1),(\beta_0^{- 1},\beta_1^{\pm 1}),(\gamma_0^{\pm 1},\gamma_1^{\pm 1}))$,
     \item $((\alpha_0,\alpha_1),(\beta_0^{\pm 1},\beta_1^{\pm 1}),(\gamma_0^{\pm 1},\gamma_1^{\pm 1}))$,
     \item $((\alpha_0,\alpha_1),(\beta_0^{\pm 1},\beta_1^{\pm 1}),(\gamma_0^{\pm 1},\gamma_1^{\pm 1}))$,
 \end{itemize}
  where $\alpha_0\colon F_0 \to F_1,\alpha_1\colon F_1 \to F_2\in\Mor_\B(X,Y)$, $\beta_0\colon G_0 \to G_1,\beta_1\colon G_1 \to G_2\in\Mor_\B(Y,Z)$ and $\gamma_0\colon H_0 \to H_1,\gamma_1\colon H_1 \to H_2\in\Mor_\B(Z,X)$.
 
 We focus first on relations induced by elements of the form $((\alpha_0\alpha_1),(\beta_0,\beta_1),(\gamma_0,\gamma_1))$. Using the morphisms given in \cref{rmk:explicit}, and their compositions, it follows that
  \begin{itemize}
     \item $d_0((\alpha_0,\alpha_1),(\beta_0,\beta_1),(\gamma_0,\gamma_1))=(\gamma_1 \ast \alpha_1,\beta_1)$
     \item $d_1((\alpha_0,\alpha_1),(\beta_0,\beta_1),(\gamma_0,\gamma_1))=(\alpha_1\alpha_0,\beta_1\beta_0 \ast \gamma_1\gamma_0)$
     \item $d_2((\alpha_0,\alpha_1),(\beta_0,\beta_1),(\gamma_0,\gamma_1))=(\alpha_0 \ast \beta_0,\gamma_0)$
 \end{itemize}
 which means that the following relation holds.
 
 \begin{equation} \label{eq:thh two}
 (\alpha_1\alpha_0, \beta_1\beta_0 \ast \gamma_1\gamma_0) =  (\gamma_1 \ast \alpha_1,\beta_1) \circ (\alpha_0 \ast \beta_0,\gamma_0)
  \end{equation}
  
 Fix $\alpha\colon F \to G \in \Mor_\B(X,Y)$ and $\beta\colon H \to L \in \Mor_\B(Y,X)$.
If $(\alpha_0,\beta_0,\gamma_0) = (\alpha,\beta,\id_X)$ and $(\alpha_1,\beta_1,\gamma_1) = (\rid_G,\rid_L,\id_X)$, then $(\alpha,\beta) = (\rid_G,\rid_L) \circ (\alpha \ast \beta,\id_X)$. On the other hand, if $(\alpha_0,\beta_0,\gamma_0) = (\rid_F,\rid_H,\id_X)$ and $(\alpha_1,\beta_1,\gamma_1) = (\alpha,\beta,\id_X)$, then $(\alpha,\beta) = (\beta\ast\alpha,\id_X) \circ (\rid_F,\rid_H)$. Hence,
 \begin{equation} \label{eq:commutative} 
 (\alpha,\beta) = (\rid_G,\rid_L) \circ (\alpha \ast \beta,\id_X) = (\beta\ast\alpha,\id_X) \circ (\rid_F,\rid_H).
 \end{equation}
 
 Now, by symmetry, we can repeat the same arguments for $(\alpha,\beta^{-1})$ to conclude that
 \begin{equation} \label{eq:commutative inverse} 
 (\alpha,\beta^{-1}) = (\alpha\ast \beta,\id_X) \circ (\rid_F,\rid_H^{-1}) = (\rid_G,\rid_L^{-1}) \circ (\beta \ast\alpha,\id_X)
 \end{equation}

 Henceforth, we denote the morphisms of the form $(\alpha, \id_X)$ by  $\alpha$, morphisms of the form $(\rid_F,\rid_G)$ by $(F,G)$, and morphisms of the form $(\rid_F,\rid_G^{-1})$ by $(F,G)^{-1}$. 
 
 In \eqref{eq:commutative} and \eqref{eq:commutative inverse}, we have already established that every arbitrary morphism in $\biHH(\B)$ is generated by these three classes of morphisms. In order to finish the proof we need to understand how these morphisms interact with each other and in particular confirm the relations from the statement.
 
\begin{enumerate}
    \item 
 The first relation follows from the definition of the fundamental category.
 
 \item Let $\alpha_0\colon F \to G, \alpha_1\colon G \to H \in \Mor_\B(X,X)$ be two morphisms. \eqref{eq:thh two}  with $(\alpha_0,\beta_0,\gamma_0) = (\alpha_0,\id_X,\id_X)$ and $(\alpha_1,\beta_1,\gamma_1) = (\alpha_1,\id_X,\id_X)$ implies that 
 $$(\alpha_1\alpha_0,\id_X) = (\alpha_1, \id_X)\circ  (\alpha_0,\id_X),$$
 which proves the second equation.
 
\item  Now for two objects $F\colon X \to Y$ and $G\colon Y \to X$ we have by definition of $I[2]$ a $2$-cell that we denote $(\rid_F,\rid_G,\rid_G^{-1})$ that witnesses the composition $(\rid_F,\rid_G^{-1}) \circ (\rid_F,\rid_G) = (\rid_{FG},\id_X)$. We can similarly deduce that $(\rid_F,\rid_G) \circ (\rid_F,\rid_G^{-1})= (\rid_{GF},\id_Y)$. This proves that $(F,G)$ is in fact an isomorphism with inverse $(F,G)^{-1}$
  
\item  We already confirmed the fourth condition in \eqref{eq:commutative}.
 
\item  Finally, we want to understand when two morphisms of the form $(F,G)$ commute. Plugging in $(\alpha_0,\beta_0,\gamma_0) = (\alpha_1,\beta_1,\gamma_1) = (\rid_F,\rid_G,\rid_H)$ into \eqref{eq:thh two} we get 
$$(\rid_F,\rid_G \ast\rid_H) = (\rid_H \ast \rid_F,\rid_G) \circ (\rid_F \ast \rid_G,\rid_H), $$
which gives us the desired relation $(F, G \ast H) = (H\ast F,G) \circ (F \ast G,H)$.
\end{enumerate}
 As we have checked all possible relations between all generating morphisms, we have a complete characterization of $\biHH(\B)$ and hence are done.
\end{proof} 

In certain cases we can simplify the result in \cref{thm:thh two cat} further. For a given strict monoidal category $\C$, let $B\C$ be the category with on object $*$ and $B\C(*,*) = \C$.

\begin{cor} \label{cor:thh two cat on obj}
 Let $(\C, \otimes, \id)$ be a strict monoidal category. Then $\biHH(B\C)$ admits the following presentation.
 \begin{itemize}
      \item $\Obj_{\biHH(\B)} =  \Obj_{\C}$
     \item Generating Morphisms
     \begin{itemize}
         \item Symbols $\alpha\colon \Dom_\alpha \to \Cod_\alpha$ where $\alpha \in \Mor_{\C}$.
         \item Symbols $(X,Y)\colon X \otimes Y \to Y \otimes X$ where $X, Y \in  \Obj_{\C}$
     \end{itemize}
     \item Subject to the Relations
     \begin{itemize}
         \item $(X,\id)\colon X \to X$ is the identity morphism of the object $X \in \Obj_{\C}$.
         \item For two symbols $\alpha, \beta$ coming from $\Mor_{\C}$ such that $\Dom_\beta = \Cod_\alpha$, the composition is given by the composition $\beta\circ \alpha$ given in $\Mor_{\C}$.
         \item The symbols $(X,Y)$ are isomorphisms.
         \item For symbols $\alpha\colon X \to Y$ and $\beta\colon Z \to W$ in $\Mor_{\C}$, we have
         $(Y,W) \circ \alpha\otimes\beta = \beta\otimes\alpha \circ (X,Z)$.
         \item For three objects $X,Y,Z \in \Obj_{\C}$ we have the equality
         $(X, Y \otimes Z) = (Z\otimes X,Y) \circ (X \otimes Y,Z)$.
    \end{itemize}
 \end{itemize}
\end{cor}

\begin{proof}
 Apply \cref{thm:thh two cat} to the $2$-category $B\C$ with only one object.
\end{proof}

To conclude this appendix, we apply  \cref{thm:thh two cat} and \cref{cor:thh two cat on obj} to the explicit computation of some key examples.

\begin{prop} \label{prop:thh bn}
 $\ds\biHH(B\bN) \simeq \{0\} \coprod_{n \in \{1,2,3,...\}} \{n\} \times B\bZ.$
\end{prop}

\begin{proof}
 We apply \cref{cor:thh two cat on obj} to symmetric monoidal category $\bN$ with only identity morphisms. It has objects $\{0,1,2,...\}$ and no non-trivial morphisms. Moreover, for every $n,m \in \bN$ there is an isomorphism $(n,m)\colon n+m \to n+m$.
 
 Now, the last relation in \cref{cor:thh two cat on obj} implies that for all $m>1$
 $$(n,m)  = (n,(m -1) + 1 ) = (1+n, m - 1) \circ (n+ (m - 1) , 1).$$
 By induction, this implies that $(n,m) = (n + m - 1,1)^m$ and so every object $n>0$ has a unique automorphism $(n-1,1)$. This finishes the proof.
\end{proof}

\begin{prop}\label{prop:thh bnn} 
 $ \ds\biHH(B(\bN \ast \bN)) \simeq \{(0,0)\} \coprod_{(n,m) \in  \bN \times \bN \backslash \{(0,0)\}} \{(n,m)\} \times B\bZ.$
\end{prop}

\begin{proof}
 Again, we use \cref{cor:thh two cat on obj} for the symmetric monoidal category $B(\bN \ast \bN)$. According to the result the objects are isomorphic to the elements in $\bN \ast \bN$. However, for two elements $x,y \in \bN \ast \bN$ we have $xy\cong yx$ in the category $\biHH(B(\bN \ast \bN))$ and so it suffices to take one object from each isomorphism class, which correspond to the commutator classes and are precisely $\bN \prod \bN$. 
 
 Now for a given object $(n,m)$, an automorphism is given by a tuple $((n_1,m_1),(n_2,m_2))$ such that $n = n_1 + n_2$ and $m =m_1 + m_2 = m_2 + m_1$. If $n= 0$ or $m = 0$, then this reduces to \cref{prop:thh bn}, and it follows that there is a unique generating automorphism for all cases and no automorphism for the case $n = m = 0$.

 If $n, m \neq 0$, then the only elements that commute with $(n,m)$ are $((n,m),(0,0))$ and $((0,0),(n,m))$. The first is the identity, by \cref{cor:thh two cat on obj}, and so we again have one free automorphism and so the desired result follows.
\end{proof}

\begin{rmk}
 As $\bN$ is not a group, we could not have used \cref{ex:thh groupoids} to do the computation above. However, the group completion of $\bN$ is $\bZ$, so we can ask ourselves how the previous two results compare to the computation of  $\biHH(B\bZ)$ and  $\biHH(B(\bZ * \bZ))$.
 
 First, the free loop space of $B\bZ$ is equivalent to $\bZ \times B\bZ$, which admits an inclusion 
 \[\{0\} \coprod_{n \in \{1,2,3,...\}} \{n\} \times B\bZ \to \bZ \times B\bZ.\]
 On the other hand the free loop space on $B(\bZ *\bZ)$ is equivalent to the groupoid $\Fun(B\bZ,B(\bZ*\bZ))$, which has objects automorphisms of $B(\bZ*\bZ)$, i.e., $\bZ *\bZ$ and morphisms that are natural transformations, which are given by conjugation. Isomorphism classes of objects are given by conjugacy classes, which correspond to $\bZ \times \bZ$, while the automorphism group is  the centralizer, which for $(0,0)$ is  $\bZ \times \bZ$ and for any other object is  $\bZ$. There is again an inclusion 
 $$\ds\{(0,0)\} \coprod_{(n,m) \in \bN \times \bN \backslash \{(0,0)\}} \{(n,m)\} \times B\bZ \to B(\bZ *\bZ) \coprod_{(n,m) \in \bZ \times \bZ \backslash \{(0,0)\}}B\bZ.$$
 
 Neither inclusions is full, making it challenging to deduce $\biHH(B\bN)$ from $\biHH(B\bZ)$.
\end{rmk}

We now tackle a more complicated example.  Let $\DD_+$ be the category of finite ordinals and order-preserving morphisms, which can also be characterized as the category $\DD$ together with one additional initial object corresponding to the empty ordinal. It is a strictly monoidal category with monoidal structure given by disjoint union and unit given by the empty set.

Before computing $\biHH(B\DD_+)$, we need to review the paracyclic category, as defined in \cite[Example I.22]{dyckerhoffkapranov2015crossed}

 \begin{defn} \label{defn:paracyclic}
  Let $\Lambda_\infty$ be the paracyclic category, with objects $\{ 0,1,2, ... \}$ and morphisms 
  $n \to m$ given by linear functions $f\colon \bZ \to \bZ$ such that $f(l+m+1)= f(l) + n + 1$.
 \end{defn}
 
 The paracyclic category was introduced in \cite{lodayzbigniew1991paracyclicorigin}, but the name comes from \cite{getzlerjones1993paracyclicname}. We need the following concrete characterization of the category as described in \cite[Example I.28]{dyckerhoffkapranov2015crossed}.
 
 \begin{rmk} \label{rmk:paracyclic relations}
  The paracyclic category can be characterized via the following generators and relations.
  \begin{itemize}
      \item Objects $\{[0],[1],[2],...\}$
      \item Generating morphisms:
      \begin{itemize}
          \item For $n \geq 1$ and $0 \leq i \leq n$, $d^i\colon [n-1] \to [n]$
          \item For $n \geq 0$ and $0 \leq i \leq n$, $s^i\colon [n+1] \to [n]$
          \item For $n \geq 0$, $t^n\colon [n] \to [n]$
      \end{itemize}
      \item $s^i$ and $d^i$ satisfy the cosimplicial relations and additionally we have for $n \geq 1$
      \begin{itemize}
          \item $t^nd^i=d^{i-1}t^{n-1}$ where $i>0$
          \item $t^nd^0 = d^n$
          \item $t^ns^i= s^{i-1}t^{n+1}$, where $i>0$
          \item $t^ns^0 = s^n(t^{n+1})^2$
      \end{itemize}
  \end{itemize}
 \end{rmk}

 \begin{thm} \label{thm:thh paracyclic}
  $\biHH(B\DD_+) \simeq (\Lambda_\infty)^\lhd.$
 \end{thm}
 
 \begin{proof}
 It suffices to prove that $\biHH(B\DD_+)$ has the same presentation as $(\Lambda_\infty)^\lhd$ described in \cref{rmk:paracyclic relations}. By \cref{cor:thh two cat on obj}, the set of objects of $\biHH(B\DD_+)$ is $\bN$. The monoidal structure is the same as for $B\bN$, thus following \cref{prop:thh bn}, there is a unique isomorphism $t^n\colon n \to n$ for every $n>0$. 
 
 By \cref{cor:thh two cat on obj} the morphisms are generated by the $d^i,s^i$, and $t^i$, where the interaction of $s^i$ and $d^i$ is given by the cosimplicial relations. What remains is to check how the $s^i$ and $d^i$ interact with the $t^i$.

 We start with $s^0\colon[n+1] \to [n]$. We have 
 $$t^ns^0 = ([n-1],[0]) \circ s^0 = s^0 ([n-1],[1]) = s^0([n],[0])\circ ([n],[0]) = s^0(t^{n+1})^2.$$
If $s^i\colon [n+1] \to [n]$, where $i>0$, then $s^i = \rid_{[0]} \coprod s^{i-1}\colon [0] \coprod [n-1] \to [0] \coprod [n]$, whence 
 \[t^ns^i = ([n-1],[0]) \circ (\rid_{[0]} \coprod s^{i-1}) = (\rid_{[0]} \coprod s^{i-1} ([n],[0]) = s^{i-1}t^{n+1}. \]
 Consider now $d^i$. If $i_0\colon \emptyset \to [0]$ is the unique map, then $d^0= i_0 \coprod \rid_{[n-1]}\colon [n-1] \to [n]$, whence
 $$t^nd^0 = ([n-1],[0]) \circ i_0 \coprod \rid_{[n-1]} = \rid_{[n-1]} \coprod i_0 ([n],\emptyset) = d^n.$$
 For $i>0$, $d^i= \rid_{[0]} \coprod d^{i-1}\colon [0] \coprod [n-1] \to [0] \coprod [n]$, and so 
 $$t^nd^i= ([n-1],[0])\rid_{[0]} \coprod d^{i-1} = d^{i-1} \coprod \rid_{[0]} ([n-2],[0]) = d^{i-1}t^{n-1}.$$
 
 Finally we need to confirm that $\emptyset$ is still the initial object.  Since
 \[t^0i_0=(\emptyset,[0]) i_0=i_0(\emptyset,\emptyset)=i_0,\]
it follows that $\emptyset$ is still initial. This confirms all the relations, and hence we are done.
 \end{proof}
  
  We can now use this result to establish the main result of interest, which requires us to review some concepts regarding the free adjunction category $\Adj$. 
  
  \begin{rmk} \label{rmk:adj}
  The {\it free adjunction bicategory} $\Adj$ (\cref{def:adj}) can be described as follows. It has two objects $0,1$, and the following morphisms.
  \begin{itemize}
      \item $\Adj(0,0)= \DD_+$, where we think of $[n]$ as the morphism $(GF)^n$.
      \item $\Adj(1,1) = (\DD_+)^{op}$, the elements of which we denote by $[n]^{op}$ to distinguish them from the previous item and think of as $(FG)^n$. 
      \item $\Adj(0,1) = \DD_{max}$, the wide subcategory of $\DD$ consisting of morphisms in $\DD$ that preserve the maximum, the elements of which we denote by $[n]_{max}$ and think of as $(FG)^nF$
      \item $\Adj(0,1) = \DD_{min}$, the wide subcategory of $\DD$ consisting of morphisms in $\DD$ that preserve the minimum, with elements denoted $[n]_{min}$ and think of as $(GF)^nG$.
  \end{itemize}
 \end{rmk}
 For more details, see the original description of the free adjunction in \cite{schanuelstreet1986freeadj}.
  
  The full subcategory of $\Adj$ with unique object $0$ is the free monad  $2$-category, denoted by $\Mon$, while the full subcategory of $\Adj$ with unique object $1$ is the free comonad $2$-category, denoted $\CoMon$. We now compute $\biHH$ of each of these 2-categories.
  
 \begin{thm} \label{thm:thh adj}
  There is a diagram of equivalences of categories
  \begin{center}
      \begin{tikzcd}
        \biHH(\Mon) \arrow[r, hookrightarrow] & \biHH(\Adj) \arrow[r, hookleftarrow] & \biHH(\CoMon)  \\
        (\Lambda_\infty)^{\lhd} \arrow[r, hookrightarrow] \arrow[u, "\simeq"] & (\Lambda_\infty)^{\lhd\rhd} \arrow[u, "\simeq"]  \arrow[r, hookleftarrow] & (\Lambda_\infty^{op})^{\rhd} \arrow[u, "\simeq"]  
      \end{tikzcd}
  \end{center}
 \end{thm}

 \begin{proof}
  The two equivalences on the left and right side follow easily from \cref{thm:thh paracyclic}. It remains to prove that the middle map is an equivalence and that the squares commute.
  
  Following \cref{thm:thh two cat} and the explicit description of the $2$-category $\Adj$ given above in \cref{rmk:adj}, the category $\biHH(\Adj)$ admits the following description.
  \begin{itemize}
      \item Objects: two copies of $\bN$, which we denote $\emptyset,[0],[1], \cdots$ and $\emptyset^{op},[0]^{op},[1]^{op},\cdots$.
      \item Morphisms: $\sigma\colon [m] \to [n]$, $\sigma^{op}\colon[n]^{op} \to [m]^{op}$ for all $\sigma \in \DD([m],[n])$. 
      \item Isomorphisms: $([n],[m])\colon [n+m+1] \to [n+m+1]$, $([n]^{op},[m]^{op})\colon [n+m+1]^{op} \to [n+m+1]^{op}$ for objects $n,m \in \bN$.
      \item Isomorphisms: 
      $([n]_{min},[m]_{max})\colon[n+m+1] \to [n+m+1]^{op}$ and 
      $([m]_{max},[n]_{min})\colon[n+m+1]^{op} \to [n+m+1]$, where $n,m \in \bN$.
      \item Relations between the morphisms described in \cref{thm:thh two cat}.
  \end{itemize}
  We now use the various relations to reduce the structure and obtain the desired result.
  
  First, if we restrict to the objects $[0], [1], \cdots$, then we have exactly the generating morphisms and relations of $\DD_+$ and thus by \cref{thm:thh paracyclic}, obtain a copy of $(\Lambda_\infty)^{\lhd}$. Similarly, the generating morphisms and isomorphisms restricted to $[0]^{op},[1]^{op},\cdots$ gives rise to a copy of $(\Lambda_\infty^{op})^{\rhd}$. What remains is to explain how the additional isomorphisms $([n]_{min},[m]_{max})$ and     $([m]_{max},[n]_{min})$ influence $\biHH(\Adj)$.
  
  If $n,m \in \bN$, then the last relation in \cref{thm:thh two cat} implies that 
  \begin{align*}
    ([n]_{min},[m]_{max}) & =  ([n]_{min},[m-1]_{max}[0]^{op}) \\
    & = ([0]^{op}[n]_{min},[m-1]_{max}) \circ ([n+m],[0]) \\
    & = ([n+1]_{min},[m-1]_{max}) \circ ([n+m],[0])  
  \end{align*}
  and so, by induction,
  $$([n]_{min},[m]_{max}) = ([n+m]_{min},[0]_{max}) \circ ([n+m],[0])^m$$
  and similarly 
  $$([n]_{max},[m]_{min}) = ([n+m],[0])^m \circ ([n+m]_{max},[0]_{min}).$$
  All isomorphisms are therefore expressed as composition of the isomorphisms $([n]_{min},[0]_{max})\colon [n+1] \to [n+1]^{op}$ and $([n]_{max},[0]_{min})\colon[n+1]^{op} \to [n+1]$, for all $n \in \bN$. 
  
  Finally, again by the last relation in \cref{thm:thh two cat},  
  \[([n],[0])^2 = ([n-1],[1]) = ([n]_{max},[0]_{min}) \circ ([n]_{min},[0]_{max}).\]
There is therefore a simplified explicit description of $\biHH(\Adj)$, formulated as follows.
  \begin{itemize}
      \item Objects: $[0],[1],...$ and $[0]^{op},[1]^{op},...$
      \item Generating morphisms:
      \begin{itemize}
          \item $d^i\colon[n] \to [n+1],s^i\colon[n] \to [n-1],t^n\colon[n] \to [n]$
          \item $(d^i)^{op}\colon[n]^{op} \to [n+1]^{op},(s^i)^{op}\colon[n]^{op} \to [n-1]^{op}, (t^n)^{op}\colon [n]^{op} \to [n]^{op}$
          \item $c^n\colon[n] \to [n]^{op}$
      \end{itemize}
      \item Subject to relations:
      \begin{itemize}
          \item $d^i, s^i,t^n$ satisfy the paracyclic relations (\cref{rmk:paracyclic relations})
          \item $(d^i)^{op},(s^i)^{op},(t^n)^{op}$ satisfy the opposite paracyclic relations 
          \item $c^nt^n = (t^n)^{op}c^n$ as 
           $$([n]_{min},[0]_{max}) \circ ([n],[0]) =  ([n]_{min},[1]_{max})=([n]^{op},[0]^{op})\circ ([n]_{min},[0]_{max})$$
          \item $c^nd^i = (d^i)^{op}c^n$, which follows directly from the fourth relation in \cref{thm:thh two cat},
          \item $c^ns^i = (d^i)^{op}c^n$.
      \end{itemize}
  \end{itemize}
  
  The description of $\biHH(\Adj)$ as two copies of $\Lambda_{\infty}$ and $(\Lambda_\infty)^{op}$ and an additional isomorphism $c_n$ is just an explicit description of the following pseudo-pushout:
  \begin{center}
      \begin{tikzcd}
        \Lambda_\infty \arrow[r] \arrow[d] & (\Lambda_\infty)^{\lhd} \arrow[d] \\
        ((\Lambda_\infty)^{op})^{\rhd} \arrow[r] & \biHH(\Adj)
      \end{tikzcd}
  \end{center}
  which is equivalent to $(\Lambda_\infty)^{\lhd\rhd}$. By construction the squares in the diagram in the statement of the theorem commute.
 \end{proof}
 
 \begin{rmk}
In recent work Ayala and Francis have (independently) computed $\biHH(\Adj)$ \cite{ayalafrancis2021tracesthh} using methods from factorization homology \cite{ayalafrancis2017thhfactorization}. 
\end{rmk}

We now make one last computation, also necessary for our work on Morita invariance. For this last theorem we rely on a number of computations throughout this section. 

\begin{thm} \label{thm:thh adjend}
 $\ds\biHH(\AdjEnd) \simeq (\Lambda_\infty)^{\lhd\rhd} \times \{0\} \coprod_{n \in \{1,2,3,...\}} B\bZ \times \DD_+ \times \{n\}$
\end{thm}

\begin{proof}
 Let us start by reducing the set of objects, so that no two objects are isomorphic.
 Following \cref{rmk:morphism in adjbn} and \cref{thm:thh two cat}, the set of objects in $\biHH(\AdjEnd)$ admits a bijection with 
 \begin{align*}
     \Obj_{\AdjEnd(0,0)}  \coprod \Obj_{\AdjEnd(1,1)}& \cong (\bN \ast \bN) \coprod (\bN \ast \bN) \\
     & = (\{\emptyset, [0],...\} \ast \{0,1,...\}) \coprod (\{\emptyset^{op},[0]^{op},...\} \ast \{0,1,...\}),
 \end{align*}
 where we are using \cref{rmk:adj}.
  
If $x,y \in \{\emptyset, [0],...\} \ast \{0,1,...\}$, then \cref{thm:thh two cat} implies that there is an isomorphism $(x,y)\colon xy \to yx$, whence the isomorphism classes are the commutators classes of the free words (as already explained in \cref{prop:thh bnn}). Isomorphism classes of objects of $\biHH(\AdjEnd)$ can thus in our first step be reduced to 
 \begin{align*}
     \bN \times \bN \coprod \bN \times \bN & =  \{\emptyset, [0],...\} \times \{0,1,...\} \coprod \{\emptyset^{op},[0]^{op},...\} \times \{0,1,...\} \\
     & \cong \{[n]k,[n]^{op}k\colon [n] \in \Obj_{\DD},k \in \bN\}
 \end{align*}
 
 We make one further reduction of the set of objects. 
By the same argument as in \cref{thm:thh adj} (and again using the notation in \cref{rmk:adj}), it follows that 
 $$([n-1]_{min},[0]_{max})k\colon[n]k \to [n]^{op}k$$
 is an isomorphism. Hence, we can conclude that the isomorphism classes of objects are precisely given by 
 $$\bN \times \bN \cong \{[n]k\colon [n] \in \Obj_{\DD_+}, k \in \bN\}.$$  
 
Now that we have determined the objects, we can move on to determining the generating morphisms and isomorphisms, again using \cref{thm:thh two cat}.

By the explanation in \cref{rmk:morphism in adjbn}, morphisms in $\AdjEnd(0,0)$ are given by words of morphisms in $\DD_+$, so there is a morphism from $[n_1]m_1$ to $[n_2]m_2$ if and only if $m_1 = m_2$, whence 
\begin{equation}\label{eq:mor}
  \Mor_{\AdjEnd(0,0)} = \{\sigma k\colon[n]k \to [m]k\colon \sigma \in \DD_+([n],[m]), k \in \bN\}  
\end{equation}
 
On the other hand, the monoidal structure on $\AdjEnd(0,0)$ coincides with that on $B(\bN\ast\bN)$, whence, as explained in \cref{prop:thh bnn}, every object $[n]k$ has a unique generating non-trivial automorphism, i.e., $\Aut_{\biHH(\AdjEnd)}([n]k) = \bZ$.  

In order to finish the proof, we need to understand the interaction between the automorphisms and morphisms, using the relations given in \cref{thm:thh two cat}. Let $\biHH(\AdjEnd)_k$ be the full subcategory of $\biHH(\AdjEnd)$ consisting of objects of the form $[n]k$. The explanation in \eqref{eq:mor} implies that 
$$\biHH(\AdjEnd) \cong \coprod_{k \in \bN} \biHH(\AdjEnd)_k.$$
We can therefore break our analysis down into the different $\biHH(\AdjEnd)_k$.

Let us start with $k=0$. In that case there are morphisms $\sigma 0\colon [n]0 \to [m]0$, for $\sigma\colon [n] \to [m]$ in $\DD_+$, and automorphisms $t^n\colon [n]0 \to [n]0$, which, by the relations given in \cref{thm:thh adj}, interact precisely as stated in  \cref{thm:thh adj}, whence 
$$\biHH(\AdjEnd)_0 \simeq (\Lambda_\infty)^{\lhd\rhd}.$$

If $k >0$, then the generating isomorphism of $[n]k$ is given by the symbol $(\emptyset0,[n]k)$, which, by \cref{thm:thh two cat}, satisfies the equality
$$(\emptyset0,[n]k) = ([n]0,\emptyset k) \circ (\emptyset k,[n]0),$$
where $([n]0,\emptyset k)\colon [n]k \to k[n]$ is a twisting isomorphism, and $(\emptyset k,[n]0)$ is defined similarly.
For an arbitrary morphism $\sigma k\colon [n]k \to [m]k$ in $\biHH(\AdjEnd)_k$, where $\sigma\colon[n] \to [m]$,  
\begin{align*}
  (\emptyset0,[m]k)\circ \sigma k & =
 ([m]0,\emptyset k) \circ (\emptyset k,[m]0) \circ \sigma \\
 & =  ([m]0,\emptyset k) \circ k \sigma \circ (\emptyset k,[n]0) \\
 & =
 \sigma k \circ ([n]0,\emptyset k) \circ (\emptyset k,[n]0) \\
 & =
 \sigma k \circ (\emptyset0,[n]k).  
\end{align*}
 The third equality above is a consequence of the fact that $m \neq 0$. It follows that $(\emptyset0,[n]k)$ commutes with all morphisms, proving the desired equivalence 
 $$\biHH(\AdjEnd)_m \simeq B\bZ \times \DD_+,$$
 and finishing the proof.
\end{proof}

\bibliographystyle{alpha}
\bibliography{main.bib}

\end{document}